\documentclass[11pt,letterpaper]{amsart}
\usepackage{etex}
\usepackage[utf8]{inputenc}
\usepackage[english]{babel}
\usepackage{amssymb,amsmath,amsthm}
\usepackage{graphicx,caption}
\usepackage{float}
\usepackage{enumerate}
\usepackage{tikz}
\usetikzlibrary{patterns}
\usepackage{pgfplots}
\usepackage[all]{xy}
\usepackage{tikz-cd}
\usetikzlibrary{matrix,arrows,backgrounds}
\usepackage{comment}
\usepackage{cite}
\usepackage{faktor}
\usepackage{setspace}
\usepackage{amssymb,amstext}
\usepackage{appendix}
\usepackage{xcolor}
\pgfplotsset{compat=1.10}

\makeatletter
\@namedef{subjclassname@2020}{\textup{2020} Mathematics Subject Classification}
\makeatother

\numberwithin{equation}{section}

\newcommand{\extp}{\@ifnextchar^\@extp{\@extp^{\,}}}
\def\extp^#1{\mathop{\textstyle\bigwedge\nolimits^{\!#1}}}
\makeatother

\theoremstyle{plain}
\newtheorem{teo}{Theorem}[section]
\newtheorem{prop}{Proposition}[section]
\newtheorem{cor}{Corollary}[section]
\newtheorem{lema}{Lemma}[section]
\newtheorem{conj}{Conjecture}[section]

\theoremstyle{definition}
\newtheorem{defi}{Definition}[section]
\newtheorem{rem}{Remark}[section]

\DeclareMathOperator\coker{coker}
\DeclareMathOperator\res{res}

\newcommand{\tree}[4]{
\begin{tikzpicture}[baseline=-3pt, scale=.6]

\draw[thick] (-1,0) -- (1,0);
\draw[thick] (-1,0.3) -- (-1, -0.3);
\draw[thick] (1,0.3) -- (1, -0.3);

\node[above] at (-1,0.3) {$#1$};
\node[below] at (-1,-0.3) {$#2$};
\node[above] at (1,0.3) {$#4$};
\node[below] at (1,-0.3) {$#3$};

\end{tikzpicture}
}

\begin{document}

\title[Invariants of $\mathbb{Q}$-Homology Spheres]{Invariants of $\mathbb{Z}/p$-Homology 3-Spheres from the Abelianization of the Level-p Mapping Class Group}
\author{Wolfgang Pitsch}
\author{Ricard Riba}
\address{Universitat Autònoma de Barcelona, Departament de Matemàtiques, Bellaterra, Spain}
\email{pitsch@mat.uab.es}
\address{Universitat Autònoma de Barcelona, Departament de Matemàtiques, Bellaterra, Spain}
\email{riba@mat.uab.cat}
\thanks{This work was  supported by MEC grants MTM2016-80439-P and  PID2020-116481GB-I00
and the AGUR grant 2021-SGR-01015}



\date{\today}

\begin{abstract}
We study the relation between the set of oriented $\mathbb{Z}/d$-homology $3$-spheres and the level-$d$ mapping class groups, the kernels of the canonical maps from the mapping class group of an oriented surface to the symplectic group with coefficients in $\mathbb{Z}/d\mathbb{Z}$.  We formulate a criterion to decide whenever a $\mathbb{Z}/d$-homology $3$-sphere can be constructed from a Heegaard splitting with gluing map an element of the level-$d$ mapping class group.
Then we give a tool to construct invariants of $\mathbb{Z}/d$-homology $3$-spheres from families of trivial $2$-cocycles on the level-$d$ mapping class groups.
We apply this tool to find all the invariants of $\mathbb{Z}/p$-homology $3$-spheres constructed from families of $2$-cocycles on the abelianization of the level-$p$ mapping class group with $p$ prime and to disprove the conjectured extension of the Casson invariant modulo a prime $p$ to rational homology $3$-spheres due B. Perron.
\end{abstract}

\maketitle

\section{Introduction}

Let $\Sigma_g$ be an oriented surface of genus $g$ standardly embedded in the oriented $3$-sphere $\mathbf{S}^3$. Denote by $\Sigma_{g,1}$ the complement of the interior of a small disc embedded in $\Sigma_g$. 
The surface $\Sigma_g$ separates $\mathbf{S}^3$ into two genus $g$ handlebodies $\mathbf{S}^3=\mathcal{H}_g\cup -\mathcal{H}_g$ with opposite induced orientation.
Denote by $\mathcal{M}_{g,1}$ the mapping class group of $\Sigma_{g,1},$ i.e. the group of orientation-preserving diffeomorphism of $\Sigma_g$ which are the identity on our fixed disc modulo isotopies which again fix that small disc point-wise. The embedding of $\Sigma_g$ in $\mathbf{S}^3$ determines three natural subgroups of $\mathcal{M}_{g,1}$: the subgroup $\mathcal{B}_{g,1}$ of mapping classes that extend to the inner handlebody
$\mathcal{H}_g$, the subgroup $\mathcal{A}_{g,1}$ of mapping classes that extend to the outer handlebody $-\mathcal{H}_g$ and their intersection $\mathcal{AB}_{g,1}$, which are the mapping classes that extend to $\mathbf{S}^3$. Denote by $\mathcal{V}$ the set of diffeomorphism classes of  closed  oriented smooth $3$-manifolds.
By the theory of Heegaard splittings we know that any element in $\mathcal{V}$ can be obtained by cutting $\mathbf{S}^3$ along $\Sigma_g$ for some $g$ and gluing back the two handlebodies by some element of the mapping class group $\mathcal{M}_{g,1}.$
The lack of injectivity of this construction is controlled by the handlebody subgroups $\mathcal{A}_{g,1}$ and $\mathcal{B}_{g,1}$; J.~Singer~\cite{singer} proved that there is a bijection:
\begin{align*}
\lim_{g\to \infty}\mathcal{A}_{g,1}\backslash\mathcal{M}_{g,1}/\mathcal{B}_{g,1} & \longrightarrow \mathcal{V}, \\
\phi & \longmapsto S^3_\phi=\mathcal{H}_g \cup_{\iota_g\phi} -\mathcal{H}_g.
\end{align*}

In 1989 S. Morita, using Singer's bijection, proved the analogous result for the set of integral homology $3$-spheres $\mathcal{S}_{\mathbb{Z}}^3$ and the Torelli group $\mathcal{T}_{g,1}$, which is the group of those elements of $\mathcal{M}_{g,1}$ that act trivially on  the first homology group of $\Sigma_{g,1}$. The above map $\phi \mapsto S^3_\phi$ induces a bijection:
\[
\lim_{g\to \infty}\mathcal{A}_{g,1}\backslash\mathcal{T}_{g,1}/\mathcal{B}_{g,1} \longrightarrow \mathcal{S}_{\mathbb{Z}}^3.
\]

Let $\mathcal{M}_{g,1}[d]$ denote the level-$d$ mapping class group, that is the kernel of the map: $\mathcal{M}_{g,1} \rightarrow Sp_{2g}(\mathbb{Z}/d\mathbb{Z})$. Restricting the map $\phi \mapsto S^3_\phi$ to these subgroups gives us a subset $\mathcal{S}^3[d] \subset \mathcal{V}$. If we denote by $\mathcal{S}^3(d)$ the set of mod-$d$ homology spheres, then the Mayer-Vietoris exact sequence shows that for any $d \geq 2$, $ \mathcal{S}^3_\mathbb{Z} \subseteq \mathcal{S}^3[d] \subseteq \mathcal{S}^3(d)$. Moreover, if we denote by $\mathcal{S}^3_\mathbb{Q}$ the subset of rational homology spheres, we have that
\[
\bigcup_{p \text{ prime}} \mathcal{S}^3(p) = \mathcal{S}_\mathbb{Q}^3.
\]
 
Our first result describes the difference between $\mathcal{S}^3[d]$ and $\mathcal{S}^3(d)$. 

\begin{teo}
\label{teo-criteri-intro}
Let $M$ be a rational homology $3$-sphere and $n$ be the cardinal of $H_1(M;\mathbb{Z})$. For any integer $d \geq 2$, the manifold $M$  belongs to $\mathcal{S}^3[d]$  if and only if $d$ divides either $n-1$ or $n+1$.
\end{teo}

This result  implies that unlike as for integral homology $3$-spheres ($d=0$) and the Torelli group, in general, $\mathcal{S}^3[d]$ does not coincide with the set $\mathcal{S}^3(d)$. More precisely:

\begin{cor}
The sets $\mathcal{S}^3[d]$ and $\mathcal{S}^3(d)$ only coincide for $d=2,3,4,6.$
\end{cor}

However, by Theorem \ref{teo-criteri-intro}, letting $d$ vary along all primes we have that
\[
\bigcup_{p \text{ prime}} \mathcal{S}^3[p] = \mathcal{S}^3_\mathbb{Q}.
\]

Building on Singer's bijection  we produce then a bijection:
$$
	{\displaystyle
\lim_{g\to \infty}\mathcal{A}_{g,1}\backslash\mathcal{M}_{g,1}[d]/\mathcal{B}_{g,1}} \longrightarrow \mathcal{S}^3[d].
$$
As in the case of integral homology spheres (cf. \cite[Lemma 4]{pitsch}), if we set $\mathcal{A}_{g,1}[d]=\mathcal{M}_{g,1}[d]\cap \mathcal{A}_{g,1}$ and
$\mathcal{B}_{g,1}[d]=\mathcal{M}_{g,1}[d]\cap \mathcal{B}_{g,1}$, we can rewrite this equivalence relation as follows:
\begin{prop}
There is a bijection:
$$
\lim_{g\to \infty}(\mathcal{A}_{g,1}[d]\backslash\mathcal{M}_{g,1}[d]/\mathcal{B}_{g,1}[d])_{\mathcal{AB}_{g,1}} \simeq \mathcal{S}^3[d].
$$
\end{prop}

Using this we can rewrite any normalized invariant $F$ of rational homology $3$-spheres with values in an arbitrary abelian group $A$, that is an invariant that is zero on the $3$-sphere $\mathbf{S}^3$, as a family of functions $(F_g)_{g\geq 4}$ from the level-$d$ mapping class groups $\mathcal{M}_{g,1}[d]$ to the abelian group $A$ satisfying the following properties:

\begin{enumerate}[i)]
\item[0)] Normalization:
\[F_g(Id)=0 \quad \text{where $Id$ denotes the identity of $\mathcal{M}_{g,1}[d]$}.
\]
\item Stability:
\[F_{g+1}(x)=F_g(x) \quad \text{for every }x\in \mathcal{M}_{g,1}[d].
\]
\item  Conjugation invariance:
\[
F_g(\phi x \phi^{-1})=F_g(x)  \quad \text{for every }  x\in \mathcal{M}_{g,1}[d], \; \phi\in \mathcal{AB}_{g,1}.
\]
\item Double class invariance:
 \[
 F_g(\xi_a x\xi_b)=F_g(x) \quad \text{for every } x\in \mathcal{M}_{g,1}[d],\;\xi_a\in \mathcal{A}_{g,1}[d],\;\xi_b\in \mathcal{B}_{g,1}[d].
 \]
\end{enumerate}

If we consider the associated trivial $2$-cocycles $(C_g)_{g\geq 4},$ which measure the failure of the maps $(F_g)_{g\geq 4}$ to be group homomorphisms, i.e. the maps:
\begin{align*}
C_g: \mathcal{M}_{g,1}[d]\times \mathcal{M}_{g,1}[d] & \longrightarrow A \\
 (\phi,\psi) & \longmapsto F_g(\phi)+F_g(\psi)-F_g(\phi\psi),
\end{align*}
then these  inherit the following properties:

\begin{enumerate}[(1)]
\item The $2$-cocycles $(C_g)_{g\geq 4}$ are compatible with the stabilization map.
\item The $2$-cocycles $(C_g)_{g\geq 4}$ are invariant under conjugation by elements in $\mathcal{AB}_{g,1}$.
\item If $\phi\in \mathcal{A}_{g,1}[d]$ or $\psi \in \mathcal{B}_{g,1}[d]$ then $C_g(\phi, \psi)=0$.
\end{enumerate}

Notice that there is not a one-to-one correspondence between the families of functions $(F_g)_{g\geq 4}$ and the families of $2$-cocycles $(C_g)_{g\geq 4}$.
In general there is more than one invariant associated to a $2$-cocycle, for there are homomorphisms on the level-$d$ mapping class groups that are invariants. This is akin to the Rohlin invariant for integral homology spheres, which is a $\mathbb{Z}/2\mathbb{Z}$-valued invariant and a homomorphism when viewed as a function on the Torelli groups (cf.\cite{riba1}).

\begin{prop}\label{propintro:invmorphism}
Given integers $g\geq 4$ and $d\geq 2$ such that $4\nmid d$, up to a multiplicative constant, there is a unique $\mathcal{AB}_{g,1}$-invariant $\mathbb{Z}/d$-valued homomorphism $\varphi_g$ on the level-$d$ mapping class group $\mathcal{M}_{g,1}[d]$.
%

Moreover the map $ \displaystyle \varphi = \lim_{g \to \infty} \varphi_g$ is an invariant  of rational homology $3$-spheres in $\mathcal{S}^3[d]$ for $d\neq 2$ and for $d=2$ this map is not an invariant.
\end{prop}

Our main tool to build new invariants is given by the following result, which generalizes the main result from \cite{pitsch}:

\begin{teo}\label{teo_tool-intro}
Given an integer $d\geq 3$ such that $4\nmid d$ and
$x\in A$ a $d$-torsion element, a family of $2$-cocycles $C_g: \mathcal{M}_{g,1}[d]\times \mathcal{M}_{g,1}[d] \rightarrow A$ for $g \geq 4$ satisfying conditions (1)-(3) provides compatible families of trivializations $F_g+x\varphi_g: \mathcal{M}_{g,1}[d]\rightarrow A$ that reassemble into invariants of rational homology spheres in $\mathcal{S}^3[d]$,
$$\lim_{g\to \infty}F_g+x\varphi_g: \mathcal{S}^3[d]\longrightarrow A$$
if and only if the following two conditions hold:
\begin{enumerate}[(i)]
\item The associated cohomology classes $[C_g]\in H^2(\mathcal{M}_{g,1}[d];A)$ are trivial.
\item The associated torsors $\rho(C_g)\in H^1(\mathcal{AB}_{g,1},Hom(\mathcal{M}_{g,1}[d],A))$ are trivial.
\end{enumerate}
\end{teo}

For $d=2$, because of the peculiarity of this case in Proposition~\ref{propintro:invmorphism}, the situation is slightly simpler:

\begin{teo}\label{teo_tool-introd2}
A family of $2$-cocycles $C_g: \mathcal{M}_{g,1}[2]\times \mathcal{M}_{g,1}[2] \rightarrow A$  for $g \geq 4$ satisfying conditions (1)-(3) provides a unique compatible family of trivializations $F_g: \mathcal{M}_{g,1}[2]\rightarrow A$, that reassembles into an invariant of rational homology spheres in $\mathcal{S}^3[2]$,
	$$\lim_{g\to \infty}F_g: \mathcal{S}^3[2]\longrightarrow A$$
	if and only if the following two conditions hold:
	\begin{enumerate}[(i)]
		\item The associated cohomology classes $[C_g]\in H^2(\mathcal{M}_{g,1}[2];A)$ are trivial.
		\item The associated torsors $\rho(C_g)\in H^1(\mathcal{AB}_{g,1},Hom(\mathcal{M}_{g,1}[2],A))$ are trivial.
	\end{enumerate}
\end{teo}

These two theorems provide us with a bridge between the topological problem that is to find invariants of mod-$d$ homology $3$-spheres and the  purely algebraic problem that is to find families of $2$-cocycles on the level-$d$ mapping class group satisfying the conditions of the aforementioned theorems.

To construct families of $2$-cocycles directly on the level-$d$ mapping class group is not a priori easier than to construct directly invariants. However now we can take advantage of being able to pull-back cocycles from easier to understand quotients of the level-$d$ mapping class groups, and our most natural candidate is given by the abelianizations of these groups.  For convenience in the sequel we shift from an arbitrary integer  $d$ to a prime number $p\geq 5$. Set $H=H_1(\Sigma_g;\mathbb{Z})$, $H_p = H_1(\Sigma_g;\mathbb{Z}/p)$ and let $\mathfrak{sp}_{2g}(\mathbb{Z}/p)$ be the symplectic Lie algebra with coefficients in $\mathbb{Z}/p$.
In \cite{per}, \cite{Putman_abel}, \cite{sato_abel} independently  B.~Perron, A.~Putman and M.~Sato computed the abelianization of the level-$p$ mapping class group $\mathcal{M}_{g,1}[p]$, for an odd prime $p$ and an integer $g\geq 3$, getting a split short exact sequence of $\mathbb{Z}/p$-modules, which turns out to uniquely split as a sequence of $\mathcal{M}_{g,1}$-modules:
\[
	\xymatrix@C=7mm@R=7mm{
	0\ar[r] & \Lambda^3H_p \ar[r] & H_1( \mathcal{M}_{g,1}[p], \mathbb{Z}) \ar[r] & \mathfrak{sp}_{2g}(\mathbb{Z}/p) \ar[r]& 0.
}
	\]

In the integral case (cf. \cite{pitsch}), one gets the Casson invariant by pulling-back the unique candidate cocycle from $\Lambda^3 H,$ part of the abelianization of the Torelli group. Applying the same strategy in our case, we were surprised to see that  the suitable families of $2$-cocycles come from the side that we did not expect:

\begin{prop}\label{prop-cocy_sp-intro}
Given a prime number $p\geq 5$, the  $2$-cocycles on the abelianization of the levle-$p$ mapping class group with $\mathbb{Z}/p\mathbb{Z}$-values whose pull-back to the level-$p$ mapping class group satisfy all hypothesis of Theorem~\ref{teo_tool-intro} form a $p$-dimensional vector subspace of the vector space of $2$-cocycles on $\mathfrak{sp}_{2g}(\mathbb{Z}/p).$ 
\end{prop}

Moreover we are able to give a precise description of the invariants produced by these $2$-cocycles:

\begin{teo}\label{teo-main-intro}
Given a prime number $p\geq 5$, the invariants of homology $3$-spheres in $\mathcal{S}^3[p]$ induced by families of $2$-cocycles on the abelianization of the level-$p$ mapping class group are homological invariants. More precisely, given $M\in \mathcal{S}^3[p]$ and $n_0, n_1, n_2\in \mathbb{Z}/p$ the first three coefficients of the $p$-adic expansion of $n=|H_1(M;\mathbb{Z})|$, the following functions form a  basis for the space of the aforementioned invariants,
\[
\mathcal{P} = n_0n_2+\frac{n_0-1}{2} \text{ and } \varphi^k= \Big(n_0n_1+\frac{n_0-1}{2}\Big)^k \quad \text{with } k=1,\ldots p-1.
\]
\end{teo}

Finally, from Proposition \ref{prop-cocy_sp-intro} we are able to disprove  Perron's conjecture. More precisely, we show that if the extension of the Casson invariant modulo $p$ to the level-$p$ mapping class group given in \cite{per} by B.~Perron was a well defined invariant of rational homology spheres, then its associated $2$-cocycle would be the pull-back to $\mathcal{M}_{g,1}[p]$  of a bilinear form on the $\Lambda^3H_p$ subgroup of the abelianization of $\mathcal{M}_{g,1}[p]$ satisfying all hypothesis of Theorem \ref{teo_tool-intro}, which contradicts Proposition \ref{prop-cocy_sp-intro}. Therefore we have:

\begin{cor}
\label{cor-perron-intro}
Given a prime number $p\geq 5$, the extension of the Casson invariant modulo $p$ to the level-$p$ mapping class group proposed by B. Perron is not a well defined invariant of rational homology spheres in $\mathcal{S}^3[p]$.
\end{cor}

\vspace{0.3cm}

\textbf{Plan of this paper.} In Section 2, we give preliminary results on Heegaard splittings, symplectic representations of the handlebody subgroups, Lie algebras of arithmetic groups and some homological tools that we will use throughout all this work. We will be somewhat detailed here, for we will often rely on explicit forms of classical homological results. We include these preliminaries and discussions to save the reader from having to search through the literature.
In Section 3 we study the relation between rational homology $3$-spheres and level-$d$ mapping class groups for any integer $d\geq 2$, and prove Theorem \ref{teo-criteri-intro}.
In Section~4 we study the relation between families of $2$-cocycles on the level-$d$ mapping class group and invariants of oriented rational
homology $3$-spheres in $\mathcal{S}^3[d]$ and prove Proposition \ref{propintro:invmorphism} and Theorems \ref{teo_tool-intro}, \ref{teo_tool-introd2}.
In Section~5, for any prime $p\geq 5$, we study the families of $2$-cocycles on the abelianization of the level-$p$ mapping class group whose pull-back to the level-$p$ mapping class group satisfy the hypothesis of Theorem \ref{teo_tool-intro}. Using this we finally prove Proposition \ref{prop-cocy_sp-intro}, Theorem \ref{teo-main-intro} and Corollary \ref{cor-perron-intro}.

Some of the arguments in the above sections rely on homological computations that are sometimes lengthy and could interrupt the flow of the arguments. We have therefore decided to postpone them to the Appendix as they could also be of independent interest.

\textbf{Remark.} Most of our results depend on two integral parameters: the genus $g$ of the underlying surface and the ``depth'' $d$ of the group $\mathcal{M}_{g,1}[d]$. We usually assume $g \geq 4$ to avoid singular behavior of low-genus mapping class groups. This is usually harmless since we will strongly rely on properties that are largely insensitive to a variation of $g$ (a.k.a. stability).

A more serious restriction concerns $d$. When it is just an integer, our results usually require that $d \neq 0 \; (\text{mod } 4)$. This restriction comes from the very distinct cohomological properties shown by symplectic groups with coefficients in $\mathbb{Z}/d\mathbb{Z}$ depending on whether $d = 0 \; (\text{mod } 4)$ or not (cf. \cite{benson}, \cite{funarp}, \cite{Putman}). When $d$ is moreover prime, we have to further restrict ourselves to $d \neq 2 \text{ and } 3$ since for these primes the map  from $\mathcal{M}_{g,1}[d]$ to its abelianization has a very different cohomological behavior from other primes; these cases will be explored in a future work.

\textbf{Aknowlegdments} We would like to thank Luis Paris and Louis Funar for very enlightening comments on the first stages of this work. We also would like to thank Gwénaël Massuyeau and Quentin Faes for their comments and suggestions. Richard Hain also provided us with very helpful observations on this work. Finally we thank an anonymous referee for his thourough suggestions and for pointing out an error in a previous proof of our main result.

\section{Preliminaries}\label{sec:prel}

\subsection{Heegaard splittings of $3$-manifolds}\label{subsec:heegard}

Let $\Sigma_g$ be an oriented surface of genus $g$ standardly embedded in the $3$-sphere $\mathbf{S}^3$. Denote by $\Sigma_{g,1}$ the complement of the interior of a small disc embedded in $\Sigma_g$ and fix a base point $x_0$ on the boundary of $\Sigma_{g,1}$. 

\begin{figure}[H]
\begin{center}
\begin{tikzpicture}[scale=.7]
\draw[very thick] (-2,-2) -- (7,-2);
\draw[very thick] (-2,2) -- (7,2);
\draw[very thick] (-2,2) arc [radius=2, start angle=90, end angle=270];

\draw[very thick] (-2,0) circle [radius=.4];
\draw[very thick] (4.5,0) circle [radius=.4];
\draw[very thick] (0.5,0) circle [radius=.4];

\draw[thick, dotted] (2,0) -- (3,0);

\node at (-2,0) {\tiny{1}};
\node at (0.5,0) {\tiny{2}};
\node at (4.5,0) {\tiny{g}};

\draw[thick,pattern=north west lines] (7,-2) to [out=130,in=-130] (7,2) to [out=-50,in=50] (7,-2);


\end{tikzpicture}
\end{center}
\caption{Standardly embedded $\Sigma_{g,1}$ in $\mathbf{S}^3$}
\label{fig_stand_embded}
\end{figure}
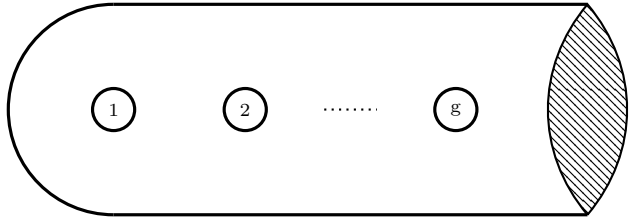

The surface $\Sigma_g$ separates the sphere into two genus $g$ handlebodies. By the inner handlebody $\mathcal{H}_g$ we  mean the one that is visible in Figure \ref{fig_stand_embded} and by the outer handlebody $-\mathcal{H}_{g}$  the complementary handlebody; in both cases we identify the handlebodies with some fixed model.

We orient our sphere so that the inner handlebody has the orientation compatible with that of $\Sigma_g$; then both handlebodies receive opposite orientations whence the notation. This presents $\mathbf{S}^3$ as the union of both handlebodies glued along their boundary by some orientation-reversing diffeomorphism $\iota_g: \Sigma_{g,1} \rightarrow \Sigma_{g,1}$, and we write $\mathbf{S}^3 = \mathcal{H}_g \cup_{\iota_g}-\mathcal{H}_g$.

Let 
$$\mathcal{M}_{g,1}=\pi_0(\text{Diff}^+(\Sigma_{g,1}, \partial \Sigma_{g,1})),$$
denote the mapping class group of $\Sigma_{g,1}$ relative to the boundary. The above decomposition of the sphere hands us out three canonical subgroups of the mapping class group:

\begin{itemize}
\item $\mathcal{A}_{g,1}$, the subgroup of mapping classes that are restrictions of diffeomorphisms of the outer handlebody $-\mathcal{H}_g,$
\item $\mathcal{B}_{g,1}$, the subgroup of mapping classes that are restrictions of diffeomorphisms of the inner handlebody $\mathcal{H}_g,$
\item $\mathcal{AB}_{g,1}$, the subgroup of mapping classes that are restrictions of diffeomorphisms of the $3$-sphere $\mathbf{S}^3.$
\end{itemize}
It is a theorem of Waldhausen, \cite{wald} that in fact, $\mathcal{AB}_{g,1}=\mathcal{A}_{g,1}\cap\mathcal{B}_{g,1}$.

We can stabilize these subgroups by gluing one of the boundary components of a two-holed torus along the boundary of $\Sigma_{g,1}$, this defines an embedding $\Sigma_{g,1} \hookrightarrow \Sigma_{g+1,1}$ and extending an element of $\mathcal{M}_{g,1}$ by the identity over the torus defines the stabilization  morphism $\mathcal{M}_{g,1}\hookrightarrow \mathcal{M}_{g+1,1}$, that is known to be injective. The definition of the above subgroups is compatible with the stabilization, and we have induced injective morphisms $\mathcal{A}_{g,1} \hookrightarrow \mathcal{A}_{g+1,1}$, etc.  Similarly, the gluing map $\iota_g$ is compatible with stabilization in the sense that properly chosen we may assume: $\iota_{g+1|_{\Sigma_{g,1}}} = \iota_g$.

Denote by $\mathcal{V}$ the set of oriented diffeomorphism classes of closed  oriented smooth 3-manifolds. Given an element $\phi \in \mathcal{M}_{g,1}$, we denote by $S^3_\phi =  \mathcal{H}_g \cup_{\iota_g\phi}-\mathcal{H}_g$ the manifold obtained by gluing the handlebodies\footnote{To be more precise, to build $S^3_\phi$ one has to fix a representative element $f$ of $\phi $ and glue along $\iota_g f$. The diffeomorphism type of the manifold is independent of the diffeotopy class of $f$, hence the abuse of notation~\cite[Chap.~8]{Hirsch}.}  along $\iota_g\phi$. One checks that this gives a well-defined element in $\mathcal{V}$, and that this construction is compatible with stabilization, giving rise to a well defined map 
\[
\lim_{g\to \infty} \mathcal{M}_{g,1} \longrightarrow \mathcal{V},
\]  

which by a standard Morse theory argument is surjective. If $M$ is an oriented closed $3$-manifold, to fix a diffeomorphism $M \simeq S^3_\phi$ is by definition to give a Heegaard splitting of $M$. The failure of injectivity of this map is well understood. Consider the following double coset equivalence relation on $\mathcal{M}_{g,1}:$
\begin{equation}\label{eq_rel1}
\phi \sim \psi \Longleftrightarrow  \exists \zeta_a \in \mathcal{A}_{g,1}\;\exists \zeta_b \in \mathcal{B}_{g,1}  \text{ such that }  \zeta_a \phi \zeta_b=\psi.
\end{equation}
This equivalence relation is again compatible with the stabilization map, which induces well-defined maps between the quotient sets for genus $g$ and $g+1$, and we have:

\begin{teo}[J. Singer \cite{singer}]
\label{bij_MCG_3man}
The following map is well defined and is bijective:
\[
\begin{array}{ccl}
	{\displaystyle\lim_{g\to \infty}\mathcal{A}_{g,1}\backslash \mathcal{M}_{g,1}/\mathcal{B}_{g,1}} & \longrightarrow & \mathcal{V} \\
\phi & \longmapsto & S^3_\phi.
\end{array}
\]
\end{teo}

Our first goal will be to show in Section \ref{sec:parammoddhlgysphere} that this sort of parametrization by double classes holds true for interesting subsets of $\mathcal{V}$.

\subsection{The Symplectic representation}\label{subsec:symprep}

Fix a basis of $H_1(\Sigma_{g,1};\mathbb{Z})$ as in Figure~\ref{fig:homology_basis}. Transverse intersection of oriented paths on $\Sigma_{g,1}$ induces a symplectic form $\omega$ on $H_1(\Sigma_{g,1};\mathbb{Z})$, with $\omega(a_i,b_i)=-\omega(b_i,a_i)=1$ and zero otherwise. Moreover, both sets of displayed homology classes $\{ a_i \ | \ 1 \leq i \leq g\}$ and $\{ b_i \ | \ 1 \leq i \leq g \}$ form a symplectic basis, and in particular generate supplementary transverse Lagrangians $A$ and $B$.

As a symplectic space we write $H_1(\Sigma_{g,1};\mathbb{Z}) = A \oplus B$.     

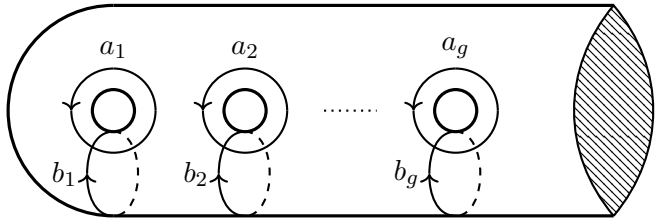
\begin{figure}[H]
\begin{center}
\begin{tikzpicture}[scale=.7]
\draw[very thick] (-4.5,-2) -- (5,-2);
\draw[very thick] (-4.5,2) -- (5,2);
\draw[very thick] (-4.5,2) arc [radius=2, start angle=90, end angle=270];

\draw[very thick] (-4.5,0) circle [radius=.4];
\draw[very thick] (-2,0) circle [radius=.4];
\draw[very thick] (2,0) circle [radius=.4];

\draw[thick, dotted] (-0.5,0) -- (0.5,0);

\draw[<-,thick] (1.2,0) to [out=90,in=180] (2,0.8);
\draw[thick] (2.8,0) to [out=90,in=0] (2,0.8);
\draw[thick] (1.2,0) to [out=-90,in=180] (2,-0.8) to [out=0,in=-90] (2.8,0);

\draw[<-, thick] (-5.3,0) to [out=90,in=180] (-4.5,0.8);
\draw[thick] (-3.7,0) to [out=90,in=0] (-4.5,0.8);
\draw[thick] (-5.3,0) to [out=-90,in=180] (-4.5,-0.8) to [out=0,in=-90] (-3.7,0);

\draw[<-,thick] (-2.8,0) to [out=90,in=180] (-2,0.8);
\draw[thick] (-1.2,0) to [out=90,in=0] (-2,0.8);
\draw[thick] (-2.8,0) to [out=-90,in=180] (-2,-0.8) to [out=0,in=-90] (-1.2,0);

\draw[thick] (-4.5,-0.4) to [out=180,in=90] (-5,-1.2);
\draw[->,thick] (-4.5,-2) to [out=180,in=-90] (-5,-1.2);
\draw[thick, dashed] (-4.5,-0.4) to [out=0,in=0] (-4.5,-2);
\draw[thick] (-2,-0.4) to [out=180,in=90] (-2.5,-1.2);
\draw[->,thick] (-2,-2) to [out=180,in=-90] (-2.5,-1.2);
\draw[thick, dashed] (-2,-0.4) to [out=0,in=0] (-2,-2);
\draw[thick] (2,-0.4) to [out=180,in=90] (1.5,-1.2);
\draw[->,thick] (2,-2) to [out=180,in=-90] (1.5,-1.2);
\draw[thick, dashed] (2,-0.4) to [out=0,in=0] (2,-2);

\node [left] at (-5,-1.2) {$b_1$};
\node [above] at (-4.5,0.8) {$a_1$};
\node [left] at (-2.5,-1.2) {$b_2$};
\node [left] at (1.5,-1.2) {$b_g$};
\node [above] at (-2,0.8) {$a_2$};
\node [above] at (2,0.8) {$a_g$};

\draw[thick,pattern=north west lines] (5,-2) to [out=130,in=-130] (5,2) to [out=-50,in=50] (5,-2);

\end{tikzpicture}
\end{center}
\caption{Homology basis of $H_1(\Sigma_{g,1};\mathbb{Z})$}
\label{fig:homology_basis}
\end{figure}

 The symplectic form is preserved by the natural action of the mapping class group on the first homology of $\Sigma_{g,1}$ and gives rise to the symplectic representation:
\[
 \mathcal{M}_{g,1}\longrightarrow Sp_{2g}(\mathbb{Z}).
\]
It is known, because for instance Dehn twists project onto transvections and these generate the symplectic groups, that this map is surjective (cf. \cite{burk}).
We will usually write elements in the symplectic groups by block matrices with respect to our choice of transverse Lagrangians  $A \oplus B$.

\subsection{Computing the first homology group from a Heegaard splitting}\label{subsec:computeH1(M)}

Fix a Heegaard splitting of $M \in \mathcal{V}$. This describes $M$ as a push-out:
\[
\xymatrix{
\Sigma_{g,1}  \ar@{^(->}[r] \ar@{^(->}[d]_{\iota_g \phi}& \mathcal{H}_g \ar[d]  \\
  -\mathcal{H}_{g} \ar[r] & M.
}
\]
Observe that canonically, for any coefficient ring $R,$ $H_1(\mathcal{H}_g;R) \simeq A \otimes R$ and $H_1(-\mathcal{H}_g;R) \simeq B \otimes R$. Write
\[
H_1(\phi;R)=\left(\begin{matrix}
E & F \\
G & H
\end{matrix}\right),
\]
and consider the Mayer-Vietoris sequence associated to the push-out:
\[
\xymatrix{
	0 \ar[r] & H_2(M;R) \ar[r] & H_1( \Sigma_g;R) \ar[r]^-{\Phi} & H_1(\mathcal{H}_g;R)\oplus  H_1(-\mathcal{H}_g;R)
	\ar@{->} `r/8pt[d] `/10pt[l] `^dl[ll] `^r/3pt[dll] [dll] \\
	& H_1(M;R) \ar[r] & 0 & \\ 
}
\]

Homologically, the map $\iota_g$ exchanges the basis elements $a_i$ and $-b_i$. 
A direct inspection shows that then  $\Phi=\left(\begin{smallmatrix}
Id & 0 \\
G & H
\end{smallmatrix}\right).$
Therefore,
\begin{align*}
H_2(M;R)= & \ker\left(\Phi: H_1(\Sigma_g;R) \longrightarrow H_1(\mathcal{H}_g;R)\oplus  H_1(-\mathcal{H}_g;R)\right)=\ker(H), \\
H_1(M;R)= & \coker\left(\Phi: H_1(\Sigma_g;R) \longrightarrow H_1(\mathcal{H}_g;R)\oplus  H_1(-\mathcal{H}_g;R)\right) \\
& =\coker(H).
\end{align*}

We are particularly interested in $R$-homology spheres:

\begin{defi}\label{def:hmlgyspheres}
 An oriented $3$-manifold $M$ is an $R$-homology $3$-sphere if there is an isomorphism\footnote{If the isomorphism exists abstractly it is easy to check that the map $M \rightarrow \mathbf{S}^3$ that collapses the complement of any small ball embedded in $M$ induces then another such isomorphism.} :
\[
H_*(M;R) \simeq H_*(\mathbf{S}^3;R).
\]
 We denote the set of $\mathbb{Q}$-homology spheres by $\mathcal{S}_{\mathbb{Q}}$, the set of integral homology spheres by $\mathcal{S}_ \mathbb{Z}$ and the set of $\mathbb{Z}/d\mathbb{Z}$-spheres by $\mathcal{S}^3(d)$.
\end{defi}

The above computation shows in particular that if a mapping class  $\phi \in \ker H_1(-;R)$, then $\Phi = Id$ and the resulting manifold $S^3_\phi$ is an $R$-homology sphere. The case $R = \mathbb{Z}$ is classical and has been extensively studied.

Consider the short exact sequence:
\begin{equation}
\label{symp_rep_ext}
\xymatrix@C=7mm@R=10mm{1 \ar@{->}[r] & \mathcal{T}_{g,1} \ar@{->}[r] & \mathcal{M}_{g,1} \ar@{->}[r] & Sp_{2g}(\mathbb{Z}) \ar@{->}[r] & 1.}
\end{equation}

The subgroup $\mathcal{T}_{g,1}$ is known as the  Torelli group, and we have just argued that performing a Heegaard splitting with gluing map $\phi \in \mathcal{T}_{g,1}$ gives an integral homology sphere. It is a result of S. Morita~\cite{mor} that the converse is true, any integral homology sphere can be obtained with such a gluing map, so that Singer's map induces a bijection:

\[
\begin{array}{ccl}
	{\displaystyle
\lim_{g\to \infty}\mathcal{A}_{g,1}\backslash \mathcal{T}_{g,1}/\mathcal{B}_{g,1}} & \longrightarrow & \mathcal{S}^3_\mathbb{Z}, \\
\phi & \longmapsto & S^3_\phi=\mathcal{H}_g \cup_{\iota_g\phi} -\mathcal{H}_g.
\end{array}
\]
Here $\mathcal{A}_{g,1}\backslash \mathcal{T}_{g,1}/\mathcal{B}_{g,1}$ stands for the image of the canonical map $\mathcal{T}_{g,1}\rightarrow \mathcal{A}_{g,1}\backslash \mathcal{M}_{g,1}/\mathcal{B}_{g,1}$.

We wish to find the same kind of parametrization for  $\mathcal{S}^3_\mathbb{Q}$ and $\mathcal{S}^3(d)$. To do this we need to properly define the corresponding Torelli and handlebody subgroups for these situations.

\subsection{Symplectic representation  of  handlebody subgroups}
\label{homol_homoto_actions}

By construction the action of $\mathcal{B}_{g,1}$ (resp. $\mathcal{A}_{g,1})$ on $H_1(\Sigma_{g,1};\mathbb{Z})$ preserves the Lagrangian $B$ (resp. $A)$. Accordingly, the images of these subgroups and of $\mathcal{AB}_{g,1}$ lie in the stabilizers of $B$ (resp. $A$, resp. of the pair $(A,B)$) in the symplectic groups, which we denote $Sp^A_{2g}(\mathbb{Z})$, $Sp^B_{2g}(\mathbb{Z})$, resp. $Sp^{AB}_{2g}(\mathbb{Z})$.  Checking for instance on generators (see for example Griffith \cite{grif}) one proves that these images are exactly these stabilizer subgroups in $Sp_{2g}(\mathbb{Z})$, which again we write by blocks: 
\begin{align*}
\mathcal{A}_{g,1} \twoheadrightarrow Sp_{2g}^{A}(\mathbb{Z})=& \{ \text{symplectic matrices of the form } \left(\begin{smallmatrix}
* & * \\
0 & *
\end{smallmatrix} \right) \} ,\\
\mathcal{B}_{g,1} \twoheadrightarrow Sp_{2g}^{B}(\mathbb{Z})=& \left\lbrace \text{symplectic matrices of the form } \left(\begin{smallmatrix}
* & 0 \\
* & *
\end{smallmatrix} \right) \right\rbrace ,\\
\mathcal{AB}_{g,1} \twoheadrightarrow Sp_{2g}^{AB}(\mathbb{Z})=& \left\lbrace \text{symplectic matrices of the form } \left(\begin{smallmatrix}
* & 0 \\
0 & *
\end{smallmatrix} \right) \right\rbrace.
\end{align*}

For any integer $g$, and any coefficient ring $R,$ let $Sym_g(R)$ denote the group of symmetric $g \times g$ matrices. In the above decompositions 
a matrix of the form $\left(\begin{smallmatrix}
G_1 & 0 \\
M & G_2
\end{smallmatrix} \right)$ (resp. $\left(\begin{smallmatrix}
G_1 & N \\
0 & G_2
\end{smallmatrix} \right)$)
is symplectic if and only if $G_1,$ $G_2$ are invertible with $G_2={}^tG_1^{-1}$
and ${}^tG_1 M$ (resp. $G_1^{-1}N$) is symmetric. Therefore, we have isomorphisms:
\begin{equation*}
\begin{aligned}
Sp_{2g}^B(\mathbb{Z}) & \longrightarrow GL_g(\mathbb{Z}) \ltimes Sym^B_g(\mathbb{Z}), \\
\left(\begin{smallmatrix}
G & 0 \\
M & {}^tG^{-1}
\end{smallmatrix}\right) & \longmapsto (G,{}^tGM)
\end{aligned} \quad
\begin{aligned}
Sp_{2g}^A(\mathbb{Z}) & \longrightarrow GL_g(\mathbb{Z}) \ltimes Sym^A_g(\mathbb{Z}), \\
\left(\begin{smallmatrix}
G & N \\
0 & {}^tG^{-1}
\end{smallmatrix}\right) & \longmapsto (G,G^{-1}M)
\end{aligned} 
\end{equation*}
\begin{align*}
Sp_{2g}^{AB}(\mathbb{Z}) & \longrightarrow GL_g(\mathbb{Z}). \\
\left(\begin{smallmatrix}
G & 0 \\
0 & {}^tG^{-1}
\end{smallmatrix} \right) & \longmapsto  G
\end{align*}
Here the compositions on the semi-direct products are given by the rules
\begin{align*}
(G,S)(H,T)= & (GH,{}^tHSH+T) \quad \text{for } Sp_{2g}^B(\mathbb{Z}), \\
(G,S)(H,T)= & (GH,{}^tH^{-1}SH^{-1}+T) \quad \text{for } Sp_{2g}^A(\mathbb{Z}).
\end{align*}
In all what follows we use the superscripts $A$, $B$ to distinguish between these two composition rules. Moreover, we will often use the embedding $GL_g(\mathbb{Z})\simeq Sp_{2g}^{AB}(\mathbb{Z}) \hookrightarrow Sp_{2g}(\mathbb{Z})$ without mention.
We also denote by 
\[
\mathcal{TA}_{g,1} = \mathcal{A}_{g,1} \cap \mathcal{T}_{g,1}, \ \mathcal{TB}_{g,1} = \mathcal{B}_{g,1} \cap \mathcal{T}_{g,1} \ \text{ and } \ \mathcal{TAB}_{g,1} = \mathcal{AB}_{g,1} \cap \mathcal{T}_{g,1},
\]
 the kernels of the symplectic representation restricted to $\mathcal{A}_{g,1}$, $\mathcal{B}_{g,1}$, $\mathcal{AB}_{g,1}$.

\subsection{Modulo $d$ handlebody and Torelli groups}

Let $d\geq 2$ be an integer, and let $Sp_{2g}(\mathbb{Z}/d)$ denote the symplectic group  with mod $d$ coefficients. It is a non-obvious fact ( cf. \cite[Thm.~1]{newman}) that mod $d$ reduction of the coefficients gives a surjective map:
\[
Sp_{2g}(\mathbb{Z}) \longrightarrow Sp_{2g}(\mathbb{Z}/d).
\]

Composing the symplectic representation of the mapping class group with this reduction we get a short exact sequence, whose kernel is by definition the mod $d$ Torelli group:
\begin{equation}
\label{ses_def_M[p]}
\xymatrix@C=7mm@R=7mm{1 \ar@{->}[r] & \mathcal{M}_{g,1}[d] \ar@{->}[r] & \mathcal{M}_{g,1} \ar@{->}[r] & Sp_{2g}(\mathbb{Z}/d) \ar@{->}[r] & 1 .}
\end{equation}

Restricting the mod $d$ Torelli group to $\mathcal{A}_{g,1}$, $\mathcal{B}_{g,1}$, $\mathcal{AB}_{g,1}$ we get the following mod $d$ handlebody subgroups:
\begin{align*}
\mathcal{A}_{g,1}[d] & =\mathcal{M}_{g,1}[d]\cap \mathcal{A}_{g,1}, \\
 \mathcal{B}_{g,1}[d] & =\mathcal{M}_{g,1}[d]\cap \mathcal{B}_{g,1}, \\ 
 \mathcal{AB}_{g,1}[d],& =\mathcal{M}_{g,1}[d]\cap \mathcal{AB}_{g,1}.
\end{align*}

The mod $d$ Torelli group and the Torelli group are closely related. Indeed, if we denote $Sp_{2g}(\mathbb{Z},d)=\ker (Sp_{2g}(\mathbb{Z})\rightarrow Sp_{2g}(\mathbb{Z}/d))$, the level $d$ congruence subgroup, the symplectic representation restricted to the mod $d$ Torelli group gives a short exact sequence
\begin{equation*}
\xymatrix@C=7mm@R=7mm{1 \ar@{->}[r] & \mathcal{T}_{g,1} \ar@{->}[r] & \mathcal{M}_{g,1}[d] \ar@{->}[r] & Sp_{2g}(\mathbb{Z},d) \ar@{->}[r] & 1 .}
\end{equation*}

 The first trivial  observation is that by construction:

\begin{prop}\label{prop:homologySphmodd}
	Let $d \geq 2$ be an integer. If $\phi \in \mathcal{M}_{g,1}[d]$, then the resulting manifold  $S^3_{\phi}= \mathcal{H}_g \cup_{\iota_g\phi} -\mathcal{H}_g$ is a  $\mathbb{Z}/d$-homology sphere.
\end{prop}

To understand whether the converse of this proposition holds true, we will first explicitly write down the action of these handlebody subgroups  on $H_1(\Sigma_{g,1};\mathbb{Z})$.

\subsection{Modulo d homology actions of handlebody groups}\label{subsec:modhandlebody}

To avoid too heavy notation we will sometimes abbreviate $H:=H_1(\Sigma_{g,1};\mathbb{Z}),$ $H_d:=H_1(\Sigma_{g,1};\mathbb{Z}/d)$ and similarly $A_d$, $B_d$ for the canonical images of the Lagrangians $A$ and $B$ in $H_d$.

The main difference between the action of the mapping class group on $H_d$ and its action on $H$ is  that, when restricted to  $\mathcal{B}_{g,1}$, the action does not induce a surjective homomorphism onto the stabilizer of the Lagrangian $B$ mod $d$.

As in the integer case,  writing the matrices of the symplectic group $Sp_{2g}(\mathbb{Z}/d)$ by blocks according to the decomposition $H_d=A_d\oplus B_d$, by direct inspection we have isomorphisms for the stabilizers of these Lagrangians:
\begin{align*}
Sp_{2g}^A(\mathbb{Z}/d)  & \simeq GL_g(\mathbb{Z}/d) \ltimes Sym_g^A(\mathbb{Z}/d), \\
Sp_{2g}^B(\mathbb{Z}/d) & \simeq GL_g(\mathbb{Z}/d) \ltimes Sym_g^B(\mathbb{Z}/d), \\
Sp_{2g}^{AB}(\mathbb{Z}/d) & \simeq  GL_g(\mathbb{Z}/d).
\end{align*}
Let us also denote the kernels of the maps induced by reducing the coefficients mod $d$ by:
\begin{align*}
SL_g(\mathbb{Z},d)= & \ker\left(r_d:SL_g(\mathbb{Z})\rightarrow SL_g(\mathbb{Z}/d)\right), \\
Sym_g(d\mathbb{Z})= & \ker\left(r_d:Sym_g(\mathbb{Z})\rightarrow Sym_g(\mathbb{Z}/d)\right).
\end{align*}
Since the  action in mod $d$ homology of $\mathcal{A}_{g,1}$ (resp. $\mathcal{B}_{g,1}$, resp. $\mathcal{AB}_{g,1})$ still stabilizes the Lagrangian $A_d$ (resp. $B_d$, resp. both Lagrangians), we have a commutative diagram:
\[
\xymatrix{
0 \ar[r] & \mathcal{TA}_{g,1}\ar[r] \ar[d] &  \mathcal{A}_{g,1} \ar@{=}[d] \ar[r] & Sp^A_{2g}(\mathbb{Z}) \ar[d]^{r_d}
 \ar[r] & 0 \\
0 \ar[r] & \mathcal{A}_{g,1}[d]  \ar[r] & \mathcal{A}_{g,1} \ar[r] & Sp^A_{2g}(\mathbb{Z}/d)
}
\]
There are of course similar diagrams for $\mathcal{B}_{g,1}$ and $\mathcal{AB}_{g,1}$. 

The main obstacle for us in proving a Singer-type result for $\mathbb{Z}/d$-homology spheres is that  in almost no case is the rightmost vertical arrow surjective. This is due to the fact that the map $GL_g(\mathbb{Z})\rightarrow GL_g(\mathbb{Z}/d)$ is not in general surjective.

We now proceed to compute the image of the handlebody subgroups in $Sp_{2g}(\mathbb{Z}/d)$. For this we take advantage of the description of these stabilizer subgroups as semi-direct products. We may however, by the kernel-cokernel exact sequence, record for further reference the action in integral homology of the mod $d$ handlebody subgroups:

\begin{lema} \label{lem:extensionshandelbody}
	Given an integer $d\geq 3,$ the restriction of the symplectic representation to the mod d handlebody subgroups $\mathcal{A}_{g,1}[d],$ $\mathcal{B}_{g,1}[d],$ $\mathcal{AB}_{g,1}[d]$ gives us the following extensions of groups:
	\begin{align*}
	& \xymatrix@C=7mm@R=3mm{
		1 \ar@{->}[r] & \mathcal{TA}_{g,1} \ar@{->}[r] & \mathcal{A}_{g,1}[d] \ar@{->}[r]  & SL_g(\mathbb{Z},d)\ltimes Sym_g^A(d\mathbb{Z}) \ar@{->}[r] & 1,\\
		1 \ar@{->}[r] & \mathcal{TB}_{g,1} \ar@{->}[r] & \mathcal{B}_{g,1}[d] \ar@{->}[r]  & SL_g(\mathbb{Z},d)\ltimes Sym_g^B(d\mathbb{Z}) \ar@{->}[r] & 1,\\
		1 \ar@{->}[r] & \mathcal{TAB}_{g,1} \ar@{->}[r] & \mathcal{AB}_{g,1}[d] \ar@{->}[r] & SL_g(\mathbb{Z},d) \ar@{->}[r] & 1 .}
	\end{align*}
\end{lema}

In the following sections, to avoid unnecessarily cumbersome notation, we denote by $Sp_{2g}^{A}(\mathbb{Z},d)$ (resp. $Sp_{2g}^{B}(\mathbb{Z},d)$) the image of $\mathcal{A}_{g,1}[d]$ (resp. $\mathcal{B}_{g,1}[d]$) in $Sp_{2g}(\mathbb{Z},d).$

Let $R$ be a ring. Denote by $E_g(R)$ the subgroup of $SL_g(R)$ generated by the elementary matrices of rank $g$, i.e. those matrices of the form $Id + e_{ij}$ with $i \neq j$ where $e_{ij}$ is zero but for entry $(i,j)$ which is equal to $1$. In general, for an arbitrary ring $R$, the subgroup $E_g(R)$ differs from $SL_g(R)$. However, if we assume that $R$ is a semilocal ring (in particular if $R=\mathbb{Z}/d$ with $d$ a positive integer) then these two groups coincide for any positive integer $g$, (cf.\cite[Theorem~4.3.9]{omeara}). 
Denote by $D_g$ the matrix $\left(\begin{smallmatrix}
-1 & 0 \\
0 & Id_{g-1}
\end{smallmatrix}\right).$ 
Finally set $SL^\pm_g(\mathbb{Z}/d)= \{ A\in M_{g\times g}(\mathbb{Z}/d) \mid det(A)=\pm 1 \}$.

\begin{lema}\label{lem_Sp[p]_2}
Given an integer $d\geq 3$, mod $d$ reduction of the coefficients induces a short exact sequence of groups
$$
\xymatrix@C=7mm@R=13mm{1 \ar@{->}[r] & SL_g(\mathbb{Z},d) \ar@{->}[r] & GL_g(\mathbb{Z}) \ar@{->}[r]^-{r_d} & SL_g^\pm(\mathbb{Z}/d) \ar@{->}[r] & 1. }
$$
\end{lema}
\begin{proof} The only non-trivial fact is surjectivity.  Clearly, mod-$d$ reduction induces a map $GL_g(\mathbb{Z}) \rightarrow SL_g^\pm(\mathbb{Z}/d)$, since the only invertibles in $\mathbb{Z}$ are $\pm 1$. To show this is surjective, observe that for any ring $R$ we have $SL_g^{\pm}(R) = D_g\cdot SL_g(R)$. Clearly mod-$d$ reduction maps the matrix $D_g$ onto $D_g$. Finally,    by \cite[Theorem 4.3.9]{omeara} for any positive integer $d$, $SL_g(\mathbb{Z}/d)=E_g(\mathbb{Z}/d)$, and clearly $r_d(E_g(\mathbb{Z})) = E_g(\mathbb{Z}/d)$. 
\end{proof}

By \cite[Lemma 1]{newman},
for any integers $d$ and $g$ there is a short exact sequence of groups
$$
\xymatrix@C=7mm@R=13mm{1 \ar@{->}[r] & Sym_g(d\mathbb{Z}) \ar@{->}[r] & Sym_g(\mathbb{Z}) \ar@{->}[r]^-{r_d} & Sym_g(\mathbb{Z}/d) \ar@{->}[r] & 1 .}
$$
From Lemma~\ref{lem_Sp[p]_2} and this exact sequence we immediately get:

\begin{lema}\label{lem_Sp[p]_3}
Given an integer $d\geq 3$, mod-$d$ reduction of the coefficients induces a short exact sequence of groups:
$$
\xymatrix@C=4mm@R=4mm{1 \ar@{->}[r] & SL_g(\mathbb{Z},d) \ltimes Sym_g(d\mathbb{Z}) \ar@{->}[r] & GL_g(\mathbb{Z})\ltimes Sym_g(\mathbb{Z}) 	\ar@{->} `r/8pt[d] `/10pt[l] `^dl[ll] `^r/3pt[dll] [dl] \\
  & SL^\pm_g(\mathbb{Z}/d)\ltimes Sym_g(\mathbb{Z}/d) \ar@{->}[r] & 1.}
$$
\end{lema}

\begin{rem}
\label{rem:d=2}
For $d=2$, the Lemmas \ref{lem:extensionshandelbody}, \ref{lem_Sp[p]_2}, \ref{lem_Sp[p]_3} do not hold since $SL_g^\pm(\mathbb{Z}/2)$ coincides with $GL_g(\mathbb{Z}/2)$ and as a consequence, the kernel of $r_2:GL_g(\mathbb{Z})\rightarrow SL_g^\pm(\mathbb{Z}/2)$ is $GL_g(\mathbb{Z},2)$ instead of $SL_g(\mathbb{Z},2)$. Nonetheless, if we replace $SL_g(\mathbb{Z},2)$ by $GL_g(\mathbb{Z},2)$ the aforementioned lemmas hold.
\end{rem}

We now have, for any positive integer $d$, maps of short exact sequences of groups (we only treat the case $\mathcal{A}_{g,1}$, the other two are similar):

\[
\xymatrix{
	0 \ar[r] & \mathcal{TA}_{g,1}\ar[r] \ar[d] &  \mathcal{A}_{g,1} \ar@{=}[d] \ar[r] & Sp^A_{2g}(\mathbb{Z}) \ar[d]^{r_d}
	\ar[r] & 0 \\
	0 \ar[r] & \mathcal{A}_{g,1}[d]  \ar[r] & \mathcal{A}_{g,1} \ar[r] & SL_g^{\pm}(\mathbb{Z}/d) \ltimes Sym_g^A(\mathbb{Z}/d) \ar[r]& 0
}
\]

\begin{lema}\label{lem_B[p]}
Given a positive integer $d$, the restriction of the symplectic representation modulo $d$ to $\mathcal{A}_{g,1},$ $\mathcal{B}_{g,1}$ and $\mathcal{AB}_{g,1}$ give us the following extensions of groups:
\begin{align*}
& \xymatrix@C=7mm@R=3mm{1 \ar@{->}[r] & \mathcal{A}_{g,1}[d] \ar@{->}[r] & \mathcal{A}_{g,1} \ar@{->}[r] & SL_g^\pm(\mathbb{Z}/d)\ltimes Sym_g^A(\mathbb{Z}/d) \ar@{->}[r] & 1,\\
1 \ar@{->}[r] & \mathcal{B}_{g,1}[d] \ar@{->}[r] & \mathcal{B}_{g,1} \ar@{->}[r] & SL_g^\pm(\mathbb{Z}/d)\rtimes Sym_g^B(\mathbb{Z}/d) \ar@{->}[r] & 1,\\
1 \ar@{->}[r] & \mathcal{AB}_{g,1}[d] \ar@{->}[r] & \mathcal{AB}_{g,1} \ar@{->}[r] & SL_g^\pm(\mathbb{Z}/d) \ar@{->}[r] & 1. }
\end{align*}
\end{lema}

\subsection{Lie algebras of arithmetic groups}\label{subsec:Liealgebras}

In this section we present the Lie algebras of some arithmetic groups and related properties that we will use throughout this work.

Let $d\geq 2$ be a positive integer, we denote by $\mathfrak{gl}_{g}(\mathbb{Z}/d)$, $\mathfrak{sl}_{g}(\mathbb{Z}/d)$, $\mathfrak{sp}_{2g}(\mathbb{Z}/d)$ the Lie algebras of $GL_{g}(\mathbb{Z}/d)$, $SL_{g}(\mathbb{Z}/d)$, $Sp_{2g}(\mathbb{Z}/d)$ respectively, which can be described as follows:
\begin{align*}
\mathfrak{gl}_g(\mathbb{Z}/d)=&\lbrace\text{additive group of }  g\times g \text{ matrices  with coefficients in } \mathbb{Z}/d \rbrace,\\
\mathfrak{sl}_g(\mathbb{Z}/d)=&\lbrace M\in \mathfrak{gl}_g(\mathbb{Z}/d) \mid tr(M)=0 \rbrace, \\
\mathfrak{sp}_{2g}(\mathbb{Z}/d)= & \lbrace M\in \mathfrak{gl}_{2g}(\mathbb{Z}/d) \mid {}^tM J + J M = 0  \text{ where } J = \left( \begin{smallmatrix} 0 & Id \\ -Id & 0  \end{smallmatrix} \right) \rbrace.
\end{align*}
From the definition of the Symplectic Lie algebra, if one writes its elements by blocks of the form $\left(\begin{smallmatrix}
		\alpha & \beta \\
		\gamma & \delta
		\end{smallmatrix}\right)$
one gets that $\alpha=-{}^t\delta,$ $\beta={}^t\beta,$ $\gamma={}^t\gamma.$ Therefore there is a decomposition of $\mathbb{Z}/d$-modules:
\begin{equation}
	\label{dec_sp}
	\mathfrak{sp}_{2g}(\mathbb{Z}/d)\simeq \mathfrak{gl}_g(\mathbb{Z}/d)\oplus Sym^A_g(\mathbb{Z}/d)\oplus Sym^B_g(\mathbb{Z}/d).
	\end{equation}
	where the superscript $A$ or $B$ refers to the position of the symmetric group ($A$ for position of block matrix $\beta$ and $B$ for position of block matrix $\gamma$).
	
Moreover this is a decomposition as $GL_g(\mathbb{Z})$-modules, where  $GL_g(\mathbb{Z})$ acts  on $\mathfrak{sp}_{2g}(\mathbb{Z}/d)$ by conjugation of matrices of the form $\left(\begin{smallmatrix}
G & 0 \\
0 & {}^tG^{-1}
\end{smallmatrix}\right)$ with $G\in GL_g(\mathbb{Z})$. More precisely $GL_g(\mathbb{Z})$ acts on $\mathfrak{gl}_g(\mathbb{Z}/d)$  by conjugation, on $Sym^A_g(\mathbb{Z}/d)$  by $G.\beta=G\beta\;{}^t G$
and  on $Sym^B_g(\mathbb{Z}/d)$ by $G.\gamma={}^tG^{-1}\gamma \;G^{-1}$.

The above Lie algebras are closely related to the abelianization functors. We will only explain the details for the symplectic group since this is the case of most interest to us. In \cite[ Section 1]{Lee}, R. Lee and R.H. Szczarba showed that given an integer $d\geq 2$ there is an $Sp_{2g}(\mathbb{Z})$-equivariant homomorphism
\begin{equation}
\label{def-alpha}
\begin{aligned}
\alpha: Sp_{2g}(\mathbb{Z},d) & \longrightarrow \mathfrak{sp}_{2g}(\mathbb{Z}/d) \\
Id_{2g} +dA & \longmapsto A \;(\text{mod }d)
\end{aligned}
\end{equation}

which induces a short exact sequence
\begin{equation}
\label{ses_abel_d1}
\xymatrix@C=7mm@R=13mm{1 \ar@{->}[r] & Sp_{2g}(\mathbb{Z},d^2) \ar@{->}[r] & Sp_{2g}(\mathbb{Z},d) \ar@{->}[r]^{\alpha} & \mathfrak{sp}_{2g}(\mathbb{Z}/d) \ar@{->}[r] & 1 .}
\end{equation}

For almost all values of $d$ this exact sequence computes  the abelianization of the level $d$ congruence subgroup. To be more precise, in \cite{per}, \cite{Putman_abel}, \cite{sato_abel}, B.~Perron, A. Putman  and M. Sato independently proved that
for any $g\geq 3$ and an odd integer $d\geq 3,$
$$[Sp_{2g}(\mathbb{Z},d),Sp_{2g}(\mathbb{Z},d)]=Sp_{2g}(\mathbb{Z},d^2).$$
For $d$ even they show that the subgroup $[Sp_{2g}(\mathbb{Z},d),Sp_{2g}(\mathbb{Z},d)]$ coincides with the Igusa subgroup,
$$
	Sp_{2g}(\mathbb{Z},d,2d)=\left\lbrace A\in Sp_{2g}(\mathbb{Z},d) \; \Big|
	\begin{array}{l}  A=Id_{2g}+dA',\\
	  A'_{g+i,i}\equiv A'_{i,g+i} \equiv 0 \;(\text{mod } 2) \;\forall i \end{array}\right\rbrace,	
	$$
and by \cite[Lemma 1]{igusa}, the quotient of short exact sequence \eqref{ses_abel_d1} by this group gives another short exact sequence of $Sp_{2g}(\mathbb{Z})$-modules,
\begin{equation}
\label{eq:ses-Sp-d2}
\xymatrix@C=5mm@R=10mm{0 \ar@{->}[r] & H_1(\Sigma_{g,1};\mathbb{Z}/2) \ar@{->}[r] & H_1(Sp_{2g}(\mathbb{Z},d);\mathbb{Z}) \ar@{->}[r] & \mathfrak{sp}_{2g}(\mathbb{Z}/d) \ar@{->}[r] & 0 .}
\end{equation}

\subsection{Homological tools}
\label{subsec-hom-tools}
We now record for the convenience of the reader some of the homological tools that we will constantly use.

\subsubsection*{Hochshild-Serre spectral sequence} Let 
\[
\xymatrix{
1 \ar[r] & K \ar[r] & G \ar[r] & Q \ar[r] & 1
}
\]
be a group extension. Then, for any $G$-module $M$, there are two strongly convergent first-quadrant spectral sequences:

the homological spectral sequence
\[
E^{2}_{p,q} = H_p(Q;H_q(K,M)) \Rightarrow H_{p+q}(G;M),
\] 
and the cohomological spectral sequence
\[
E_{2}^{p,q} = H^p(Q;H^q(K,M)) \Rightarrow H^{p+q}(G;M).
\]

\subsubsection*{Exact sequences in low homological degree}

A basic application of the above spectral sequences is to produce  the classical $5$-term exact sequence associated to a group extension as above. For any $G$-module $M$ there are exact sequences:
\[
\xymatrix{
H_2(G;M) \ar[r] & H_2(Q,M_K) \ar[r] & H_1(K,M)_Q  \ar@{->} `r/8pt[d] `/10pt[l] `^dl[ll] `^r/3pt[dll] [dll]\\
  H_1(G;M) \ar[r] & H_1(Q;M_K) \ar[r] & 0
}
\]
and
\[
\xymatrix{0 \ar[r] & H^1(Q;M^K) \ar[r] &  H^1(G;M) \ar[r] & H^1(K;M)^Q  \ar@{->} `r/8pt[d] `/10pt[l] `^dl[lll] `^r/3pt[dll] [dll] \\
	 & H^2(Q;M^K) \ar[r] &H^2(G;M) &  }  
\]

We will need two variants of these sequences, one weaker form and one stronger form.
First recall that the morphism $H_1(K,M)_Q \rightarrow H_1(G;M)$ is induced by the inclusion $K \hookrightarrow G$, in particular we have an exact sequence:
\[
 H_1(K,M) \rightarrow H_1(G;M) \rightarrow H_1(Q;M_K) \rightarrow 0
\]
which we will refer to as the ``$3$-term exact sequence". The analogous exact sequence obviously exists in cohomology. These turn to be convenient when we apply further a right-exact functor like the coinvariants in the homological case and a left exact functor, like the invariants in the cohomological case.

A deeper result involves \emph{extending} the $5$-term exact sequences to the left for the homological one and to the right for the cohomological one. A lot of work has been put into this, and we will use here the extension by Dekimpe-Hartl-Wauters \cite{hartl} which states that given a group extension as above there is a functorial $7$-term exact sequence in cohomology:
 \[
 \xymatrix@C=7mm@R=5mm{
 0 \ar[r] & H^1(Q;M^K) \ar[r] & H^1(G;M) \ar[r] & H^1(K;M)^Q  \ar@{->} `r/8pt[d] `/10pt[l] `^dl[lll] `^r/3pt[dll] [dll] \\
 &   H^2(Q;M^K)  \ar[r] & H^2(G;M)_1 \ar@{->} `r/8pt[d] `/10pt[l] `^dl[ll] `^r/3pt[dll] [dl] &\\
  & H^1(Q;H^1(K;M)) \ar[r] & H^3(Q;M^K), &
}
 \]
where $H^2(G;M)_1$ denotes the kernel of the restriction map
$$H^2(G;M) \rightarrow H^2(K;M).$$

\subsubsection*{The Universal coefficient Theorem for p-elementary abelian groups, p odd.}
\label{sec-saunders}

Let $Q$ and $A$ be $p$-elementary abelian groups with $Q$ acting trivially on $A$, by the Universal Coefficient Theorem (from now on UCT for short) there is a split short exact sequence of $\mathbb{Z}/p$-modules:
\begin{equation}
\label{eq:UCT}
\xymatrix@C=5mm@R=13mm{ 0 \ar@{->}[r] & Ext^1_\mathbb{Z}(Q,A) \ar@{->}[r]^{i} & H^2(Q; A) \ar@{->}[r]^-{\theta} & Hom(\Lambda^2 Q;A) \ar@{->}[r] & 0 .}
\end{equation}

The map $\theta$ is a kind of antisymmetrization: it sends the class of the normalized $2$-cocycle $c$ to the alternating bilinear map
$$(g,h) \mapsto c(g,h) - c(h,g).$$

In \cite{McLane}, S. MacLane gave an effective calculation of
$Ext_{\mathbb{Z}}^1(Q;A)$ by constructing a  natural homomorphism
\begin{equation}
\label{iso-ext-hom}
\nu: H^2(Q;A)\rightarrow Hom(Q,A)
\end{equation}
whose restriction to $Ext_{\mathbb{Z}}^1(Q;A)$ is an isomorphism as follows. 
Given a central extension $c$
$$
\xymatrix@C=6mm@R=13mm{ 0 \ar@{->}[r] & A \ar@{->}[r]^-{j} & G \ar@{->}[r]^-{\pi} & Q \ar@{->}[r] & 0,}
$$
observe that for any lift $\tilde{x}$ of an element $x \in G$, $\tilde{x}^p \in A$, because $Q$ is a $p$-group.
The  homomorphism $\nu(c)\in Hom(Q,A)$ is the map that sends each element $x\in Q$ to $\widetilde{x}\;{}^p\in A$.

The inverse of $\nu$ restricted to $Ext_{\mathbb{Z}}^1(Q;A)$ is given by functoriality. Fix  an  abelian extension $d\in Ext_{\mathbb{Z}}^1(A;A)$ with $\nu(d)=id$. Then for a homomorphism $f\in Hom(Q,A),$ we have that
$f=f^*(id)=f^*\nu(d)=\nu(f^*d).$
Therefore the inverse of $\nu$ is given by $\nu^{-1}(f)=f^*d\in Ext_{\mathbb{Z}}^1(Q;A).$

In the sequel we exhibit an explicit splitting  of short exact sequence \eqref{eq:UCT} by giving a section of $\theta$ and a retraction of $i$.
Observe that any bilinear form $\beta$ on $Q$ with values in $A$ is a $2$-cocycle. In particular, if $\beta$ is antisymmetric, $\theta(\beta) = 2\beta$. As $p\neq 2$, $2$ is invertible in $\mathbb{Z}/p$, and this gives us a canonical section of $\theta$, by sending an alternating bilinear map $\beta$ to the extension defined by the $2$-cocycle $\frac{1}{2}\beta$. Then the retraction of $i$ is given by the map that sends the class of a $2$-cocycle $c$ to $\nu(c-\frac{1}{2}\theta(c))= \nu(c)-\frac{1}{2}\nu(\theta(c))$. In fact, since $\theta(c)$ is an antisymmetric bilinear form, we show that $\nu(\theta(c))$ is zero and therefore the retraction of $i$ coincides with $\nu$. To see that $\nu(\theta(c))$ is zero, consider the extension with associated $2$-cocycle $\theta(c)$:
$$
\xymatrix@C=6mm@R=13mm{ 0 \ar@{->}[r] & A \ar@{->}[r]^-{j} & A\times_{\theta(c)} Q \ar@{->}[r] & Q \ar@{->}[r] & 0.}
$$
Here we recall that $A\times_{\theta(c)} Q$ stands for the set $A\times Q$ with group law given by the product:
$$(a,g)(b,h)=(a+b+\theta(c)(g,h), g+h).$$

Then $\nu(\theta(c))$ is the homomorphism in $Hom(Q,A)$ which sends $x$ to $(0,x)^p$.
Knowing that $\theta(c)$ is bilinear and antisymmetric, we compute directly:
\[
 (0,x)^p=\Big(\sum_{i=1}^{p-1} \theta(c)(x,ix), 0\Big)=\Big(\sum_{i=1}^{p-1} i\theta(c)(x,x), 0\Big)=(0,0)=0.
\]

Summing up, for an odd prime $p$ and $A$, $Q$ elementary abelian $p$-groups, we have a natural isomorphism:
\begin{equation}
\label{mcLane-UCT-iso}
\theta\oplus \nu:\;H^2(Q;A)\rightarrow Hom(\Lambda^2 Q;A) \oplus Hom(Q,A).
\end{equation}

\subsubsection*{A tool to annihilate homology and cohomology groups.}

Finally, a number of our computations will rely on the next two classical results. The first one allows to switch from cohomology to homology  

\begin{lema}[\cite{brown}, Proposition VI.7.1.]
\label{lema_duality_homology}
Given a finite group $G$, an integer $d$ and $M$ a $\mathbb{Z}G/d$-module, denote by $M^*$ the dual module $Hom(M,\mathbb{Z}/d)$. Then, for every $k\geq 0$ there is a natural isomorphism
$$H^k(G;M^*)\simeq (H_k(G;M))^*.$$
\end{lema}

The second is a classical vanishing result:

\begin{lema}[Center kills lemma (\cite{dupont}, Lemma 5.4)]
\label{lem_cen_kill}
Given an arbitrary group $G$ and $M$ a (left) $RG$-module ($R$ any commutative ring), if there exists a central element $\gamma\in G$  such that for some $r\in R,$ $\gamma x=rx$ for all $x\in M.$
Then $(r-1)$ annihilates $H_*(G,M).$
\end{lema}

\section{Parametrization of mod $d$ homology $3$-spheres}\label{sec:parammoddhlgysphere}

We now turn to our first main goal: the parametrization of  $\mathbb{Z}/d$-homology spheres via the mod $d$ Torelli group. Given an integer $d\geq 2$, a Heegaard splitting with gluing map an element of the mod $d$ Torelli group is a 
$\mathbb{Z}/d$-homology sphere, and therefore a $\mathbb{Q}$-homology sphere.  Every $\mathbb{Q}$-homology sphere can be obtained as a Heegaard splitting with  gluing map an element of the mod $d$ Torelli group for an appropriate $d$. This last result was announced by B.~Perron in \cite[Prop. 6]{per}, where he stated that given a $\mathbb{Q}$-homology $3$-sphere $M$, if we set $n=|H_1(M;\mathbb{Z})|$, the cardinality of this homology group,  then $M$ can be obtained as a Heegaard splitting with gluing map an element of the mod $d$ Torelli group for any $d$ that divides $n-1$. We shall here provide  a proof of a slightly more general statement  in which we allow $d$ to divide either $n-1$ or $n+1$, and show that the divisibility condition is necessary.

If we specify these results to parameterize the set $\mathcal{S}^3(d)$ of $\mathbb{Z}/d$-homology spheres the situation is slightly more subtle. Let us denote by $\mathcal{S}^3[d]$ the set of those manifolds that admit a Heegaard splitting with gluing map an element of $\mathcal{M}_{g,1}[d]$ for some $g \geq 1$.  As we observed in Proposition~\ref{prop:homologySphmodd}, by definition of the mod $d$ Torelli groups $\mathcal{S}^3[d] \subset \mathcal{S}^3(d)$, but as we shall see equality seldom occurs.

Let us first determine some homological properties of the gluing maps of  $\mathbb{Q}$-homology spheres. The following result is the classic description of matrix equivalence classes in $M_n(\mathbb{Z})$, the group of $n\times n$ matrices with integral coefficients.

\begin{prop}[\cite{Gock}, Theorem 368]\label{prop:smith}
Let $A\in  M_{n}(\mathbb{Z})$ be given. Then there  exist matrices $U,V\in GL_n(\mathbb{Z})$ and a diagonal matrix $D\in M_{n}(\mathbb{Z})$ whose diagonal entries are integers $d_1,d_2,\ldots , d_r,0,\ldots , 0,$ with $d_i \neq 0$ and $d_i|d_{i+1}$ for $i=1,2,\ldots , r-1,$ such that
\[
A = UDV.
\]
Moreover the matrix $D$ is unique up to the sign of each entry; it is called the Smith normal form of $A$ and the integers  $d_1,d_2,\ldots, d_r$ are called the \textit{invariant factors} of $A$.
\end{prop}

As an immediate application we have:
\begin{lema}\label{lema_n_detH}
Let $\phi \in \mathcal{M}_{g,1}$ and assume $S^3_f=\mathcal{H}_g\cup_{\iota_g \phi} -\mathcal{H}_g$ is a $\mathbb{Q}$-homology sphere. Let $n$ denote the cardinal of $H_1(S^3_\phi;\mathbb{Z})$ and write
$H_1(\phi;\mathbb{Z})=\left(\begin{smallmatrix}
E & F \\
G & H
\end{smallmatrix}\right).$ Then $n=|\det(H)|.$
\end{lema}
\begin{proof}
By the results in Section \ref{subsec:computeH1(M)} we  know that
$H_1(S^3_\phi;\mathbb{Z})\simeq \coker(H)$ and by Theorem~\ref{prop:smith} above, we have matrices $U,V\in GL_{g}(\mathbb{Z})$ such that
$$UHV=\left(\begin{smallmatrix}
d_1 & & & & & \\
 & \ddots & & & & \\
 & & d_k & & & \\
 & & & 0 & & \\
 & & & & \ddots & \\
 & & & & & 0
 \end{smallmatrix}\right).$$
As a consequence,
\begin{equation*}
\coker(H)\simeq \mathbb{Z}/d_1\times \mathbb{Z}/d_2\times \cdots \mathbb{Z}/d_k\times \mathbb{Z}^{g-k}.
\end{equation*}
Since $S^3_\phi$ is a $\mathbb{Q}$-homology sphere, by the UCT, we have that
$$0=H_1(S^3_\phi;\mathbb{Q})\simeq H_1(S^3_\phi;\mathbb{Z})\otimes \mathbb{Q}\simeq \mathbb{Q}^{g-k},\quad  \text{so} \quad g=k.$$
Then
\begin{equation*}
\det(H)=\pm \det(UHV)=\pm \prod_{i=1}^{g}d_i\neq 0,
\end{equation*}
and indeed
\begin{equation*}
n=|H_1(S^3_\phi;\mathbb{Z})|= \Big|\prod_{i=1}^{g}d_i \Big| \neq 0.
\end{equation*}
\end{proof}

We now turn to our first result showing the relation between $\mathbb{Q}$-homology spheres and elements by the mod $d$ Torelli groups, it is inspired by \cite[Prop. 6]{per} due to B.~Perron. We remind the reader that a $\mathbb{Q}$-homology sphere has finite first integral homology group.
To avoid unnecessarily cumbersome notation, in what follows we denote by $Sp_{2g}^{A\pm}(\mathbb{Z}/d)$ (resp. $Sp_{2g}^{B\pm}(\mathbb{Z}/d)$) the image of $\mathcal{A}_{g,1}$ (resp. $\mathcal{B}_{g,1}$) in $Sp_{2g}(\mathbb{Z}/d).$

\begin{teo}\label{teo_rat_homology_gen}
Let $M$ be a $\mathbb{Q}$-homology sphere and $n=|H_1(M;\mathbb{Z})|$, the cardinal of this finite homology group.
Then $M\in \mathcal{S}^3[d]$ for some $d \geq 2$
if and only if $d$ divides either $n-1$ or $n+1.$
\end{teo}
\begin{proof}
We use the notation of Lemma \ref{lema_n_detH} above. Given $d\geq 2,$ assume that $M\in \mathcal{S}^3[d].$ Fix an element $\phi \in \mathcal{M}_{g,1}[d]$ such that $M\simeq\mathcal{H}_g\cup_{\iota_g\phi} -\mathcal{H}_g.$
Since $\phi\in \mathcal{M}_{g,1}[d],$ we know that $H_1(\phi;\mathbb{Z})=
\left(\begin{smallmatrix}
E & F \\
G & H
\end{smallmatrix}\right)\in Sp_{2g}(\mathbb{Z},d)$. Then $\det H = 1 \;(\text{mod } d)$ and by Lemma \ref{lema_n_detH} we deduce that $n\equiv \pm 1 \; (\text{mod } d)$.

For the converse, assume that $d\geq 2$ divides either $n-1$ or $n+1$.

By Theorem \ref{bij_MCG_3man}, there exists an element $\phi \in\mathcal{M}_{g,1}$ such that $M$ is homeomorphic to $\mathcal{H}_g\cup_{\iota_g\phi} -\mathcal{H}_g.$
By definition of $\mathcal{M}_{g,1}[d]$ we have a short exact sequence
\begin{equation}
\label{ses_def_MCGmodp}
\xymatrix@C=7mm@R=10mm{1 \ar@{->}[r] & \mathcal{M}_{g,1}[d] \ar@{->}[r] & \mathcal{M}_{g,1} \ar@{->}[r] & Sp_{2g}(\mathbb{Z}/d) \ar@{->}[r] & 1 .}
\end{equation}
Thus it is enough to show that there exist matrices $X\in Sp_{2g}^{A\pm}(\mathbb{Z}/d)$ and $Y\in Sp_{2g}^{B\pm}(\mathbb{Z}/d),$
such that
\begin{equation}
\label{eq_XY}
XH_1(\phi;\mathbb{Z}/d)Y=Id.
\end{equation}
Then, by Lemma~\ref{lem_B[p]}, there are elements $\xi_a\in \mathcal{A}_{g,1},$ $\xi_b\in\mathcal{B}_{g,1}$ such that $H_1(\xi_a;\mathbb{Z}/d)=X,$ $H_1(\xi_b;\mathbb{Z}/d)=Y$, and as a consequence,
\[
H_1(\xi_a \phi \xi_b;\mathbb{Z}/d)=Id,
\]
i.e. $\phi$ is equivalent to $\xi_a \phi\xi_b \in \mathcal{M}_{g,1}[d]$.

We proceed to construct the  matrices $X,Y$.
First  notice that by our hypothesis and Lemma~\ref{lema_n_detH}, $H\in SL_g^\pm(\mathbb{Z}/d).$
Consider the matrix
$$X=\left(\begin{matrix}
Id & A \\
0 & Id
\end{matrix}\right)\quad \text{with}, \; A=-FH^{-1}.$$
Since $H_1(\phi;\mathbb{Z}/d)\in Sp_{2g}(\mathbb{Z}/d),$ $({}^tH)F$ is symmetric and:
\begin{align*}
A= & -FH^{-1}=-({}^tH^{-1})({}^tH)FH^{-1}=-({}^tH^{-1})({}^tF)HH^{-1} = \\
= & -({}^tH^{-1})({}^tF)=-{}^t(FH^{-1})={}^tA.
\end{align*}
Therefore $X\in Sp_{2g}^{A\pm}(\mathbb{Z}/d).$
Finally, notice that
\begin{align*}
\left(\begin{matrix}
Id & A \\
0 & Id
\end{matrix}\right)
\left(\begin{matrix}
E & F \\
G & H
\end{matrix}\right) & =
\left(\begin{matrix}
E+AG & F+AH \\
G & H
\end{matrix}\right) \\
& = \left(\begin{matrix}
E+AG & 0 \\
G & H
\end{matrix}\right)\in Sp_{2g}^{B\pm}(\mathbb{Z}/d).
\end{align*}
Hence, setting
$$Y=\left(\begin{matrix}
E+AG & 0 \\
G & H
\end{matrix}\right)^{-1}\in Sp_{2g}^{B\pm}(\mathbb{Z}/d),
$$
we get the desired result.
\end{proof}

Letting $d$ vary along all prime numbers we deduce:

\begin{cor}\label{cor:mg1dcubreSQ}
The following equality holds:
\[
\mathcal{S}^3_\mathbb{Q} = \bigcup_{p \text{ prime}} \mathcal{S}^3[p].
\]
\end{cor}

We now turn more precisely to the inclusion $\mathcal{S}^3[d] \subseteq \mathcal{S}^3(d)$, and study the difference between being a $\mathbb{Z}/d$-homology sphere and being built out of a map in the mod $d$-Torelli group. 

\begin{teo}\label{teo:counterexample_M[d]}
The sets $\mathcal{S}^3[d]$ and $\mathcal{S}^3(d)$ coincide if and only if $d=2,3,4,6.$
\end{teo}
\begin{proof}
Let $M$ be a $\mathbb{Q}$-homology sphere with $n=|H_1(M;\mathbb{Z})|$ and $d$ a positive integer. By definition of $\mathcal{S}^3(d)$ and $\mathcal{S}^3[d],$
\begin{align*}
M\in \mathcal{S}^3(d) & \Leftrightarrow gcd(n,d)=1 \Leftrightarrow n\in \left( \mathbb{Z}/d\right)^\times, \\
M\in \mathcal{S}^3[d] & \Leftrightarrow n\equiv \pm 1 \; (\text{mod }d).
\end{align*}
As usual $\left( \mathbb{Z}/d\right)^\times$ stands for the group of units in $\mathbb{Z}/d$.

Then $\mathcal{S}^3(d),$ $\mathcal{S}^3[d]$ certainly coincide if $|\left(\mathbb{Z}/d\right)^\times |\leq 2.$ By the Chinese Reminder Theorem,
this occurs if and only if $d=2,3,4,6.$

Finally, we show that for a fixed $d\neq 2,3,4,6,$ the sets $\mathcal{S}^3(d),$ $\mathcal{S}^3[d]$ do not coincide by providing an explicit $\mathbb{Q}$-homology sphere in $ \mathcal{S}^3(d)$ that does not belong to $\mathcal{S}^3[d]$.

Take an element $u\in (\mathbb{Z}/d)^\times$ with $u\neq \pm 1.$
Consider the matrix
$\left(\begin{smallmatrix}
{}^tG^{-1} & 0 \\
0 & G
\end{smallmatrix}\right)\in Sp_{2g}(\mathbb{Z}/d),$
with $G\in GL_g(\mathbb{Z}/d)$ given by
$$G=\left(\begin{matrix}
u & 0 \\
0 & Id
\end{matrix}\right).$$
By surjectivity of the symplectic representation modulo $d$,  there exists an element $\phi \in \mathcal{M}_{g,1}$ such that
$H_1(\phi;\mathbb{Z}/d)=\left(\begin{smallmatrix}
{}^tG^{-1} & 0 \\
0 & G
\end{smallmatrix}\right).$
Then the $3$-manifold $S^3_\phi = \mathcal{H}_g\cup_{\iota_g\phi} -\mathcal{H}_g$ provides the desired example, since
by Lemma~\ref{lema_n_detH}, $|H_1(S^3_\phi;\mathbb{Z})|=|det(G)|=u\neq \pm 1 \; (\text{mod } d)$.
\end{proof}

\subsection{A convenient parametrization of $\mathbb{Z}/d$-homology spheres}\label{subsec:paramhlgyspheres}

Let us denote by $\mathcal{A}_{g,1}\backslash\mathcal{M}_{g,1}[d]/\mathcal{B}_{g,1}$ the image of the canonical map 
\[
\mathcal{M}_{g,1}[d] \longrightarrow \mathcal{A}_{g,1}\backslash\mathcal{M}_{g,1}/\mathcal{B}_{g,1}.
\]
Then by Singer's Theorem~\ref{bij_MCG_3man} and the definition of $\mathcal{M}_{g,1}[d]$, we have a bijection:
$$
\begin{array}{ccl}
	{\displaystyle
\lim_{g\to \infty}\mathcal{A}_{g,1}\backslash\mathcal{M}_{g,1}[d]/\mathcal{B}_{g,1}} & \longrightarrow &\mathcal{S}^3[d], \\
\phi & \longmapsto & S^3_\phi=\mathcal{H}_g \cup_{\iota_g\phi} -\mathcal{H}_g.
\end{array}
$$

From a group-theoretical point of view, the induced equivalence relation on $\mathcal{M}_{g,1}[d],$ which is given by:
\begin{equation}
\phi \sim \psi \quad\Leftrightarrow \quad \exists \zeta_a \in \mathcal{A}_{g,1}\;\exists \zeta_b \in \mathcal{B}_{g,1} \quad \text{such that} \quad \zeta_a \phi \zeta_b=\psi,
\end{equation}
is quite unsatisfactory, as it is not internal to the group $\mathcal{M}_{g,1}[d]$. However,
as for integral homology spheres (see \cite[Lemma 4]{pitsch}), we can rewrite this equivalence relation as follows:
\begin{lema}
\label{lema_eqiv_coset[d]}
Two maps $\phi, \psi \in \mathcal{M}_{g,1}[d]$ are equivalent if and only if there exists a map $\mu \in \mathcal{AB}_{g,1}$ and two maps $\xi_a\in \mathcal{A}_{g,1}[d]$ and $\xi_b\in\mathcal{B}_{g,1}[d]$ such that $\phi=\mu \xi_a\psi \xi_b\mu^{-1}.$
\end{lema}
\begin{proof}
The ''if'' part of the Lemma is trivial. Conversely, assume that $\psi=\xi_a \phi \xi_b,$ where $\psi, \phi \in \mathcal{M}_{g,1}[d]$, $\xi_a\in \mathcal{A}_{g,1}$ and $\xi_b\in \mathcal{B}_{g,1}$ Applying the symplectic representation modulo $d$ to this equality we get
\[
Id=H_1(\xi_a;\mathbb{Z})H_1(\xi_b;\mathbb{Z}) \quad (\text{mod } d).
\]
By Section \ref{homol_homoto_actions} ,
\begin{equation*}
H_1(\xi_b;\mathbb{Z})=\left(\begin{matrix}
G & 0 \\
M & {^t}G^{-1}
\end{matrix}\right),
\qquad
H_1(\xi_a;\mathbb{Z})=\left(\begin{matrix}
H & N \\
0 & {^t}H^{-1}
\end{matrix}\right),
\end{equation*}
for some matrices $G,H\in GL_g(\mathbb{Z})$ and ${}^tGM,H^{-1}N\in Sym_g(\mathbb{Z}),$ and hence:
\begin{align*}
\left(\begin{matrix}
Id & 0 \\
0 & Id
\end{matrix}\right) & = \left(\begin{matrix}
H & N \\
0 & {^t}H^{-1}
\end{matrix}\right)
\left(\begin{matrix}
G & 0 \\
M & {^t}G^{-1}
\end{matrix}\right) (\text{mod } d) \\
& =
\left(\begin{matrix}
HG+NM & N{}^tG^{-1} \\
{}^tH^{-1}M & {^t}(HG)^{-1}
\end{matrix}\right) \; (\text{mod } d).
\end{align*}
In particular, $N = 0 = M\; (\text{mod } d)$ and $G = H^{-1}\; (\text{mod } d).$

Once again, by Section~\ref{homol_homoto_actions}, we can choose a map $\mu \in \mathcal{AB}_{g,1}$ such that
\[
H_1(\mu;\mathbb{Z})=\left(\begin{matrix}
H & 0 \\
0 & {^t}H^{-1}
\end{matrix}\right).
\]
Since $N = 0 = M\; (\text{mod } d)$ and $G = H^{-1}\; (\text{mod } d),$ the following equalities hold:
\begin{align*}
H_1(\mu^{-1}\xi_a;\mathbb{Z})& = \left(\begin{matrix}
H^{-1} & 0 \\
0 & {^t}H
\end{matrix}\right)
\left(\begin{matrix}
H & N \\
0 & {^t}H^{-1}
\end{matrix}\right) \\
& =\left(\begin{matrix}
Id & H^{-1}N \\
0 & Id
\end{matrix}\right) = Id \; (\text{mod } d). \\
H_1(\xi_b\mu;\mathbb{Z})& = 
\left(\begin{matrix}
G & 0 \\
M & {^t}G^{-1}
\end{matrix}\right)
\left(\begin{matrix}
H & 0 \\
0 & {^t}H^{-1}
\end{matrix}\right) \\
&=\left(\begin{matrix}
GH & 0 \\
MH & {^t}(GH)^{-1}
\end{matrix}\right) = Id \; (\text{mod } d).
\end{align*}
Therefore,
$$\psi=\mu (\mu^{-1}\xi_a)\phi(\xi_b \mu)\mu^{-1},$$
where $(\mu^{-1}\xi_a)\in \mathcal{A}_{g,1}[d],$ $(\xi_b\mu)\in \mathcal{B}_{g,1}[d]$ and $\mu\in \mathcal{AB}_{g,1}.$
\end{proof}

For future reference we record:

\begin{teo}\label{teo:bij_M[p]}
There is a bijection:
$$
\begin{array}{ccl}
	{\displaystyle\lim_{g\to \infty}\left(\mathcal{A}_{g,1}[d]\backslash\mathcal{M}_{g,1}[d]/\mathcal{B}_{g,1}[d]\right)_{\mathcal{AB}_{g,1}} } & \longrightarrow & \mathcal{S}^3[d] \\
	\phi & \longmapsto & S^3_\phi=\mathcal{H}_g \cup_{\iota_g\phi} -\mathcal{H}_g.
\end{array}
$$
\end{teo}

\section{Invariants and trivial 2-cocycles}

In this section, expanding on the work in~\cite{pitsch}, we relate the existence of invariants of elements in $\mathcal{S}^3[d]$ to the cohomological properties of the groups $\mathcal{M}_{g,1}[d]$. In~\cite{pitsch} this takes the form of a one-to-one correspondence between certain types of $2$-cocycles on the stabilized Torelli group and invariants. In contrast, here there is more than one invariant associated to a given convenient $2$-cocycle, for there are homomorphisms on these groups that are invariants. This is akin to the Rohlin invariant for integral homology spheres, which is a $\mathbb{Z}/2\mathbb{Z}$-valued invariant and a homomorphism when viewed as a function on the Torelli groups (cf.\cite{riba1}).

\subsection{From invariants to trivial cocycles}\label{subsec:From invariants to trivial cocycles}

Let $A$ be an abelian group. Consider a normalized $A$-valued invariant:
\[
F: \mathcal{S}^3[d] \rightarrow A,
\]
that sends $\mathbf{S}^3$ to $0.$
Precomposing $F$  with the canonical maps 
\[
\mathcal{M}_{g,1}[d]\longrightarrow \lim_{g\rightarrow \infty} \faktor{\mathcal{M}_{g,1}[d]}{\sim} \longrightarrow  \mathcal{S}^3[d]
\]
 we get a family of maps  $F_g:\mathcal{M}_{g,1}[d]\rightarrow A$ satisfying the following properties:
\begin{enumerate}[i)]
\item[0)] $F_g(Id)=0,$ 
\item $F_{g+1}(x)=F_g(x) \quad \text{for every }x\in \mathcal{M}_{g,1}[d],$
\item $F_g(\phi x \phi^{-1})=F_g(x)  \quad \text{for every }  x\in \mathcal{M}_{g,1}[d], \; \phi\in \mathcal{AB}_{g,1},$
\item $F_g(\xi_a x\xi_b)=F_g(x) \quad \text{for every } x\in \mathcal{M}_{g,1}[d],\;\xi_a\in \mathcal{A}_{g,1}[d],\;\xi_b\in \mathcal{B}_{g,1}[d].$
\end{enumerate}

Condition $0)$ is simply the normalization condition, it amounts to requiring that $F(\mathbf{S}^3)=0$, which in any case is a mild assumption.
Because of property $i)$, without loss of generality we can assume $g\geq 4$, this avoids having to deal with some peculiarities in the homology of low genus mapping class groups. 

To this family of  functions we can associate a family of   trivial $2$-cocycles:
\begin{align*}
C_g: \mathcal{M}_{g,1}[d]\times \mathcal{M}_{g,1}[d] & \longrightarrow A, \\
 (\phi,\psi) & \longmapsto F_g(\phi)+F_g(\psi)-F_g(\phi\psi),
\end{align*}
which measure the failure of $F_g$ to be a homomorphism. The sequence of $2$-cocycles $(C_g)_{g\geq 4}$ inherits the following properties:
\begin{enumerate}[(1)]
\item The $2$-cocycles $(C_g)_{g\geq 4}$ are compatible with the stabilization map, in other words, for $g\geq 4$ there is a commutative triangle:
\[
\xymatrix{
	\mathcal{M}_{g,1}[d] \times \mathcal{M}_{g,1}[d] \ar@{->}[r]   \ar[dr]_-{C_{g}} & \mathcal{M}_{g+1,1}[d] \times \mathcal{M}_{g+1,1}[d] \ar[d]^-{C_{g+1}} \\
	 & A}
\]
\item The $2$-cocycles $(C_g)_{g\geq 4}$ are invariant under conjugation by elements in $\mathcal{AB}_{g,1},$
\item If either $\phi\in \mathcal{A}_{g,1}[d]$ or $\psi \in \mathcal{B}_{g,1}[d]$ then $C_g(\phi, \psi)=0.$
\end{enumerate}

Given two sequences of maps $(F_g)_{g\geq 4},$ $(F_g')_{g\geq 4}$ that satisfy conditions 0) - iii) and induce the same sequence of trivial 2-cocycles $(C_g)_{g\geq 4},$ we have that
$(F_g-F_g')_{g\geq 4}$ is a sequence of homomorphisms satisfying the same conditions. As a consequence, the number of sequences $(F_g)_{g\geq 4}$ satisfying conditions 0) - iii) that induce the same sequence of trivial $2$-cocycles $(C_g)_{g\geq 4},$ coincides with the number of homomorphisms in $Hom(\mathcal{M}_{g,1}[d],A)^{\mathcal{AB}_{g,1}}$ compatible with the stabilization map that are zero when restricted to both $\mathcal{A}_{g,1}[d]$ and $\mathcal{B}_{g,1}[d]$.
We devote the rest of this section to compute and study such homomorphisms.

Observe first that since the group $A$ is abelian and has a trivial action by the mapping class group, we have isomorphisms:
\begin{align*}
Hom(\mathcal{M}_{g,1}[d],A)^{\mathcal{AB}_{g,1}} & =  Hom(H_1(\mathcal{M}_{g,1}[d];\mathbb{Z})_{\mathcal{AB}_{g,1}}, A),  \\
Hom(\mathcal{A}_{g,1}[d],A)^{\mathcal{AB}_{g,1}} & = Hom(H_1(\mathcal{A}_{g,1}[d];\mathbb{Z})_{\mathcal{AB}_{g,1}}, A), \\
Hom(\mathcal{B}_{g,1}[d],A)^{\mathcal{AB}_{g,1}}&=  Hom(H_1(\mathcal{B}_{g,1}[d];\mathbb{Z})_{\mathcal{AB}_{g,1}}, A).
\end{align*}
Since the self-conjugation action of a group  induces the identity in homology, the $\mathcal{AB}_{g,1}$-coinvariants on the right hand side of the above equalities are by Lemma~\ref{lem_B[p]} in fact computed with respect to: 
\[
\faktor{\mathcal{AB}_{g,1}}{\mathcal{AB}_{g,1}[d]} \simeq SL_g^\pm(\mathbb{Z}/d\mathbb{Z}).
\]
Similarly, the $\mathcal{AB}_{g,1}$-coinvariants of $H_1(\mathcal{T}_{g,1},\mathbb{Z}),$ $H_1(\mathcal{TA}_{g,1},\mathbb{Z}),$ $H_1(\mathcal{TB}_{g,1},\mathbb{Z})$ will be computed with respect to:
\[
\faktor{\mathcal{AB}_{g,1}}{\mathcal{TAB}_{g,1}} \simeq GL_g(\mathbb{Z}).
\]

To summarize, it is enough to understand the three groups: 
\[
H_1(\mathcal{M}_{g,1}[d];\mathbb{Z})_{SL_g^\pm(\mathbb{Z}/d\mathbb{Z})} \quad H_1(\mathcal{A}_{g,1}[d];\mathbb{Z})_{SL_g^\pm(\mathbb{Z}/d\mathbb{Z})} \quad H_1(\mathcal{B}_{g,1}[d];\mathbb{Z})_{SL_g^\pm(\mathbb{Z}/d\mathbb{Z})}
\]

We first deal with the two groups on the right:

\begin{prop}\label{prop:H1BpAp}
For $d \geq  3$ and $g\geq 4$, the following groups are zero:
$$
H_1(\mathcal{A}_{g,1}[d];\mathbb{Z})_{SL_g^\pm(\mathbb{Z}/d\mathbb{Z})},\qquad H_1(\mathcal{B}_{g,1}[d];\mathbb{Z})_{SL_g^\pm(\mathbb{Z}/d\mathbb{Z})}.
$$
For $d=2$ and $g\geq 4$, the aforementioned groups are isomorphic to $\mathbb{Z}/2$.
\end{prop}

\begin{proof} Both cases $d \geq 3$ and $d=2$ are based on the same argument, and we only prove the result for $\mathcal{B}_{g,1}[d],$ since the cases for the other group are similar. Consider the short exact sequence of groups
	$$\xymatrix@C=7mm@R=3mm{1 \ar@{->}[r] & \mathcal{TB}_{g,1} \ar@{->}[r] & \mathcal{B}_{g,1}[d] \ar@{->}[r]  & Sp_{2g}^B(\mathbb{Z},d) \ar@{->}[r] & 1.}$$
Taking $\mathcal{AB}_{g,1}$-coinvariants on the 3-term exact sequence, we get another exact sequence,

\[
\begin{tikzcd}
	H_1(\mathcal{TB}_{g,1};\mathbb{Z})_{GL_g(\mathbb{Z})} \arrow[r] &  H_1(\mathcal{B}_{g,1}[d];\mathbb{Z})_{SL_g^\pm(\mathbb{Z}/d\mathbb{Z})} \arrow[out=350, in=175, d]  &\\ &    H_1(Sp_{2g}^B(\mathbb{Z},d);\mathbb{Z})_{GL_g(\mathbb{Z})}  \arrow[r] & 0. 
\end{tikzcd}
\]

	By \cite[Lemma 4.2]{riba2}, for $g\geq 4$, the first group of this sequence is zero. Then we conclude by Proposition \ref{prop:com_SpB}.
\end{proof}

We turn to  the computation of $H_1(\mathcal{M}_{g,1}[d];\mathbb{Z})_{GL_g(\mathbb{Z})}$ and show:

\begin{prop}\label{prop:coinvablelevdmap}
Given integers $g\geq 4$ and $d\geq 2$ such that $4 \nmid d$, the trace map on the level $d$ symplectic group induces an isomorphism:
\[
H_1(\mathcal{M}_{g,1}[d];\mathbb{Z})_{SL_g^\pm(\mathbb{Z}/d\mathbb{Z})} \simeq \mathbb{Z}/d.
\]
\end{prop}
\begin{proof}
We will show in Proposition~\ref{prop_iso_M[d],Sp[d]} that the symplectic representation induces an isomorphism
\[
H_1(\mathcal{M}_{g,1}[d];\mathbb{Z})_{SL_g^\pm(\mathbb{Z}/d\mathbb{Z})} \simeq H_1(Sp_{2g}(\mathbb{Z},d);\mathbb{Z})_{GL_g(\mathbb{Z})}
\]
and conclude by applying Proposition \ref{prop_iso_M[d],sp(Z/d)}.
\end{proof}

Before we proceed to show the isomorphism mentioned in the proof of Proposition~\ref{prop:coinvablelevdmap} we need two preliminary results. Let $\mathfrak{B}_g$ denote the Boolean algebra generated by $H_2 = H_1(\Sigma_{g,1};\mathbb{Z}/2)$ and $\mathfrak{B}_g^n$ the subspace formed by the elements of degree at most $n$ (cf. \cite{jon_3}). 

\begin{lema}[\cite{riba2}, Proposition 4.1.]
\label{lem:GlcoinH1Torelli}

For $g \geq 4$ there is an isomorphism
\[
H_1(\mathcal{T}_{g,1}; \mathbb{Z})_{GL_g(\mathbb{Z})} \simeq \mathbb{Z}/2,
\]
where $\mathbb{Z}/2$ is generated by the class of the element $1 \in \mathfrak{B}^2_g$.
\end{lema}

We now show a slight tweaking of \cite[Prop.~0.5]{sato_abel}, which was also shown in \cite[Theorem H]{Putman} for $g\geq 5$.
\begin{prop}\label{prop:homologyleveld}
For $g\geq 3$ and $d$ even with $4\nmid d,$ there is an exact sequence
\[
\begin{tikzcd}
	0 \arrow[r] & \mathbb{Z}/2 \arrow[r] & H_1(\mathcal{T}_{g,1};\mathbb{Z})_{Sp_{2g}(\mathbb{Z},d)} \arrow[r, "j"] & H_1(\mathcal{M}_{g,1}[d];\mathbb{Z}) \arrow[overlay,out=350,in=170, dl] \\
	& & H_1(Sp_{2g}(\mathbb{Z},d);\mathbb{Z}) \arrow[r] & 0 . 
\end{tikzcd}
\]
\end{prop}
\begin{proof}
	
	Following Sato's arguments in ~\cite{sato_abel}, the  inclusion $\mathcal{M}_{g,1}[d] \hookrightarrow \mathcal{M}_{g,1}[2]$ fits into a commutative diagram with exact rows:
\begin{equation*}
\xymatrix@C=7mm@R=10mm{
 0 \ar@{->}[r] & \mathcal{T}_{g,1}  \ar@{->}[r] \ar@{=}[d] & \mathcal{M}_{g,1}[d] \ar@{->}[d] \ar@{->}[r] & Sp_{2g}(\mathbb{Z},d) \ar@{->}[d] \ar@{->}[r] & 0 \\
0 \ar@{->}[r] & \mathcal{T}_{g,1}  \ar@{->}[r] & \mathcal{M}_{g,1}[2] \ar@{->}[r] & Sp_{2g}(\mathbb{Z},2) \ar@{->}[r] & 0. }
\end{equation*}
By Proposition \ref{exh_Spd2}, this diagram induces a commutative ladder with exact rows:
\[
\xymatrix@C=7mm@R=10mm{
H_2(\mathcal{M}_{g,1}[d];\mathbb{Z}) \ar[r]  \ar[d] & H_2(Sp_{2g}(\mathbb{Z},d);\mathbb{Z}) \ar@{->>}[d]\ar@{->}[r] & H_1(\mathcal{T}_{g,1};\mathbb{Z})_{Sp_{2g}(\mathbb{Z},d)} \ar@{->>}[d] \ar@{-}[r] &\\
H_2(\mathcal{M}_{g,1}[2];\mathbb{Z}) \ar[r] & H_2(Sp_{2g}(\mathbb{Z},2);\mathbb{Z}) \ar@{->}[r] & H_1(\mathcal{T}_{g,1};\mathbb{Z})_{Sp_{2g}(\mathbb{Z},2)} \ar@{-}[r] &\\
\ar@{->}[r]^-j & H_1(\mathcal{M}_{g,1}[d];\mathbb{Z}) \ar@{->}[d] \ar@{->}[r] & H_1(Sp_{2g}(\mathbb{Z},d);\mathbb{Z}) \ar@{->>}[d]\ar@{->}[r] & 0 \\
\ar@{->}[r]   & H_1(\mathcal{M}_{g,1}[2];\mathbb{Z}) \ar@{->}[r] & H_1(Sp_{2g}(\mathbb{Z},2);\mathbb{Z}) \ar@{->}[r] & 0 .}
\]
By \cite[Prop. 0.5]{sato_abel} the kernel of the map $j$ is at most $\mathbb{Z}/2$. Therefore it is enough to show that $j$ is not injective.

Suppose that the map $j$ is injective. Then by exactness in the above commutative diagram the map $H_2(\mathcal{M}_{g,1}[d];\mathbb{Z}) \rightarrow H_2(Sp_{2g}(\mathbb{Z},d);\mathbb{Z})$ is surjective and by commutativity the map $H_2(\mathcal{M}_{g,1}[2];\mathbb{Z}) \rightarrow H_2(Sp_{2g}(\mathbb{Z},2);\mathbb{Z})$
is surjective too. But this implies that the map $H_1(\mathcal{T}_{g,1};\mathbb{Z})_{Sp_{2g}(\mathbb{Z},2)} \rightarrow H_1(\mathcal{M}_{g,1}[2];\mathbb{Z})$ is injective, which contradicts \cite[Prop. 0.5]{sato_abel}. 

\end{proof}

We now finish the proof of Proposition \ref{prop:coinvablelevdmap}, by proving: 

\begin{prop}\label{prop_iso_M[d],Sp[d]}
	Given integers $g\geq 4$ and $d\geq 2$ such that $4 \nmid d$, the symplectic representation $\mathcal{M}_{g,1}[d]\rightarrow Sp_{2g}(\mathbb{Z},d)$ induces an isomorphism
	$$H_1(\mathcal{M}_{g,1}[d];\mathbb{Z})_{GL_g(\mathbb{Z})}\simeq H_1(Sp_{2g}(\mathbb{Z},d);\mathbb{Z})_{GL_g(\mathbb{Z})}.$$
\end{prop}
\begin{proof}
	Restricting the symplectic representation of the mapping class group to $\mathcal{M}_{g,1}[d]$ we have a short exact sequence of groups with a compatible action of the mapping class group on the three terms, on the first two by conjugation and on the third via the symplectic representation,
	\[
	\xymatrix{1 \ar@{->}[r] & \mathcal{T}_{g,1} \ar@{->}[r] & \mathcal{M}_{g,1}[d] \ar@{->}[r] & Sp_{2g}(\mathbb{Z},d) \ar@{->}[r] & 1 .}
	\]
	The 3-term exact sequence in homology gives us:
	\[
	\xymatrix{ H_1(\mathcal{T}_{g,1};\mathbb{Z}) \ar@{->}[r]^-{j} & H_1(\mathcal{M}_{g,1}[d];\mathbb{Z}) \ar@{->}[r] & H_1(Sp_{2g}(\mathbb{Z},d);\mathbb{Z}) \ar@{->}[r] & 0 .}
	\]
	Taking $GL_{g}(\mathbb{Z})$-coinvariants, since the action on the level $d$ congruence subgroup factors through the symplectic group and because this is a right-exact functor we get another exact sequence:
	\[
	\begin{tikzcd}
		H_1(\mathcal{T}_{g,1};\mathbb{Z})_{GL_g(\mathbb{Z})}  \arrow[r,"j"] & H_1(\mathcal{M}_{g,1}[d];\mathbb{Z})_{GL_g(\mathbb{Z})} \ar[overlay,out=350,in=170,dl]  \\  H_1(Sp_{2g}(\mathbb{Z},d);\mathbb{Z})_{GL_g(\mathbb{Z})} \arrow[r] & 1 ,
	\end{tikzcd}
\]
which for $g\geq 4$, by Lemma \ref{lem:GlcoinH1Torelli}, becomes:
\[
\xymatrix{
\mathbb{Z}/2 \ar[r]^-{j} & H_1(\mathcal{M}_{g,1}[d];\mathbb{Z})_{GL_g(\mathbb{Z})} \ar[r] &  H_1(Sp_{2g}(\mathbb{Z},d);\mathbb{Z})_{GL_g(\mathbb{Z})}
\ar[r] & 1 . 
}
\]
Finally, for $d\geq 3$ an odd integer,
the $5$-term exact sequence tells us that the map $j$ factors through $H_1(\mathcal{T}_{g,1};\mathbb{Z})_{Sp_{2g}(\mathbb{Z},d)}$ and by \cite[Prop. 6.6]{Putman}, the element $1 \in \mathfrak{B}^2_g$ is $0$ in this group. 
For $d\geq 2$ an even integer such that $4 \nmid d$, Proposition~\ref{prop:homologyleveld} asserts in particular that $1 \in \ker j$.

\end{proof}

Given our abelian group $A$, denote by $A_d$ the subgroup of elements of exponent $d$. For any element $x \in A_d$ let $\varepsilon^x:\mathbb{Z}/d\longrightarrow A_d; \; 1\mapsto x$, be the map that picks the element $x$. This is trivially an $\mathcal{AB}_{g,1}$-invariant homomorphism.  By  Proposition~\ref{prop:coinvablelevdmap} we get:

\begin{prop}\label{prop:abinv_modp}
Given integers $g\geq 4$ and $d\geq 2$ such that $4 \nmid d$, the $\mathcal{AB}_{g,1}$-invariant homomorphisms on the level-$p$ mapping class group are generated by the multiples of the trace map on the $\mathfrak{gl}_g(\mathbb{Z}/d)$-block in the decomposition of $\mathfrak{sp}_{2g}(\mathbb{Z}/d)$ pulled-back to the level-$d$ mapping class group.
Formally, the map
\begin{align*}
A_d & \longrightarrow Hom(\mathcal{M}_{g,1}[d],A)^{\mathcal{AB}_{g,1}} 
\end{align*}
that assigns to $x \in A_d$ the composite:
\[
\xymatrix{
x\varphi_g:\mathcal{M}_{g,1}[d] \ar[r] & Sp_{2g}(\mathbb{Z},d) \ar[r]^{\alpha}_{eq.~\ref{ses_abel_d1} }& \mathfrak{sp}_{2g}(\mathbb{Z}/d) \ar[r] & \\
  {} \ar[r]^-{\pi_{gl}}_-{Lem~\ref{lem:spZ/dcoinv}} & \mathfrak{gl}_g(\mathbb{Z}/d) \ar[r]^-{tr} & \mathbb{Z}/d \ar[r]^{\varepsilon^x} & A_d &
}
\]
is an isomorphism.
\end{prop}

In view of the discussion at the beginning of Section~\ref{subsec:From invariants to trivial cocycles}, and the stability result above, the maps $x\varphi_g$ are the only candidates for defining an invariant that is also homomorphism. Then Lemma~\ref{lem:extensionshandelbody} and Remark~\ref{rem:d=2} together with Proposition~\ref{prop:abinv_modp} shows that they indeed vanish on $\mathcal{A}_{g,1}[d]$, $\mathcal{B}_{g,1}[d]$ for $d \neq 2$ and do not vanish for $d=2$, hence:

\begin{prop}\label{prop:stab_modp}
The homomorphisms $x\varphi_g$ defined in Proposition~\ref{prop:abinv_modp} are compatible with the stabilization map. In particular, for $d\neq 2$, they reassemble into Cardinal($A_d$) normalized invariants, canonically labeled by the elements in $A_d$ and which we denote $x\varphi$; for $d=2$ these maps are not invariants.
\end{prop}

\subsubsection{Non-triviality of the invariants $x\varphi$}\label{subsec:invphix}

We show that the invariants that appear in  Proposition~\ref{prop:stab_modp} are non-trivial (apart from the one labeled by the $0$ element, which we discard).  We do this by showing which  Lens spaces these invariants tear apart. Observe that in the invariant $x\varphi$ the element $x \in A_d$ carries no information from the manifold, it is cleaner to study the $\mathbb{Z}/d$-valued invariant $\varphi$ that results of taking out the map $\varepsilon^x$ from the composition that defines $x\varphi.$ 
 
Let $p,$ $q$ be two coprime integers and let $L(p,q)$ be the associated Lens space. Since this is a $\mathbb{Q}$-homology $3$-sphere, by Theorem \ref{teo_rat_homology_gen}, we know that there exists an integer $d\geq 2$ for which $L(p,q)\in \mathcal{S}^3[d]$. Our first task is to find appropriate values for $d$.

\begin{prop}\label{prop:dforLensspaces}
A Lens space $L(p,q)$ is in $\mathcal{S}^3[d]$ if and only if $p\equiv  \pm 1 \;(\text{mod } d).$
\end{prop}
\begin{proof}
The homology groups of $L(p,q)$ are:
$$H_k(L(p,q);\mathbb{Z})=\left\{\begin{array}{rl}
\mathbb{Z}, & \text{for }k=0,3, \\
\mathbb{Z}/p, & \text{for } k=1, \\
0, & \text{otherwise.} 
\end{array} \right.$$
Then $|H_1(L(p,q);\mathbb{Z})|=p$, and by Theorem \ref{teo_rat_homology_gen}, $L(p,q)\in \mathcal{S}^3[d]$ if and only if $p\equiv \pm 1 \; (\text{mod }d).$
\end{proof}

As a consequence, all Lens spaces in $\mathcal{S}^3[d]$ are of the form $L(\pm 1+dk,q)$ with $k,q\in \mathbb{Z}.$ Next we compute the values of the invariant $\varphi$ on these Lens spaces. By the classification Theorem (cf. \cite{rolf}), two Lens spaces $L(p,q),$ $L(p',q')$ are homeomorphic if and only if $p'=\pm p$ and $q'\equiv \pm q^{\pm 1} \;(\text{mod } p)$. In particular $L(\pm 1+dk,q)$ and $L(1\pm dk,\pm dkq)$ are homeomorphic and it is enough to compute the value of the invariant $\varphi$ on a Lens space of the form $L(1+dk,dl)$, for some $k,l \in  \mathbb{Z}$ with $k|l.$

By definition (cf. \cite[Sec. 9.B]{rolf}) there is a Heegaard splitting of genus $1$ and gluing map $f\in \mathcal{M}_{1,1}$ such that the Lens space
$L(1+dk,dl)$ is homeomorphic to $\mathcal{H}_1\cup_{\iota f}-\mathcal{H}_1$
with
$$\Psi(f)= \left(\begin{matrix}
a & dl \\
b & 1+dk \end{matrix}\right)\in Sp_{2}(\mathbb{Z}) \quad \text{with}\quad a,b\in \mathbb{Z}.$$
Since $Sp_2(\mathbb{Z})=SL_2(\mathbb{Z}),$ the reduction modulo $d$ of $\Psi(f)$ has determinant $1$ and as a consequence $a=1+dr$ for some $r\in \mathbb{Z}.$
Let $\xi_b\in\mathcal{B}_{1,1}$ such that $\Psi(\xi_b)=\left(\begin{smallmatrix}
1 & 0 \\
b & 1 \end{smallmatrix}\right)$, then
\begin{align*}
\Psi(f\xi_b) & =\left(\begin{matrix}
1+dr & dl \\
b & 1+dk \end{matrix}\right)\left(\begin{matrix}
1 & 0 \\
-b & 1 \end{matrix}\right) \\
 & =\left(\begin{matrix}
1+dr-dlb & dl \\
-dkb & 1+dk \end{matrix}\right)\in SL_2(\mathbb{Z},d)
\end{align*}
Therefore $f_d:=f\xi_b\in\mathcal{M}_{1,1}[d]$ and by Singer's Theorem (cf. Theorem \ref{bij_MCG_3man}) the Lens space $L(1+dk,dl)$ is homeomorphic to $\mathcal{H}_1\cup_{\iota f_d}-\mathcal{H}_1.$

Stabilizing three times $f_d\in \mathcal{M}_{1,1}[d]$ we can consider $f_d$ as an element of $\mathcal{M}_{4,1}[d].$
Then,
$$\varphi(L(1+dk,dl))=\varphi_4(f_d)=tr(\pi_{gl}\circ \alpha \circ \Psi(f_d))=-k.$$
Therefore we get the following result:

\begin{prop}\label{prop:valuephiLens}
The invariant $\varphi: \mathcal{S}^3[d] \rightarrow \mathbb{Z}/d$ is not trivial,  it takes the value $-k$ on the Lens space $L(1+dk,q)$.
\end{prop}

\subsection{From trivial cocycles to invariants}

Conversely, what are the conditions for a family of trivial $2$-cocycles $C_g$ on $\mathcal{M}_{g,1}[d]$ satisfying properties (1)-(3) to actually hand us out an invariant?

Firstly we need to check the existence of an $\mathcal{AB}_{g,1}$-invariant trivialization of each $C_g.$
Denote by $\mathcal{Q}_{C_g}$ the set of all normalized  trivializations of the $2$-cocycle $C_g:$
$$\mathcal{Q}_{C_g}=\{q:\mathcal{M}_{g,1}[d]\rightarrow A\mid q(\phi)+q(\psi)-q(\phi\psi)=C_g(\phi,\psi)\}.$$
The group $\mathcal{AB}_{g,1}$ acts on $\mathcal{Q}_g$ via its conjugation action on $\mathcal{M}_{g,1}[d].$ This action confers the set $\mathcal{Q}_{C_g}$ the structure of an affine set over the abelian group $Hom(\mathcal{M}_{g,1}[d],A).$ On the other hand, choosing an arbitrary element $q\in \mathcal{Q}_{C_g}$ the map
\begin{align*}
\rho_q:\mathcal{AB}_{g,1} & \longrightarrow Hom(\mathcal{M}_{g,1}[d],A) \\ \phi & \longmapsto \phi \cdot q-q,
\end{align*}
is a derivation and hence induces a well-defined cohomology class
\[
\rho(C_g)\in H^1(\mathcal{AB}_{g,1};Hom(\mathcal{M}_{g,1}[d],A)),
\]
called the torsor of the cocycle $C_g,$ and we have the following result, which admits a straightforward proof:

\begin{prop}\label{prop_torsor}
The natural action of $\mathcal{AB}_{g,1}$ on $\mathcal{Q}_{C_g}$ admits a fixed point if and only if the associated torsor $\rho(C_g)$ is trivial.
\end{prop}

Suppose that for every $g\geq 4$ there is a fixed point $q_g$ of $\mathcal{Q}_{C_g}$ for the action of $\mathcal{AB}_{g,1}$ on $\mathcal{Q}_{C_g}.$
Since every pair of $\mathcal{AB}_{g,1}$-invariant trivializations differ by an $\mathcal{AB}_{g,1}$-invariant homomorphism, by Proposition~\ref{prop:abinv_modp}, for every $g\geq 4$ the fixed points are exactly given by the maps $q_g+x\varphi_g$ with $x\in A_d.$

By Proposition~\ref{prop:stab_modp}, all elements of $Hom(\mathcal{M}_{g,1}[d],A)^{\mathcal{AB}_{g,1}}$ are compatible with the stabilization map. Then, given two different fixed points $q_g,$ $q'_g$ of $\mathcal{Q}_{C_g}$ for the action of $\mathcal{AB}_{g,1},$ we have that
$${q_g}_{\mid\mathcal{M}_{g-1,1}[d]}-{q'_g}_{\mid\mathcal{M}_{g-1,1}[d]}=(q_g-q'_g)_{\mid\mathcal{M}_{g-1,1}[d]} ={x\varphi_g}_{\mid\mathcal{M}_{g-1,1}[d]}=x\varphi_{g-1}.$$
Therefore the restriction of the trivializations of $\mathcal{Q}_{C_g}$ to $\mathcal{M}_{g-1,1}[d],$ give us a bijection between the fixed points of $\mathcal{Q}_{C_g}$ for the action of $\mathcal{AB}_{g,1}$ and the fixed points of $\mathcal{Q}_{C_{g-1}}$ for the action of $\mathcal{AB}_{g-1,1}$. Given an $\mathcal{AB}_{g,1}$-invariant trivialization $q_g,$ for each $x\in A_d$ we get a well-defined map
$$
q+x\varphi= \lim_{g\to \infty}q_g+x\varphi_g: \lim_{g\to \infty}\mathcal{M}_{g,1}[d]\longrightarrow A.
$$
These are the only candidates to be $A$-valued invariants of rational homology spheres in $\mathcal{S}^3[d]$ with associated family of $2$-cocycles $(C_g)_g.$ For these maps to be invariants, since they are already $\mathcal{AB}_{g,1}$-invariant, we only have to prove that they are constant on the double cosets $\mathcal{A}_{g,1}[d]\backslash \mathcal{M}_{g,1}[d]/\mathcal{B}_{g,1}[d].$
From property (3) of our cocycle we have that $\forall \phi\in \mathcal{M}_{g,1}[d],$
$\forall \psi_a\in \mathcal{A}_{g,1}[d]$ and $\forall \psi_b\in \mathcal{B}_{g,1}[d],$
\begin{equation}
\label{eq_A[p],B[p]_constant}
\begin{aligned}
(q_g+x\varphi_g)(\phi)-(q_g+x\varphi_g)(\phi \psi_a)= & -(q_g+x\varphi_g)(\psi_a) , \\
(q_g+x\varphi_g)(\phi)-(q_g+x\varphi_g)(\psi_b\phi )= & -(q_g+x\varphi_g)(\psi_b).
\end{aligned}
\end{equation}
Thus, in particular, taking $\phi\in \mathcal{A}_{g,1}[d]$ and $\phi\in \mathcal{B}_{g,1}[d]$ in above equations, we have that $q_g+x\varphi_g$ with $x\in A_d,$ are homomorphisms on $\mathcal{A}_{g,1}[d],$ $\mathcal{B}_{g,1}[d].$ Then, by Proposition~\ref{prop:H1BpAp}, for $d\neq 2$, the homomorphisms $q_g+x\varphi_g$ are zero on these last two groups and we conclude by equalities \eqref{eq_A[p],B[p]_constant}.

Summarizing, we get the following result:

\begin{teo}
\label{teo_cocy_p}
Given an integer $d\geq 3$ such that $4\nmid d$ and
$x\in A$ a $d$-torsion element, a family of $2$-cocycles $C_g: \mathcal{M}_{g,1}[d]\times \mathcal{M}_{g,1}[d] \rightarrow A$ for $g \geq 4$ satisfying conditions (1)-(3) provides compatible families of trivializations $F_g+x\varphi_g: \mathcal{M}_{g,1}[d]\rightarrow A$, that reassemble into invariants of rational homology spheres in $\mathcal{S}^3[d]$,
$$\lim_{g\to \infty}F_g+x\varphi_g: \mathcal{S}^3[d]\longrightarrow A$$
if and only if the following two conditions hold:
\begin{enumerate}[(i)]
\item The associated cohomology classes $[C_g]\in H^2(\mathcal{M}_{g,1}[d];A)$ are trivial.
\item The associated torsors $\rho(C_g)\in H^1(\mathcal{AB}_{g,1},Hom(\mathcal{M}_{g,1}[d],A))$ are trivial.
\end{enumerate}
\end{teo}

For $d=2$, Proposition~\ref{prop:H1BpAp} does not imply that homomorphisms $q_g+x\varphi_g$ are zero on $\mathcal{A}_{g,1}[2]$, $\mathcal{B}_{g,1}[2].$
In particular, given $x\in A$ a $2$-torsion element different from zero, the maps $x\varphi$ are not zero on $\mathcal{A}_{g,1}[2]$, $\mathcal{B}_{g,1}[2].$

Nevertheless, we have the following result:

\begin{lema}
\label{lema:diag-a,b-2}
For $g\geq 4$ the inclusions $\mathcal{AB}_{g,1}[2]\subset \mathcal{A}_{g,1}[2],\mathcal{B}_{g,1}[2]\subset \mathcal{M}_{g,1}[2]$ induce a commutative diagram of monomorphisms:
\begin{equation*}
\begin{tikzcd}
Hom(\mathcal{M}_{g,1}[2],A)^{\mathcal{AB}_{g,1}} \ar[r, "\sim"]  \ar[d, sloped, "\sim"] & Hom(\mathcal{A}_{g,1}[2],A)^{\mathcal{AB}_{g,1}} \ar[d, hook] \\
Hom(\mathcal{B}_{g,1}[2],A)^{\mathcal{AB}_{g,1}} \ar[r, hook ] & Hom(\mathcal{AB}_{g,1}[2],A)^{\mathcal{AB}_{g,1}},
\end{tikzcd}
\end{equation*}
where the two hooked arrows indicate monomorphisms.
\end{lema}

\begin{proof}
Consider the following commutative diagram with exact rows:
\[
\begin{tikzcd}[column sep={1.3cm,between origins},row sep=0.5cm]
1 \arrow[rr] && \mathcal{TAB}_{g,1} \arrow[rr]\arrow[rd] \arrow[dd] && \mathcal{AB}_{g,1}[2] \arrow[rr] \arrow[rd] \arrow[dd] && GL_g(\mathbb{Z},2) \arrow[rr] \arrow[rd] \arrow[dd] && 1 &  \\
& 1 \arrow[rr, crossing over] && \mathcal{TA}_{g,1} \arrow[rr, crossing over]  && \mathcal{A}_{g,1}[2] \arrow[rr, crossing over] && Sp_{2g}^A(\mathbb{Z},2) \arrow[rr] && 1  \\
1 \arrow[rr] && \mathcal{TB}_{g,1} \arrow[rr] \arrow[rd] &&\mathcal{B}_{g,1}[2] \arrow[rr] \arrow[rd] && Sp_{2g}^B(\mathbb{Z},2) \arrow[rr] \arrow[rd] && 1 & \\
& 1 \arrow[rr] && \mathcal{T}_{g,1} \arrow[rr] \arrow[uu, crossing over, leftarrow] && \mathcal{M}_{g,1}[2] \arrow[rr] \arrow[uu, crossing over, leftarrow] && Sp_{2g}(\mathbb{Z},2) \arrow[uu, crossing over, leftarrow] \arrow[rr]&& 1,
\end{tikzcd}
\]

Taking  $\mathcal{AB}_{g,1}$-coinvariants in the associated  5-term exact sequences for each row, and applying Lemma \ref{lem:spZ/dcoinv} and Propositions \ref{prop:H1BpAp}, \ref{prop_iso_M[d],Sp[d]}, \ref{prop:com_SpB} to show that some of the involved homology groups are trivial and some of the involved maps are isomorphisms, 
we get a commutative diagram,
\[
\begin{tikzcd}[column sep={2.6cm,between origins},row sep=0.5cm]
 H_1(\mathcal{AB}_{g,1}[2];\mathbb{Z})_{GL_g(\mathbb{Z})} \arrow[rr, twoheadrightarrow] \arrow[rd] \arrow[dd] && H_1(GL_g(\mathbb{Z},2);\mathbb{Z})_{GL_g(\mathbb{Z})}   \arrow[rd, sloped, "\sim"] \arrow[dd, sloped, "\qquad\sim"] & \\
& H_1(\mathcal{A}_{g,1}[2];\mathbb{Z})_{GL_g(\mathbb{Z})}  \arrow[rr, crossing over, sloped, "\sim\;\quad "] && H_1(Sp_{2g}^A(\mathbb{Z},2);\mathbb{Z})_{GL_g(\mathbb{Z})} \arrow[dd , sloped, "\sim"]   \\
H_1(\mathcal{B}_{g,1}[2];\mathbb{Z})_{GL_g(\mathbb{Z})}  \arrow[rr, sloped, "\sim\;\quad "] \arrow[rd] && H_1(Sp_{2g}^B(\mathbb{Z},2);\mathbb{Z})_{GL_g(\mathbb{Z})}  \arrow[rd, sloped, "\sim"] & \\
& H_1(\mathcal{M}_{g,1}[2];\mathbb{Z})_{GL_g(\mathbb{Z})}  \arrow[rr, sloped, "\sim"] \arrow[uu, crossing over, leftarrow] && H_1(Sp_{2g}(\mathbb{Z},2);\mathbb{Z})_{GL_g(\mathbb{Z})}  
\end{tikzcd}
\]
where the double-headed arrow indicates an epimorphism.

By diagram chasing we have a commutative diagram with isomorphisms and epimorphisms
\begin{equation*}
\begin{tikzcd}
H_1(\mathcal{AB}_{g,1}[2];\mathbb{Z})_{GL_g(\mathbb{Z})} \ar[r, twoheadrightarrow ]  \ar[d,twoheadrightarrow] & H_1(\mathcal{A}_{g,1}[2];\mathbb{Z})_{GL_g(\mathbb{Z})} \ar[d, sloped, "\sim"] \\
H_1(\mathcal{B}_{g,1}[2];\mathbb{Z})_{GL_g(\mathbb{Z})} \ar[r, "\sim" ] & H_1(\mathcal{M}_{g,1}[2];\mathbb{Z})_{GL_g(\mathbb{Z})},
\end{tikzcd}
\end{equation*}
and we conclude applying the right exact functor $Hom(-;A)$.
\end{proof}

Thus, given $q_g+x\varphi_g$, a trivialization of the $2$-cocycle $C_g$, its restriction to $\mathcal{A}_{g,1}[2]$ and $\mathcal{B}_{g,1}[2]$ respectively give homomorphisms $F^a_g,$ $F^b_g$ that coincide when restricted to $\mathcal{AB}_{g,1}[2]$. By Lemma \ref{lema:diag-a,b-2}, the homomorphisms $F^a_g,$ $F^b_g$ can be lifted to the same element $y\varphi_g\in Hom(\mathcal{M}_{g,1}[2],A)^{\mathcal{AB}_{g,1}}$ with $y\in A$ an element of $2$-torsion.

Therefore, $q_g+x\varphi_g-y\varphi_g$ is the unique trivialization of the $2$-cocycle $C_g$ that is zero on $\mathcal{A}_{g,1}[2]$, $\mathcal{B}_{g,1}[2],$ and by equalities \eqref{eq_A[p],B[p]_constant} this trivialization is the unique one which is constant on the double cosets $\mathcal{A}_{g,1}[2]\backslash \mathcal{M}_{g,1}[2]/\mathcal{B}_{g,1}[2]$.

\begin{teo}
	\label{teo_cocy_p_2}
A family of $2$-cocycles $C_g: \mathcal{M}_{g,1}[2]\times \mathcal{M}_{g,1}[2] \rightarrow A$  for $g \geq 4$ satisfying conditions (1)-(3) provides a unique compatible family of trivializations $F_g: \mathcal{M}_{g,1}[2]\rightarrow A$, that reassembles into an invariant of rational homology spheres in $\mathcal{S}^3[2]$,
	$$\lim_{g\to \infty}F_g: \mathcal{S}^3[2]\longrightarrow A$$
	if and only if the following two conditions hold:
	\begin{enumerate}[(i)]
		\item The associated cohomology classes $[C_g]\in H^2(\mathcal{M}_{g,1}[2];A)$ are trivial.
		\item The associated torsors $\rho(C_g)\in H^1(\mathcal{AB}_{g,1},Hom(\mathcal{M}_{g,1}[2],A))$ are trivial.
	\end{enumerate}
\end{teo}

\textbf{The case $\mathcal{M}_{g,1}[p]$ and $A=\mathbb{Z}/p$ with $p$ an odd prime.}
In this particular case there is the following characterization of the torsor class $\rho(C_g):$

\begin{prop}\label{prop:torsorforprimep}
The torsor class $\rho(C_g)$ is naturally an element of the group $$Hom(H_1(\mathcal{AB}_{g,1}[p])\otimes(\Lambda^3 H_p\oplus \mathfrak{sp}_{2g}(\mathbb{Z}/p)),\mathbb{Z}/p)^{SL^{\pm}_g(\mathbb{Z}/p)},$$
where $H_1(\mathcal{AB}_{g,1}[p])\otimes(\Lambda^3 H_p\oplus \mathfrak{sp}_{2g}(\mathbb{Z}/p))$ is endowed with the diagonal action and $\mathbb{Z}/p$ with the trivial one.
\end{prop}
\begin{proof}
By \cite[Theorem 0.4]{sato_abel}, for $p$ an odd prime and $g\geq 3$, $H_1(\mathcal{M}_{g,1}[p];\mathbb{Z})\simeq \Lambda^3H_p\oplus \mathfrak{sp}_{2g}(\mathbb{Z}/p)$, where this decomposition is in fact as $\mathcal{M}_{g,1}$-modules. Then we have that
$$Hom(\mathcal{M}_{g,1}[p],\mathbb{Z}/p)=Hom(H_1(\mathcal{M}_{g,1}[p];\mathbb{Z}),\mathbb{Z}/p)=(\Lambda^3H_p)^*\oplus (\mathfrak{sp}_{2g}(\mathbb{Z}/p))^*,$$
where the asterisk means $\mathbb{Z}/p$-dual.
As a consequence the torsor class $\rho(C_g)$
is an element of the group
$H^1(\mathcal{AB}_{g,1}; (\Lambda^3H_p)^*\oplus (\mathfrak{sp}_{2g}(\mathbb{Z}/p))^*)$.
To make notation lighter we set $M=(\Lambda^3 H_p)^*\oplus (\mathfrak{sp}_{2g}(\mathbb{Z}/p))^*$.

By Lemma \ref{lem_B[p]} we have a short exact sequence
$$
\xymatrix@C=7mm@R=13mm{1 \ar@{->}[r] & \mathcal{AB}_{g,1}[p] \ar@{->}[r] & \mathcal{AB}_{g,1} \ar@{->}[r] & SL^{\pm}_g(\mathbb{Z}/p) \ar@{->}[r] & 1 .}
$$
Consider its associated 5-term exact sequence with coefficients in $M$,
\[
\begin{tikzcd}
	0 \arrow[r] &  H^1(SL^{\pm}_g(\mathbb{Z}/p); M^{\mathcal{AB}_{g,1}[p]}) \arrow[r] & H^1(\mathcal{AB}_{g,1}; M)  \ar[out=350, in=170, dl,overlay] \\
	& H^1(\mathcal{AB}_{g,1}[p]; M)^{SL^{\pm}_g(\mathbb{Z}/p)}. &
\end{tikzcd}
\]

Since $\mathcal{AB}_{g,1}[p]$ acts trivially on $M$ we have that
\begin{align*}
H^1(SL^{\pm}_g(\mathbb{Z}/p); M^{\mathcal{AB}_{g,1}[p]})= & H^1(SL^{\pm}_g(\mathbb{Z}/p); M), \\
H^1(\mathcal{AB}_{g,1}[p]; M)^{SL^{\pm}_g(\mathbb{Z}/p)}= & Hom(H_1(\mathcal{AB}_{g,1}[p]); M)^{SL^{\pm}_g(\mathbb{Z}/p)}.
\end{align*}
By the tensor-hom adjunction this last group is isomorphic to
$$Hom(H_1(\mathcal{AB}_{g,1}[p])\otimes(\Lambda^3 H_p\oplus \mathfrak{sp}_{2g}(\mathbb{Z}/p)),\mathbb{Z}/p)^{SL^{\pm}_g(\mathbb{Z}/p)}.$$
To finish the proof we show that 
$H^1(SL^{\pm}_g(\mathbb{Z}/p); M)=0$.

Since $M=(\Lambda^3 H_p)^*\oplus (\mathfrak{sp}_{2g}(\mathbb{Z}/p))^*$ as $Sp_{2g}(\mathbb{Z}/p)$-modules, there is
an isomorphism
$$
	H^1(SL_g^\pm(\mathbb{Z}/p);M) \;\simeq
	\begin{array}{c}
	H^1(SL^\pm_g(\mathbb{Z}/p);(\Lambda^3 H_p)^*) \\
	\oplus \\
	H^1(SL^\pm_g(\mathbb{Z}/p);(\mathfrak{sp}_{2g}(\mathbb{Z}/p))^*).
	\end{array}
$$

By the Center kills Lemma, the group $H^1(SL^\pm_g(\mathbb{Z}/p);(\Lambda^3 H_p)^*)$ is zero, since $-Id$ belongs to the center of $SL_g^\pm(\mathbb{Z}/p)$ and acts on $\Lambda^3H_p$ as the multiplication by $-1$.

Thus it is enough to show that $H^1(SL^\pm_g(\mathbb{Z}/p);(\mathfrak{sp}_{2g}(\mathbb{Z}/p))^*)$ is zero. We make this computation by proving the following statements:
\begin{enumerate}[a)]
\item $H^1(SL^{\pm}_g(\mathbb{Z}/p); (\mathfrak{sp}_{2g}(\mathbb{Z}/p))^*)\longrightarrow H^1(SL_g(\mathbb{Z}/p); (\mathfrak{sp}_{2g}(\mathbb{Z}/p))^*)$ is a mo\-no\-mor\-phism.
\item $H^1(SL_g(\mathbb{Z}/p); (\mathfrak{sp}_{2g}(\mathbb{Z}/p))^*)=0.$
\end{enumerate}

\textbf{a)} By definition we have a short  short exact sequence
$$
\xymatrix@C=7mm@R=13mm{1 \ar@{->}[r] & SL_g(\mathbb{Z}/p) \ar@{->}[r] & SL^{\pm}_g(\mathbb{Z}/p) \ar@{->}[r]^-{det} & \mathbb{Z}/2 \ar@{->}[r] & 1 }
$$
and the statement follows then from the $5$-term exact sequence, and the fact that since $(\mathfrak{sp}_{2g}(\mathbb{Z}/p))^*$ is a $p$-group and $p$ is odd,  $H^\ast(\mathbb{Z}/2;(\mathfrak{sp}_{2g}(\mathbb{Z}/p))^*) =0$ in degrees $ \geq 1$.

\textbf{b)} By \cite[Lemma 4.5]{Putman}, $(\mathfrak{sp}_{2g}(\mathbb{Z}/p))^*\simeq \mathfrak{sp}_{2g}(\mathbb{Z}/p)$ as $SL_g(\mathbb{Z}/p)$-modules. As a consequence, we have an isomorphism
$$H^1(SL_g(\mathbb{Z}/p);(\mathfrak{sp}_{2g}(\mathbb{Z}/p))^*)\simeq H^1(SL_g(\mathbb{Z}/p);\mathfrak{sp}_{2g}(\mathbb{Z}/p)).
$$
Furthermore,  as $SL_g(\mathbb{Z}/p)$-modules we know that:
\[
\mathfrak{sp}_{2g}(\mathbb{Z}/p)\simeq \mathfrak{gl}_g(\mathbb{Z}/p)\oplus Sym^A_g(\mathbb{Z}/p)\oplus Sym^B_g(\mathbb{Z}/p).
\]
Remember that $H_p=A_p\oplus B_p$. If we set $V_p=A_p$ then $V_p^\ast=B_p$ and we have the following isomorphisms of $SL_g(\mathbb{Z}/p)$-modules:
\[
\mathfrak{gl}_g(\mathbb{Z}/p)\simeq  \; Hom(V_p,V_p),\quad
Sym_g^A(\mathbb{Z}/p)\simeq  \; S^2(V_p), Sym_g^B(\mathbb{Z}/p)\simeq  \; S^2(V_p^\ast).
\]

Hence
$H^1(SL_g(\mathbb{Z}/p);\mathfrak{sp}_{2g}(\mathbb{Z}/p))$ is isomorphic to the direct sum 
\begin{align*}
	H^1(SL_g(\mathbb{Z}/p);Hom(V_p,V_p)) \oplus H^1(SL_g(\mathbb{Z}/p);S^2(V_p)) \\
	 \oplus H^1(SL_g(\mathbb{Z}/p);S^2(V_p^\ast)).
\end{align*}

By \cite[Thm. 5]{chih} we know that $H^1(SL_g(\mathbb{Z}/p);Hom(V_p,V_p))=0$.

Next we show that $H^1(SL_g(\mathbb{Z}/p);S^2(V_p))=0$.

Consider the following split short exact sequence of $SL_g(\mathbb{Z}/p)$-modules
$$
\xymatrix@C=7mm@R=13mm{0 \ar@{->}[r] & S^2(V_p) \ar@{->}[r] & V_p\otimes V_p \ar@{->}[r] & \Lambda^2 V_p \ar@{->}[r] & 0 .}
$$

It induces the long exact sequence
\[
\begin{tikzcd}
	0 \arrow[r] & H^1(SL_g(\mathbb{Z}/p);S^2(V_p)) \arrow[r] & H^1(SL_g(\mathbb{Z}/p);V_p\otimes V_p) \arrow[overlay,out=350,in=170, dl] \\
	 & H^1(SL_g(\mathbb{Z}/p);\Lambda^2 V_p) .
\end{tikzcd}
\]
By \cite[Thm. 5]{chih}, we know that $H^1(SL_g(\mathbb{Z}/p);V_p\otimes V_p)=0$ and hence $H^1(SL_g(\mathbb{Z}/p);S^2(V_p))=0$.

For the final case, $H^1(SL_g(\mathbb{Z}/p);S^2(V_p^\ast))=0$, we observe that as $SL_g(\mathbb{Z}/p)$-modules, $V_p^\ast \otimes V_p^\ast \simeq V_p^\ast\otimes \Lambda^{g-1}V_p \simeq Hom(V_p,\Lambda^{g-1}V_p)$. Then  \cite[Thm. 5]{chih} shows again that $H^1(SL_g(\mathbb{Z}/p);V_p^\ast \otimes V_p^\ast)=0$, and we finish the computation as in the previous case.
\end{proof}

\section{Construction of an invariant from the abelianization of Mod-p Torelli Group }

We now want to use Theorem~\ref{teo_cocy_p} to actually get new invariants. As in the integral case (cf. \cite{pitsch}), we will lift families of $2$-cocycles on abelian quotients of $\mathcal{M}_{g,1}[p]$ with $p$ prime. In order to avoid some peculiarities in the homology of low genus mapping class groups as well as some issues involving the prime numbers $2$ and $3$, in all what follows we restrict ourselves to prime numbers $p\geq 5$ and genus $g\geq 4.$

Let $\Gamma_k$ denote the lower central series of $\pi_1(\Sigma_{g,1})$ defined inductively by
$$\Gamma_1=\pi_1(\Sigma_{g,1}),\qquad \Gamma_{k+1}=[\Gamma_1,\Gamma_k] \quad \forall k\geq 1.$$
The Zassenhauss and Stallings mod-$p$ central series are respectively defined by
\begin{equation*}
\Gamma_k^Z=\prod_{ip^j\geq k} (\Gamma_i)^{p^j},
\qquad \qquad
\Gamma_k^S=\prod_{i+j= k} (\Gamma_i)^{p^j},
\end{equation*}
These are respectively the fastest descending series
with $[\Gamma_k^Z,\Gamma_l^Z]<\Gamma_{k+1}^Z,$
$(\Gamma_k^Z)^p<\Gamma_{pk}^Z$ and $[\Gamma_k^S,\Gamma_l^S]<\Gamma_{k+1}^S,$
$(\Gamma_k^S)^p<\Gamma_{k+1}^S.$ For further information about these series we refer the interested reader to \cite{riba1}.

In the same way we construct the higher Johnson homomorphisms (cf. \cite{mor_ext}), we have induced representations:
$$\rho_k^Z:\; \mathcal{M}_{g,1}\longrightarrow Aut(\Gamma/\Gamma_{k+1}^Z),\qquad \rho_k^S:\; \mathcal{M}_{g,1}\longrightarrow Aut(\Gamma/\Gamma_{k+1}^S).$$
Set $\mathcal{L}_{k+1}^Z=\Gamma^Z_{k+1}/\Gamma^Z_{k+2}$ and $\mathcal{L}_{k+1}^S=\Gamma^S_{k+1}/\Gamma^S_{k+2}.$
The restriction of $\rho_{k+1}^Z$ (resp. $\rho_{k+1}^S$) to the kernel of $\rho_k^Z$ (resp. $\rho_k^S$) give homomorphisms:
\begin{equation*}
\tau_k^Z: \ker(\rho_k^Z) \longrightarrow Hom(H_p,\mathcal{L}_{k+1}^Z),
\qquad 
\tau_k^S: \ker(\rho_k^S) \longrightarrow Hom(H_p,\mathcal{L}_{k+1}^S),
\end{equation*}
called \textit{the Zassenhaus (resp. Stallings) mod-$p$ Johnson homomorphisms}.

In this section we only use these homomorphisms for $k=1,2$ and $p\geq 5$.
Notice that by definition of Zassenhauss mod-$p$ central series, for $l\leq p$, we have that $\Gamma^Z_l=\Gamma_l\cdot \Gamma^p$. Then, for $k+1<p$ (in particular for $k=1,2$ and $p\geq 5$), by the classical commutator identities (cf. \cite[Chapter 10]{hall}), we have that
$\mathcal{L}_{k+1}^Z=\Gamma^Z_{k+1}/\Gamma^Z_{k+2}=(\Gamma_{k+1}\Gamma^p)/(\Gamma_{k+2}\Gamma^p)=\mathcal{L}_{k+1}\otimes \mathbb{Z}/p=:\mathcal{L}_{k+1}(H_p),$ where $\mathcal{L}_{k+1}$ stands for $\Gamma_{k+1}/\Gamma_{k+2}.$ 

Moreover, $\ker(\rho_{1}^Z)=\ker(\rho_{1}^S)=\mathcal{M}_{g,1}[p]$ and by \cite{coop} the images of $\tau_1^Z$ and $\tau_1^S$ are respectively isomorphic to $\Lambda^3H_p$ and to the abelianization of $\mathcal{M}_{g,1}[p]$ when $p$ is an odd prime.

In \cite{per}, \cite{Putman_abel}, \cite{sato_abel}, independently B.~Perron, A.~Putman and M.~Sato computed that for $g\geq 3$ and an odd prime $p$ there is an isomorphism of $\mathbb{Z}/p$-modules:
$$H_1(\mathcal{M}_{g,1}[p];\mathbb{Z})\simeq \Lambda^3 H_p \oplus \mathfrak{sp}_{2g}(\mathbb{Z}/p).$$
It turns out that this decomposition holds true as $\mathcal{M}_{g,1}$-modules. Indeed, there is a commutative diagram with exact rows:
\begin{equation}
\label{diag-split-M[p]}
\begin{aligned}
	\xymatrix{
1 \ar[r] &  \mathcal{T}_{g,1} \ar[d] \ar[r] & \mathcal{M}_{g,1}[p] \ar[r] \ar[d]^{\tau_1^S} \ar@{..>}[dl]_{\tau^Z_1} \ar@{..>}[dr]^{\alpha \circ \Psi} & Sp_{2g}(\mathbb{Z},p) \ar[r] \ar[d] & 1 \\
	0\ar[r] & \Lambda^3H_p \ar[r] & H_1( \mathcal{M}_{g,1}[p]; \mathbb{Z}) \ar[r] & \mathfrak{sp}_{2g}(\mathbb{Z}/p) \ar[r]& 0.	
}
\end{aligned}
\end{equation}
By definition, all these maps are compatible with the action by $\mathcal{M}_{g,1}$. The equivariant map $\tau_1^Z$ induces a retraction of the bottom exact sequence,
which shows our claim.
Moreover the splitting as $\mathcal{M}_{g,1}$-modules is unique since different $\mathcal{M}_{g,1}$-equivariant sections differ by an element in the group $H^1(\mathfrak{sp}_{2g}(\mathbb{Z}/p);\Lambda^3 H_p)^{Sp_{2g}(\mathbb{Z}/p)},$ but this group is zero by the Center Kills Lemma.

\subsection{The images of $\mathcal{A}_{g,1}[p]$ and $\mathcal{B}_{g,1}[p]$ under the abelianization of $\mathcal{M}_{g,1}[p]$.}

Our standard decomposition $H_1(\Sigma_{g,1};\mathbb{Z})=A\oplus B$ induces decompositions:
\[
\Lambda^3 H_p=  W_{AB}^p\oplus W_{A}^p\oplus W_{B}^p
\]
and
\[\
\mathfrak{sp}_{2g}(\mathbb{Z}/p)=  \mathfrak{gl}_g(\mathbb{Z}/p)\oplus Sym^A_g(\mathbb{Z}/p)\oplus Sym^B_g(\mathbb{Z}/p),
\]
where $\quad W_{A}^p=\Lambda^3 A_p,\quad W_B^p=\Lambda^3 B_p,\quad W_{AB}^p=B_p\wedge (\Lambda^2 A_p)\oplus A_p \wedge (\Lambda^2 B_p).$

\begin{prop}
	\label{prop-im-ext}
	Given an integer $g\geq 3$ and an odd prime $p$, the images of $\mathcal{A}_{g,1}[p]$ and $\mathcal{B}_{g,1}[p]$ in the abelianization of the level-$p$ mapping class group are respectively,
	$$
	W_{AB}^p\oplus W_{A}^p\oplus \mathfrak{sl}_{g}(\mathbb{Z}/p)\oplus Sym_g^A(\mathbb{Z}/p),$$
	$$ \text{and}$$
	$$ W_{AB}^p\oplus W_{B}^p\oplus \mathfrak{sl}_{g}(\mathbb{Z}/p)\oplus Sym_g^B(\mathbb{Z}/p).$$
\end{prop}

The first half of the computation is given by:

\begin{lema}\label{lem_im_A[p],B[p]}
	For an integer $g\geq 3$ and an odd prime $p$, the images of $\mathcal{A}_{g,1}[p]$ and $\mathcal{B}_{g,1}[p]$ in $\bigwedge^3H_p$ are respectively $$W_A^p\oplus W_{AB}^p\quad\text{and}\quad W_B^p\oplus W_{AB}^p.$$ 
\end{lema}
\begin{proof}
	We only do the proof for $\mathcal{B}_{g,1}[p]$. For $\mathcal{A}_{g,1}[p]$ the argument is analogous.
	Consider the following commutative diagram with exact rows:
	\begin{equation}
		\label{diag_com_B}
		\xymatrix@C=7mm@R=10mm{
			1 \ar@{->}[r] & \mathcal{TB}_{g,1} \ar@{->}[r] \ar@{->}[d]^{\tau_1^Z} & \mathcal{B}_{g,1} \ar@{->}[r]^-{\Psi} \ar@{->}[d]^-{\rho_2^Z} & Sp_{2g}^B(\mathbb{Z}) \ar@{->}[r] \ar@{->}[d]^{r_p} & 1\\
			0  \ar@{->}[r] & W_B^p\oplus W_{AB}^p\ar@{->}[r]
			& \rho_2^Z(\mathcal{B}_{g,1}) \ar@{->}[r] & Sp_{2g}^{B\pm}(\mathbb{Z}/p) \ar@{->}[r] & 1 .}
	\end{equation}
	Since $-Id$ acts as $-1$ on $W_B^p\oplus W_{AB}^p,$ by the Center kills Lemma the cohomology groups
	$H^i(Sp_{2g}^{B\pm}(\mathbb{Z}/p);W_B^p\oplus W_{AB}^p)$ with $i=1,2$ are zero and therefore the bottom row of diagram \eqref{diag_com_B} splits with only one $Sp_{2g}^{B\pm}(\mathbb{Z}/p)$-conjugacy class of splittings.
	
	Composing a retraction map $r: \rho_2^Z(\mathcal{B}_{g,1}) \rightarrow W_B^p\oplus W_{AB}^p$ with $\rho_2^Z$ we get a crossed homomorphism:
	$$k_B: \mathcal{B}_{g,1}\longrightarrow W_B^p\oplus W_{AB}^p.$$
	By commutativity of the left hand square, $k_B$ and $\tau_1^Z$ coincide on $\mathcal{TB}_{g,1}$
	and by Proposition~\ref{prop_res_B} these homomorphisms coincide on $\mathcal{B}_{g,1}[p]$ as well.
\end{proof}

The second half of the computation is:

\begin{lema}\label{lema:imA[d]B[d]sp}
	Given an integer $g\geq 3$ and an odd prime $p$, the images of $\mathcal{A}_{g,1}[p]$ and $\mathcal{B}_{g,1}[p]$ in $\mathfrak{sp}_{2g}(\mathbb{Z}/p)$ are respectively
	$$ \mathfrak{sl}_{g}(\mathbb{Z}/p)\oplus Sym_g^A(\mathbb{Z}/p) \quad \text{and}\quad \mathfrak{sl}_{g}(\mathbb{Z}/p)\oplus Sym_g^B(\mathbb{Z}/p).$$
\end{lema}

\begin{proof}
We only do the proof for $\mathcal{B}_{g,1}[p]$. For $\mathcal{A}_{g,1}[p]$ the proof is analogous.
As we have already seen in Lemma~\ref{lem:extensionshandelbody}, the image of $\mathcal{B}_{g,1}[p]$ under the symplectic representation is:
$$Sp_{2g}^B(\mathbb{Z},p) =SL_g(\mathbb{Z},p)\ltimes Sym_g^B(p\mathbb{Z}).$$
By \cite[Thm. 1.1]{Lee} and \cite[Lemma~1]{newman}, the map $\alpha:Sp_{2g}(\mathbb{Z},p) \rightarrow \mathfrak{sp}_{2g}(\mathbb{Z}/p)$ restricted to $SL_g(\mathbb{Z},p)$ and $Sym_g(p\mathbb{Z})$ give epimorphisms
$$SL_g(\mathbb{Z},p) \twoheadrightarrow \mathfrak{sl}_g(\mathbb{Z}/p), \qquad Sym_g(p\mathbb{Z})\twoheadrightarrow Sym_g(\mathbb{Z}/p),$$
and therefore we get the result.
\end{proof}

\subsection{Trivial cocycles in the abelianization of $\mathcal{M}_{g,1}[p]$}
\label{sec-triv-cocy-abel-Mp}
Fix an odd prime $p$. A family of $2$-cocycles $(B_g)_{g\geq 4}$ on $H_1(\mathcal{M}_{g,1}[p];\mathbb{Z})$ whose lift to $\mathcal{M}_{g,1}[p]$ satisfy properties (1)-(3) given in Theorem \ref{teo_cocy_p}, has the following properties:
\begin{enumerate}[($1'$)]
\item The $2$-cocycles $(B_g)_{g \geq 4}$ are compatible with the stabilization map, in other words, for $g\geq 4$ there is a commutative triangle:

$$
\begin{tikzcd}
	H_1(\mathcal{M}_{g,1}[p];\mathbb{Z}) \times H_1(\mathcal{M}_{g,1}[p];\mathbb{Z}) \arrow[d]    \arrow[dr, "B_{g}"] & \\
	H_1(\mathcal{M}_{g+1,1}[p];\mathbb{Z}) \times H_1(\mathcal{M}_{g+1,1}[p];\mathbb{Z}) \arrow[r,"B_{g+1}"'] 
	& \mathbb{Z}/p.
\end{tikzcd}
$$

\item The $2$-cocycles $(B_g)_{g \geq 4}$ are invariant under conjugation by elements in $GL_g(\mathbb{Z}),$
\item If either $\phi\in \tau_1^S(\mathcal{A}_{g,1}[p])$ or $\psi \in \tau_1^S(\mathcal{B}_{g,1}[p])$ then $B_g(\phi, \psi)=0.$
\end{enumerate}

The key observation for the following computations is that condition $(3')$ forces the $2$-cocycle $B_g$ to be  linear on a large portion of $H_1(\mathcal{M}_{g,1}[p];\mathbb{Z})$.

\begin{lema}\label{lem:linearpropBg}
	For an integer $g\geq 4$ and an odd prime $p$, let $B_g$ be a $2$-cocycle on  $H_1(\mathcal{M}_{g,1}[p];\mathbb{Z})$ that satisfies property $(3')$ above, and $x,y \in H_1(\mathcal{M}_{g,1}[p];\mathbb{Z})$. Then
	\begin{enumerate}
		\item[i)] For any $x_a \in  \tau_1^S(\mathcal{A}_{g,1}[p])$, $B_g(x_a+x,y)= B_g(x,y)$,
		\item[ii)] For any $y_b \in  \tau_1^S(\mathcal{B}_{g,1}[p])$, $B_g(x,y)= B_g(x,y+y_b)$.
	\end{enumerate}
		In particular, if $c,d \in \Lambda^3H_p$ then
	\begin{enumerate}
		\item[iii)] $B(c+d,y)= B(c,y) + B(d,y)$ and $B(y,c+d) = B(y,c) + B(y,d)$.
	\end{enumerate}
\end{lema} 

\begin{proof}
	To prove $i)$ we apply the cocycle condition to $B(x_a+x,y)$:
	\begin{align*}
		B(x_a+x,y) = & B(x_a,x+y) - B(x_a,x) \\
		& + B(x,y),
	\end{align*}
and observe that the first two terms on the right-hand side of this equation vanish because of property $(3')$.
The same argument proves $ii)$. To prove point $iii)$, remember from Lemma~\ref{lem_im_A[p],B[p]} that any element $z \in \Lambda^3 H_p$ decomposes (non uniquely) as a sum $z=z_a + z_b$, where $z_a \in \tau_1^Z(\mathcal{A}_{g,1}[p])$ and $z_b \in \tau_1^Z(\mathcal{B}_{g,1}[p])$. Then, by the cocycle identity for the second equality,  point $ii)$ above and property $(3')$ for the third equality we have that
\begin{align*}
	B(c+d,y) & = B(c_b+d_b,y) \\
	& = B(c_b,d_b+y) - B(c_b,d_b) + B(d_b,y) \\
	& = B(c_b,y) + B(d_b,y) \\
	&=  B(c,y) + B(d,y).
\end{align*}
The proof of the equality $B(y,c+d)$ is entirely analogous.
\end{proof}
The following lemma prompts us to search  families of $2$-cocycles on $\Lambda^3 H_p$ and $\mathfrak{sp}_{2g}(\mathbb{Z}/p)$ that independently satisfy conditions analogous to ($1'$)-($3'$).

\begin{lema}
\label{lema-sep-coc}
Given an odd prime p, a family of $2$-cocycles $(B_g)_{g\geq 4}$ on $H_1(\mathcal{M}_{g,1}[p];\mathbb{Z})$ satisfies conditions ($1'$)-($3'$) if and only if it can be written as the sum of a family of 2-cocycles $(B_g^\Lambda)_{g\geq 4}$ pulled-back from $\Lambda^3 H_p$ and a family of 2-cocycles $(B_g^{sp})_{g\geq 4}$ pulled-back from $\mathfrak{sp}_{2g}(\mathbb{Z}/p)$ that independently satisfy the analogous conditions.
\end{lema}

\begin{proof}
The proof of this statement is based on the fact that the bottom short exact sequence in commutative diagram \eqref{diag-split-M[p]} has a unique splitting as $\mathcal{M}_{g,1}$-modules.

That the condition is sufficient is clear, hence we just prove that it is necessary.
Assume that we have a family of 2-cocycles $(B_g)_{g\geq 4}$ on $H_1(\mathcal{M}_{g,1}[p];\mathbb{Z})$ that satisfies conditions ($1'$)-($3'$).
Since $-Id\in GL_g(\mathbb{Z})$ acts as multiplication by $-1$ on $\Lambda^3 H_p$ and trivially on $\mathfrak{sp}_{2g}(\mathbb{Z}/p),$ property ($2'$) and Lemma~\ref{lem:linearpropBg} $iii)$ show that the 2-cocycles $B_g$ are zero on $\Lambda^3 H_p\times \mathfrak{sp}_{2g}(\mathbb{Z}/p)$ and $\mathfrak{sp}_{2g}(\mathbb{Z}/p)\times \Lambda^3H_p.$ 
Given two elements $x=x_\Lambda+x_{sp}$ and $y=y_\Lambda+y_{sp}$ in $H_1(\mathcal{M}_{g,1}[p];\mathbb{Z}),$ written according to the decomposition $\Lambda^3 H_p \oplus \mathfrak{sp}_{2g}(\mathbb{Z}/p),$
the $2$-cocycle relation give us:
\begin{align*}
B_g(x_\Lambda+x_{sp},y_\Lambda+y_{sp})= &  B_g(x_{sp},y_\Lambda+y_{sp})+B_g(x_\Lambda,x_{sp}+y_\Lambda+y_{sp}) \\
= & B_g(x_{sp},y_{sp}+y_\Lambda)+B_g(x_\Lambda,y_\Lambda+(x_{sp}+y_{sp}))\\
= & B_g(x_{sp},y_{sp})+B_g(x_\Lambda,y_\Lambda)=B^{sp}_g(x,y)+B^\Lambda_g(x,y),
\end{align*}
where $B^\Lambda_g$ and $B^{sp}_g$ stand for the respective restrictions of $B_g$ to $\Lambda^3 H_p$ and to $\mathfrak{sp}_{2g}(\mathbb{Z}/p)$ pulled-back to $H_1(\mathcal{M}_{g,1}[p];\mathbb{Z}).$
Since all maps involved in commutative diagram \eqref{diag-split-M[p]} are $\mathcal{M}_{g,1}$-equivariant and compatible with the stabilization map, the families of 2-cocycles $(B_g^\Lambda)_{g\geq 4}$ and $(B_g^{sp})_{g\geq 4}$ satisfy properties ($1'$) and ($2'$). Finally, for any $a\in \tau_1^S(\mathcal{A}_{g,1}[p])$ and $x=x_\Lambda+x_{sp}\in H_1(\mathcal{M}_{g,1}[p];\mathbb{Z}),$ from property ($3'$) of $B_g$ we get that
$$0=B_g(a,x_\Lambda)= B_g^\Lambda(a,x_\Lambda)+ B_g^{sp}(a,x_\Lambda)=B_g^\Lambda(a,x_\Lambda)=B_g^\Lambda(a,x).$$
Similarly, $B_g^\Lambda(x,b)=0$ with $b\in \tau_1^S(\mathcal{B}_{g,1}[p]).$ Therefore $B_g^\Lambda$ satisfy ($3'$).

An analogous argument shows that $B_g^{sp}$ also satisfy ($3'$).
\end{proof}

\subsection{Candidate families of 2-cocycles to build invariants}

We find the families of 2-cocycles on the abelianization of the level-$p$ mapping class group that satisfy conditions ($1'$)-($3'$). Because of Lemma \ref{lema-sep-coc} we search separately on $\Lambda^3 H_p$ and on $\mathfrak{sp}_{2g}(\mathbb{Z}/p)$.

\subsubsection{2-cocycles on $\Lambda^3 H_p$.} For each $g$, the intersection form on the first homology group of $\Sigma_{g,1}$ induces a bilinear form $\omega:A\otimes B\rightarrow \mathbb{Z}/p$, which in turn induces two bilinear forms
$$ J_g:W_A^p\otimes W_B^p \rightarrow \mathbb{Z}/p\quad \text{and} \quad {}^t\!J_g:W_B^p\otimes W_A^p \rightarrow \mathbb{Z}/p$$ that we extend by $0$ to degenerate bilinear forms on
$$\Lambda^3H_p=W^p_A\oplus W^p_{AB}\oplus W^p_{B}.$$

Written as matrices according to the aforementioned decomposition these are:
$$J_g:=\left(\begin{matrix}
0 & 0 & Id\\
0 & 0 & 0 \\
0 & 0 & 0 
\end{matrix}\right),
\qquad {}^t\!J_g:=\left(\begin{matrix}
0 & 0 & 0\\
0 & 0 & 0 \\
Id & 0 & 0 
\end{matrix}\right).$$
We show that only multiples of $ ({}^t\!J_g)_{g\geq 4}$ satisfy conditions ($1'$)-($3'$).
As a direct consequence of Point $iii)$ in Lemma \ref{lem:linearpropBg} we have the following result:

\begin{prop}
\label{prop-ext-cocy-red}
Given an integer $g\geq 4$ and an odd prime $p,$ every $2$-cocycle on $\Lambda^3H_p$ that satisfies condition ($3'$) is a bilinear form on $\Lambda^3H_p.$
\end{prop}

We now compute the group of $GL_g(\mathbb{Z})$-invariant bilinear forms on $\Lambda^3 H_p.$
Consider the $Sp_{2g}(\mathbb{Z}/p)$-invariant bilinear forms:
\begin{align*}
\Theta_g(c_i\wedge c_j\wedge c_k\otimes c'_i\wedge c'_j\wedge c'_k) & =\sum_{\sigma\in \mathfrak{S}_3}\varepsilon(\sigma)(\omega(c_i,c'_{\sigma(i)})\omega(c_j,c'_{\sigma(j)})\omega(c_k,c'_{\sigma(k)})),\\
Q_g(c_i\wedge c_j\wedge c_k\otimes c'_i\wedge c'_j\wedge c'_k) & =\omega(C(c_i\wedge c_j\wedge c_k),C(c'_i\wedge c'_j\wedge c'_k)),
\end{align*}
where $\varepsilon(\sigma)$ denotes the sign of the permutation $\sigma$, the map $\omega$ denotes the intersection form on $H$ and $C$ denotes the contraction map
$$C(a\wedge b\wedge c)=2[\omega(b,c)a+\omega(c,a)b+\omega(a,b)c].$$
Let $\pi_A,\pi_{A^2B},\pi_{AB^2},\pi_{B}$ denote the respective projections on each component of the decomposition of $GL_g(\mathbb{Z})$-modules $\Lambda^3 H_p=W_A^p\oplus W_{A^2B}^p\oplus W_{AB^2}^p\oplus W_B^p,$
with $W_{A}^p=\Lambda^3 A_p,$ $W_B^p=\Lambda^3 B_p,$ $W_{A^2B}^p=B_p\wedge (\Lambda^2 A_p),$ $W_{AB^2}^p= A_p \wedge (\Lambda^2 B_p).$

\begin{prop}
\label{prop:bilin_extp}
Given an odd prime $p$ and an integer $g\geq 4$, the composition of the $2$-cocycles $\Theta_g, Q_g$  with the projections $\pi_A,$ $\pi_{B^2A},$ $\pi_{A^2B},$ $\pi_B$ give $GL_g(\mathbb{Z})$-invariant bilinear forms 
$$\begin{array}{lll}
J_g=\Theta_g(\pi_A,\pi_B), & \Theta_g(\pi_{A^2B},\pi_{B^2A}), & Q_g(\pi_{A^2B},\pi_{B^2A}), \\
{}^t\!J_g=-\Theta_g(\pi_B,\pi_A), & \Theta_g(\pi_{B^2A},\pi_{A^2B}),& Q_g(\pi_{B^2A},\pi_{A^2B}), \\
\end{array}$$
and these form a basis of
$$Hom\left(\Lambda^3H_p\otimes \Lambda^3 H_p;\;\mathbb{Z}/p\right)^{GL_g(\mathbb{Z})}.$$
\end{prop}

\begin{proof}
Since the bilinear forms $\Theta_g,$ $Q_g$ are $Sp_{2g}(\mathbb{Z}/p)$-invariant and the projections $\pi_A,$ $\pi_{A^2B},$ $\pi_{AB^2},$ $\pi_{B}$ are $GL_g(\mathbb{Z})$-equivariant, by construction the bilinear forms given in the statement are $GL_g(\mathbb{Z})$-invariant.

Because $\mathbb{Z}/p$ is a trivial $GL_g(\mathbb{Z})$-module, we have an isomorphism:
$$
Hom\left(\Lambda^3H_p\otimes \Lambda^3 H_p;\;\mathbb{Z}/p\right)^{GL_g(\mathbb{Z})}\simeq Hom\left(\left(\Lambda^3H_p\otimes \Lambda^3 H_p\right)_{GL_g(\mathbb{Z})};\;\mathbb{Z}/p\right).$$
Then any $GL_g(\mathbb{Z})$-invariant bilinear form is completely determined by its values on the generators of $(\Lambda^3H_p\otimes \Lambda^3 H_p)_{GL_g(\mathbb{Z})}$ given in Proposition \ref{prop-coinv-tens-extp}. Computing the values of the bilinear forms given in the statement on these generators we get that they form a basis.
\end{proof}

Taking the antisymmetrization of the bilinear forms given in Proposition~\ref{prop:bilin_extp}, we get:

\begin{cor}
\label{cor-basis-antisym-extp}
Given an odd prime $p$ and an integer $g\geq 4,$ the antisymmetric bilinear forms $\Theta_g$, $Q_g$ and $(J_g-{}^t\!J_g)$ form a basis of the group
$$ Hom(\Lambda^3H_p\wedge\Lambda^3H_p;\mathbb{Z}/p)^{GL_g(\mathbb{Z})} .$$
\end{cor}

Finally, we show:

\begin{prop}
\label{prop-extp-unique-coc}
Given an odd prime $p$, the unique family of $2$-cocycles (up to a multiplicative constant) on $\Lambda^3 H_p$ whose pull-back along $\tau_1^Z$ on $\mathcal{M}_{g,1}[p]$ satisfies conditions (2) and (3) is given by the family of bilinear forms $({}^t\!J_g)_{g\geq 4}$. Moreover once we have fixed a common multiplicative constant the family of pulled-back cocycles satisfies also (1).
\end{prop}
\begin{proof}
By Proposition \ref{prop-ext-cocy-red}, any such $2$-cocycle that satisfies condition~($3'$) is a bilinear form.
By Lemma \ref{lem_im_A[p],B[p]}, the $GL_g(\mathbb{Z})$-invariant bilinear forms that satisfy condition~($3'$) are zero on the generators of $(\Lambda^3H_p\otimes \Lambda^3 H_p)_{GL_g(\mathbb{Z})}$ given in Proposition \ref{prop-coinv-tens-extp} except on the generator $(b_1\wedge b_2 \wedge b_3)\otimes(a_1\wedge a_2 \wedge a_3),$ and hence such bilinear forms are completely determined by their value on this generator. By construction, for each $g\geq 4$, the $GL_g(\mathbb{Z})$-bilinear form ${}^t\!J_g=-\Theta_g(\pi_B,\pi_A)$ satisfies condition ($3'$) and its value on the aforementioned generator is $3$, which is coprime with $p\geq 5.$
\end{proof}

\subsubsection{2-cocycles on $\mathfrak{sp}_{2g}(\mathbb{Z}/p)$.}
Similarly to $\Lambda^3 H_p,$ for each $g,$ the product of matrices composed with the trace map induces bilinear forms:
\begin{align*}
	K_g:Sym^A_g(\mathbb{Z}/p)\otimes Sym^B_g(\mathbb{Z}/p) & \rightarrow \mathbb{Z}/p ,\\
	{}^t\!K_g: Sym^B_g(\mathbb{Z}/p)\otimes Sym^A_g(\mathbb{Z}/p) & \rightarrow \mathbb{Z}/p,
\end{align*}
that we extend by zero to degenerate bilinear forms on 
\begin{equation}
\label{dec-sp-long-p}
\mathfrak{sp}_{2g}(\mathbb{Z}/p)\simeq \mathfrak{gl}_g(\mathbb{Z}/p)\oplus Sym^A_g(\mathbb{Z}/p)\oplus Sym^B_g(\mathbb{Z}/p).
\end{equation}
In all what follows, given an element $z\in \mathfrak{sp}_{2g}(\mathbb{Z}/p)$ we denote $z_{gl}$, $z_a$, $z_b$ the projections of $z$ on the respective components of above decomposition.

Unlike for $2$-cocycles on $\Lambda^3 H_p$, on $\mathfrak{sp}_{2g}(\mathbb{Z}/p)$, properties ($1'$)-($3'$) are not enough to ensure that a $2$-cocycle is a bilinear form. Nevertheless,

\begin{prop}
\label{prop-2coc-sp-gen}
Given an integer $g\geq 4$ and an odd prime $p,$ every $2$-cocycle on $\mathfrak{sp}_{2g}(\mathbb{Z}/p)$ that satisfies condition ($3'$) is a bilinear form up to an addition of
a $2$-cocycle on $\mathfrak{sp}_{2g}(\mathbb{Z}/p)$ pulled-back from $\mathbb{Z}/p$ along $tr \circ \pi_{gl}.$
\end{prop}

\begin{proof}
Let $B'_g$ denote an arbitrary $2$-cocycle on $\mathfrak{sp}_{2g}(\mathbb{Z}/p)$ satisfying  condition ($3'$) and $B''_g$ the $2$-cocycle given by the restriction of $B'_g$ to the $\mathfrak{gl}_{g}(\mathbb{Z}/p)$ sum\-mand and then extended by $0$ to degenerate a 2-cocycle on $\mathfrak{sp}_{2g}(\mathbb{Z}/p)$. By condition ($3'$) the 2-cocycle $B''_g$ is zero on $\mathfrak{sl}_{g}(\mathbb{Z}/p)\times \mathfrak{gl}_{g}(\mathbb{Z}/p)$, $\mathfrak{gl}_{g}(\mathbb{Z}/p)\times \mathfrak{sl}_{g}(\mathbb{Z}/p)$ and then a pull-back of a $2$-cocycle on $\mathbb{Z}/p$ by $(tr\circ \pi_{gl}).$
Next we show that $B_g=B'_g-B''_g$ is a bilinear form.

If we write each element $z\in \mathfrak{sp}_{2g}(\mathbb{Z}/p)$ as $z_{gl}+z_a+z_b$ according to the decomposition \eqref{dec-sp-long-p},
the cocycle relation together with condition ($3'$) imply that
\begin{equation*}
\forall x,y\in \mathfrak{sp}_{2g}(\mathbb{Z}/p),\qquad B_g(x,y)=B_g(x_{gl}+x_b,y_{gl}+y_a).
\end{equation*}
Then as in the proof of Point \textit{iii)} of Lemma \ref{lem:linearpropBg}, by the cocycle relation and the fact that $B_g$ is zero on $\mathfrak{gl}_{g}(\mathbb{Z}/p)$ one gets that $B_g$ is a bilinear form.
\end{proof}

We now compute the group of $GL_g(\mathbb{Z})$-invariant bilinear forms on the module $\mathfrak{sp}_{2g}(\mathbb{Z}/p).$

\begin{prop}
\label{prop:bilin_sp}
Given an odd prime $p$ and an integer $g\geq 4$, the bilinear forms
$$
\begin{aligned}
T^1_g(x,y)=tr(x_{gl}y_{gl}), &\qquad
T^2_g(x,y)=tr(x_{gl})tr(y_{gl}),\\
K_g(x,y)=tr(x_ay_b),& \qquad
{}^t\!K_g(x,y)=tr(x_by_a).
\end{aligned}
$$
form a basis of
$$ Hom(\mathfrak{sp}_{2g}(\mathbb{Z}/p)\otimes\mathfrak{sp}_{2g}(\mathbb{Z}/p);\mathbb{Z}/p)^{GL_g(\mathbb{Z})} .$$
\end{prop}

\begin{proof}
Since the trace map is $GL_g(\mathbb{Z})$-invariant, by construction the bilinear forms given in the statement are $GL_g(\mathbb{Z})$-invariant.
Because $\mathbb{Z}/p$ is a trivial $GL_g(\mathbb{Z})$-module, we have an isomorphism:
\begin{align*}
Hom\left(\mathfrak{sp}_{2g}(\mathbb{Z}/p)\otimes\mathfrak{sp}_{2g}(\mathbb{Z}/p);\;\mathbb{Z}/p\right)^{GL_g(\mathbb{Z})} &\simeq \\ Hom\left(\left(\mathfrak{sp}_{2g}(\mathbb{Z}/p)\otimes\mathfrak{sp}_{2g}(\mathbb{Z}/p)\right)_{GL_g(\mathbb{Z})};\;\mathbb{Z}/p\right). &
\end{align*}
As a consequence, any $GL_g(\mathbb{Z})$-invariant bilinear form is completely determined by its values on the generators of $(\mathfrak{sp}_{2g}(\mathbb{Z}/p)\otimes\mathfrak{sp}_{2g}(\mathbb{Z}/p))_{GL_g(\mathbb{Z})}$ given in Proposition~\ref{prop-coinv-tens-sp}. Computing the values of the given bilinear forms on these generators we get that they form a basis.
\end{proof}

Taking the antisymmetrization of the bilinear forms given in Proposition~\ref{prop:bilin_sp} we get:

\begin{cor}
\label{cor_sp_bil_anti}
Given an odd prime $p$ and an integer $g\geq 4,$ the antisymmetric bilinear form $({}^t\!K_g-K_g)$ generates the group
$$ Hom(\mathfrak{sp}_{2g}(\mathbb{Z}/p)\wedge\mathfrak{sp}_{2g}(\mathbb{Z}/p);\mathbb{Z}/p)^{GL_g(\mathbb{Z})}. $$
\end{cor}

Finally, we show:
\begin{prop}
\label{prop_cocy-sp}
Given an odd prime $p$, the families of bilinear forms $({}^t\!K_g)_{g\geq 4}$ and of $2$-cocycles on $\mathfrak{sp}_{2g}(\mathbb{Z}/p)$ lifted from $\mathbb{Z}/p$ along $tr\circ \pi_{gl}$ are the unique families of $2$-cocycles (up to linear combinations) on $\mathfrak{sp}_{2g}(\mathbb{Z}/p)$ whose pull-back along $\alpha\circ \Psi$ on $\mathcal{M}_{g,1}[p]$ satisfy conditions (2) and (3). Moreover once we have fixed a linear combination, the family of the pulled-back $2$-cocycles satisfies also (1).
\end{prop}

\begin{proof}
By Proposition \ref{prop-2coc-sp-gen},
we only need to search the linear combinations of families of bilinear forms on $\mathfrak{sp}_{2g}(\mathbb{Z}/p)$ and of $2$-cocycles on $\mathfrak{sp}_{2g}(\mathbb{Z}/p)$ pulled-back from $\mathbb{Z}/p$ along $tr \circ \pi_{gl},$ that satisfy conditions ($1'$)-($3'$).

Notice that these last families of 2-cocycles already satisfy conditions ($1'$)-($3'$), because
 $tr\circ \pi_{gl}:\; \mathfrak{sp}_{2g}(\mathbb{Z}/p)\rightarrow \mathbb{Z}/p$ is $GL_g(\mathbb{Z})$-invariant, sends $\mathfrak{sl}_g(\mathbb{Z}/p)\oplus Sym^A_g(\mathbb{Z}/p)\oplus Sym^B_g(\mathbb{Z}/p)$ to zero and is compatible with the stabilization map $\mathfrak{sp}_{2g}(\mathbb{Z}/p)\hookrightarrow \mathfrak{sp}_{2g+2}(\mathbb{Z}/p).$
As a consequence, the families of bilinear forms involved in the aforementioned linear combinations have to satisfy conditions ($1'$)-($3'$).

By Lemma \ref{lema:imA[d]B[d]sp}, the values of $GL_g(\mathbb{Z})$-invariant bilinear forms that satisfy condition ($3'$) on the generators of $(\mathfrak{sp}_{2g}(\mathbb{Z}/p)\otimes \mathfrak{sp}_{2g}(\mathbb{Z}/p))_{GL_g(\mathbb{Z})}$ given in Proposition \ref{prop-coinv-tens-sp} are zero on $u_{11}\otimes l_{11}$ and on $n_{11}\otimes n_{11}-n_{11}\otimes n_{22}.$ Hence, such bilinear forms are completely determined by their values on $n_{11}\otimes n_{11}$ and
$l_{11}\otimes u_{11}$. By construction, for each $g\geq 4$, the $GL_g(\mathbb{Z})$-bilinear forms ${}^t\!K_g,$ $T^2_g$ satisfy condition ($3'$) and their values on the aforementioned generators are $(0, 1)$ and $(1,0)$ respectively. Therefore a family of bilinear forms that satisfies conditions ($1'$)-($3'$) is a linear combination of $({}^t\!K_g)_{g\geq 4}$ and $(T^2_g)_{g\geq 4}$.

Finally, notice that by definition $T^2_g$ is a $2$-cocycle on $\mathfrak{sp}_{2g}(\mathbb{Z}/p)$ pulled-back from $\mathbb{Z}/p$ along $tr\circ \pi_{gl}$. Therefore we are only left with the 2-cocycles given in the statement.
\end{proof}

\subsection{Triviality of 2-cocycles and torsors}
\label{sec-triv-2coc-H1}

Once we know the families of $2$-cocycles on the abelianization of the level-$p$ mapping class group that satisfy conditions ($1'$)-($3'$), we check which of these families of 2-cocycles become trivial with trivial torsor when they are lifted to $\mathcal{M}_{g,1}[p]$, i.e. which of these lifted cocycles admit an $\mathcal{AB}_{g,1}$-invariant trivialization that can be made into an invariant.

We first show that the cohomology class of the unique candidate coming from $\Lambda^3H_p$, the $2$-cocycle $(\tau_1^Z)^*({}^t\!J_g)$, is non-trivial.
Then we check which $2$-cocycles that come from $\mathfrak{sp}_{2g}(\mathbb{Z}/p)$ do produce invariants of rational homology spheres in $\mathcal{S}^3[p]$.
Finally we show that it is not possible to obtain a trivial $2$-cocycle as a sum of two non-trivial $2$-cocycles coming from $\Lambda^3H_p$ and from $\mathfrak{sp}_{2g}(\mathbb{Z}/p)$ respectively.

\subsubsection{The 2-cocycles from $\Lambda^3H_p$.}
\label{sec-2coc-estp}

\begin{prop}
	\label{prop_H2_extp}
	Given an integer $g\geq 4$ and an odd prime $p$, the group
	$H^2(\Lambda^3H_p;\mathbb{Z}/p)^{GL_g(\mathbb{Z})}$ is generated by the bilinear forms $\Theta_g$, $Q_g$ and ${}^t\!J_g$.
\end{prop}

\begin{proof}
By considering the action of $-Id$ we get that $Hom(\Lambda^3H_p,\mathbb{Z}/p)^{GL_g(\mathbb{Z})}$ is zero. Since $\Lambda^3 H_p$ and $\mathbb{Z}/p$ are $p$-elementary abelian groups with $\Lambda^3 H_p$ acting trivially on $\mathbb{Z}/p$, the UCT (cf. Section \ref{sec-saunders}) gives us a natural isomorphism:
\[
\theta: \; H^2(\Lambda^3H_p; \mathbb{Z}/p)^{GL_g(\mathbb{Z})} \longrightarrow Hom(\Lambda^2(\Lambda^3H_p);\mathbb{Z}/p)^{GL_g(\mathbb{Z})}.
\]
Then, we conclude by  Corollary \ref{cor-basis-antisym-extp}.
\end{proof}

The following proposition shows that in particular the $2$-cocycle $(\tau_1^Z)^*Q_g$ is cohomologous to $(\tau_1^Z)^*(48\;{}^t\!J_g)$ and that this last 2-cocycle is non-trivial on
$\mathcal{M}_{g,1}[p].$

\begin{prop}
\label{prop-ker-im-tau}
Given an integer $g\geq 4$ and a prime $p\geq 5$, the image and the kernel of the pull-back
$$(\tau_1^Z)^*:H^2(\Lambda^3H_p;\mathbb{Z}/p)^{GL_g(\mathbb{Z})}\longrightarrow H^2(\mathcal{M}_{g,1}[p];\mathbb{Z}/p)^{GL_g(\mathbb{Z})}$$
are respectively generated by the cohomology class of $(\tau_1^Z)^*(Q_g)$ and by the cohomology classes of $(Q_g+6\Theta_g)$ and $(\Theta_g+8\;{}^t\!J_g)$.
\end{prop}

Before proving this proposition we need some notations and a preliminary result. To make our computations easier to handle we use the Lie algebra of labeled uni-trivalent trees to encode the first quotient in the Zassenhauss filtration (cf. \cite{levine}).

Let $\mathcal{A}_k(H_p)$ (resp. $\mathcal{A}_k^r(H_p)$) be the free abelian group generated by the uni-trivalent (resp. rooted) trees with $k+2$ univalent vertices labeled by elements
of $H_p$ and a cyclic order of each trivalent vertex modulo the relations $IHX$, $AS$ together with linearity of labels (cf. \cite{levine}).
We can endow $\mathcal{A}(H_p) := {\mathcal{A}_k(H_p)}_{k\geq 1}$ with a bracket operation
$$[\; \cdot \; ,\; \cdot \; ]: \mathcal{A}_k(H_p)\otimes \mathcal{A}_l(H_p)\rightarrow \mathcal{A}_{k+l}(H_p)$$
given below. For labeled trees $T_1,T_2\in \mathcal{A}(H_p)$  we define:
\begin{equation}
\label{eq:bracket}
[T_1,T_2]=\sum_{x,y}\omega(l_x, l_y)\;T_1-xy-T_2,
\end{equation}
where the sum is taken over all pairs of a univalent vertex $x$ of $T_1$, labeled by $l_x$, and $y$ of $T_2$, labeled by $l_y$, and $T_1-xy-T_2$ is the tree given by welding $T_1$ and $T_2$ at the pair.

Moreover we define the \textit{labeling map},
$$Lab:H_p\otimes \mathcal{A}_{k+1}^r(H_p)\rightarrow\mathcal{A}_k(H_p),$$
sending each elementary tensor $u\otimes T\in H_p\otimes \mathcal{A}_{k+1}^r(H_p)$ to the tree $T_u\in \mathcal{A}_k(H_p)$, which is obtained by labeling the root of $T$ by $u$, and extend it by linearity.
For more detailed explanations we refer the interested reader to Levine's paper \cite{levine}. In all what follows, to make the notation lighter, we set
$$H(a,b,c,d):=\tree{a}{b}{c}{d}\in \mathcal{A}_2(H_p).$$

Next we compute $Hom(\mathcal{A}_2(H_p),\mathbb{Z}/p)^{GL_g(\mathbb{Z})}.$

Consider the homomorphisms $d_1,d_2\in Hom(\mathcal{A}_2(H_p),\mathbb{Z}/p)$ given by
\begin{align*}
d_1(H(a,b,c,d))= & \;2\omega(a,b)\omega(d,c)+\omega(a,d)\omega(b,c)-\omega(a,c)\omega(b,d),\\
d_2(H(a,b,c,d))= & \;\overline{\omega}(a,d)\overline{\omega}(b,c)-\overline{\omega}(a,c)\overline{\omega}(b,d),
\end{align*}
where $\omega$ is the intersection form and $\overline{\omega}$ is the symmetric bilinear form associated to the matrix $\left(\begin{smallmatrix}
0 & Id \\
Id & 0
\end{smallmatrix}\right)$. By direct inspection,  the maps $d_1$ and $d_2$ are well defined, i.e. they are zero on the IHX, AS relations of $\mathcal{A}_2(H_p).$

\begin{lema} \label{lema_hom_2sym}
Given an integer $g\geq 3$ and $p$ an odd prime, the homomorphisms $d_1,d_2$ form a basis of
$$Hom(\mathcal{A}_2(H_p),\mathbb{Z}/p)^{GL_g(\mathbb{Z})}.$$
\end{lema}

\begin{proof}
Since the bilinear forms $\omega $ and $\overline{\omega}$
are $GL_g(\mathbb{Z})$-invariant, the homomorphisms given in the statement are also $GL_g(\mathbb{Z})$-invariant.
Since $\mathbb{Z}/p$ is a trivial $GL_g(\mathbb{Z})$-module, we have an isomorphism:
$$
Hom\left(\mathcal{A}_2(H_p);\;\mathbb{Z}/p\right)^{GL_g(\mathbb{Z})}\simeq Hom\left((\mathcal{A}_2(H_p))_{GL_g(\mathbb{Z})};\;\mathbb{Z}/p\right).$$ 
Then any $GL_g(\mathbb{Z})$-invariant homomorphism is completely determined by its values on the generators of $(\mathcal{A}_2(H_p))_{GL_g(\mathbb{Z})}$ given in Proposition \ref{prop-2sym-gen}. Computing the values of the given homomorphisms on these generators we get that they form a basis.
\end{proof}

\begin{proof}[Proof of Proposition \ref{prop-ker-im-tau}]

In the first half of the proof 
we show that the dimension of $\ker((\tau_1^Z)^*)$ is at least $2$ by providing two cocycles that belong to this kernel and write them in terms of the generators given in Proposition~\ref{prop_H2_extp}.

Recall that $Im(\tau_2^Z)\subset Hom(H_p,\mathcal{L}_{3}(H_p))$ for $p\geq 5.$ Moreover, similarly to \cite[Sec. 3.1]{pitsch4}, there are isomorphisms of $Sp_{2g}(\mathbb{Z}/p)$-modules:
$$\begin{aligned}
Hom(H_p,\mathcal{L}_{k+1}(H_p)) & \simeq H_p\otimes \mathcal{L}_{k+1}(H_p)\\
 & \simeq H_p\otimes \mathcal{A}_{k+1}^r(H_p) 
\end{aligned}
\qquad \text{and} \qquad
\Lambda^3 H_p  \simeq \mathcal{A}_1(H_p).$$

Consider the group extension
\[
\xymatrix{
0\ar[r] & Im(\tau_2^Z)\ar[r]& \rho_3^Z(\mathcal{M}_{g,1}) \ar[r] & \rho_2^Z(\mathcal{M}_{g,1}) \ar[r] & 1.
}
\]
 If we restrict this extension to $\rho^Z_2(\mathcal{M}_{g,1}[p])=\tau_1^Z(\mathcal{M}_{g,1}[p])= \Lambda^3 H_p $ we get another extension
$$\xymatrix@C=7mm@R=7mm{
0\ar[r] & Im(\tau_2^Z)\ar[r]& \rho_3^Z(\mathcal{M}_{g,1}[p]) \ar[r]^-{\psi_{3|2}} & \Lambda^3 H_p \ar[r] & 0
}.$$
Denote $\mathcal{X}_p$ a $2$-cocycle associated to this extension. By construction its cohomology class is $Sp_{2g}(\mathbb{Z}/p)$-invariant (cf. \cite[Sec. III.10]{brown}). Pushing-out $\mathcal{X}_p$ by $Lab: \;Im(\tau_2^Z)\rightarrow  \mathcal{A}_2(H_p)$ and subsequently by $d:\mathcal{A}_2(H_p)\rightarrow \mathbb{Z}/p$, the $GL_g(\mathbb{Z})$-invariant homomorphism $d_1$ or $d_2$ given in Lemma~\ref{lema_hom_2sym}, we get a 2-cocycle $(d\circ Lab)_*\mathcal{X}_p$ whose class belongs to $H^2(\Lambda^3H_p;\mathbb{Z}/p)^{GL_g(\mathbb{Z})}.$ In fact, this class belongs to the kernel of $(\tau_1^Z)^*:$

Since $\tau_1^Z=\psi_{3|2}\circ\rho^Z_3$, it is enough to show that the class of the 2-cocycle $(d\circ Lab)_*\mathcal{X}_p$ belongs to the kernel of $\psi_{3|2}^*.$
By construction, the cohomology class of $\psi_{3|2}^*\mathcal{X}_p$ is zero and the pull-back $\psi_{3|2}^*$ commutes with the push-out $(d\circ Lab)_*.$ Therefore we get that:
$$\psi_{3|2}^*((d\circ Lab)_*\mathcal{X}_p)=(d\circ Lab)_*(\psi_{3|2}^*\mathcal{X}_p)=0.$$

Next we find the expression of the cohomology classes of $(d_1\circ Lab)_*\mathcal{X}_p$ and $(d_2\circ Lab)_*\mathcal{X}_p$ in terms of the generators of $H^2(\Lambda^3H_p;\mathbb{Z}/p)^{GL_g(\mathbb{Z})}$ given in Proposition \ref{prop_H2_extp}. For such purpose we write the image of these cohomology classes by the natural isomorphism
$$\theta: \; H^2(\Lambda^3H_p; \mathbb{Z}/p)^{GL_g(\mathbb{Z})} \longrightarrow Hom(\Lambda^2(\Lambda^3H_p);\mathbb{Z}/p)^{GL_g(\mathbb{Z})}$$
in terms of the generators given in Corollary~\ref{cor-basis-antisym-extp}.

By naturality of $\theta$ the antisymmetric bilinear form $\theta(Lab_*\mathcal{X}_p)$ is equal to $Lab_*\theta(\mathcal{X}_p).$ By the computations done in \cite[Thm. 3.1]{mor1} modulo $p$, this last bilinear form is the bracket $[\; \cdot \; ,\; \cdot \; ]$ of the Lie algebra of labelled uni-trivalent trees given in \eqref{eq:bracket}.
 Explicitely, $[\; \cdot \; ,\; \cdot \; ]: \Lambda^3 H_p\wedge \Lambda^3 H_p\rightarrow \mathcal{A}_{2}(H_p)$ is given by:
$$[x_1 \wedge x_2\wedge x_3,y_1 \wedge y_2\wedge y_3]=$$
\begin{align*}
&\omega(x_1,y_1)H(x_2,x_3,y_2,y_3)+\omega(x_1,y_2)H(x_2,x_3,y_3,y_1) \\
&+\omega(x_1,y_3)H(x_2,x_3,y_1,y_2)  +\omega(x_2,y_1)H(x_3,x_1,y_2,y_3) \\
&+\omega(x_2,y_2)H(x_3,x_1,y_3,y_1)+\omega(x_2,y_3)H(x_3,x_1,y_1,y_2)  \\
&+\omega(x_3,y_1)H(x_1,x_2,y_2,y_3)+\omega(x_3,y_2)H(x_1,x_2,y_3,y_1) \\
& +\omega(x_3,y_3)H(x_1,x_2,y_1,y_2).
\end{align*}

Then, by naturality of $\theta,$ the antisymmetric bilinear form
$\theta(d_*Lab_*\mathcal{X}_p)$ is equal to $d_*\theta(Lab_*\mathcal{X}_p)=d_*[\; \cdot \; ,\; \cdot \; ].$ 
We evaluate these elements on the generators of $(\Lambda^2(\Lambda^3H_p))_{GL_g(\mathbb{Z})}$ given in Corollary \ref{cor-coinv-ext-extp},
\[
\begin{tabular}{c|c|c|c|c|c|} \cline{2-6}
& ${d_1}_*[\; \cdot \; ,\; \cdot \; ]$ & ${d_2}_*[\; \cdot \; ,\; \cdot \; ]$ & $\Theta_g$ & $Q_g$ & $J_g-{}^t\!J_g$\\ \hline
\multicolumn{1}{|c|}{$(a_1\wedge a_2 \wedge a_3) \wedge (b_1\wedge b_2 \wedge b_3)$} & $-3$ & $-3$ & $1$ & $0$ & $1$\\ \hline
\multicolumn{1}{|c|}{$(a_1\wedge a_2 \wedge b_2) \wedge (b_1\wedge a_2 \wedge b_2)$} & $-5$ & $1$ & $1$ & $4$ & $0$ \\ \hline
\multicolumn{1}{|c|}{$(a_1\wedge a_2 \wedge b_2) \wedge (b_1\wedge a_3 \wedge b_3)$} & $-2$ & $0$  & $0$ & $4$ & $0$\\ \hline
\end{tabular}
\]
From this table we get the following equalities:
$$
{d_1}_*[\; \cdot \; ,\; \cdot \; ]=  -3\Theta_g-\frac{1}{2}Q_g, \qquad
{d_2}_*[\; \cdot \; ,\; \cdot \; ]=  \;\Theta_g-4(J_g-{}^t\!J_g).
$$
Then by Proposition \ref{prop_H2_extp}, in $H^2(\Lambda^3H_p;\mathbb{Z}/p)^{GL_g(\mathbb{Z})}$ we have:
$$
\frac{1}{2}{d_1}_*[\; \cdot \; ,\; \cdot \; ]=  -\frac{3}{2}\Theta_g-\frac{1}{4}Q_g, \qquad
\frac{1}{2}{d_2}_*[\; \cdot \; ,\; \cdot \; ]=  \;\frac{1}{2}\Theta_g+4\;{}^t\!J_g.
$$
Multiplying these two cocycles by $-4$ and $2$ respectively, both invertible elements in $\mathbb{Z}/p$,  we get the cocycles given in the statement.

In the second half of the proof we show that the cohomology class of $(\tau_1^Z)^*(Q_g)$ is the image of a generator of 
\[
\mathbb{Z}/p=H^2(\mathcal{M}_{g,1};\mathbb{Z}/p)\hookrightarrow H^2(\mathcal{M}_{g,1}[p];\mathbb{Z}/p)
\]
(cf. Proposition~\ref{prop_inj_MGC}) and therefore is not zero; since $H^2(\Lambda^3H_p;\mathbb{Z}/p)^{GL_g(\mathbb{Z})}$ is a $3$-dimensional $\mathbb{Z}/p$-vector space, this will give us the result.

In \cite{mor_ext} S. Morita gave a crossed homomorphism
$k: \mathcal{M}_{g,1}\rightarrow \frac{1}{2}\Lambda^3 H$ that extends the first Johnson homomorphism $\tau_1:\mathcal{T}_{g,1}\rightarrow \Lambda^3 H.$ Using this crossed homomorphism, the intersection form $\omega$ and the contraction map $C$, in \cite{mor} S.~Morita defined a $2$-cocycle $\varsigma_g$ on $\mathcal{M}_{g,1}$  given by
$$\varsigma_g(\phi, \psi)=\omega((C\circ k)(\phi),(C\circ k)(\psi^{-1})),$$
whose cohomology class is $12$ times the generator of $H^2(\mathcal{M}_{g,1};\mathbb{Z})\simeq \mathbb{Z}$ (cf. \cite{harer}) and therefore a generator of $H^2(\mathcal{M}_{g,1};\mathbb{Z}/p)$ since $p \geq 5$ is coprime with $2$ and $3.$

The restriction to $\mathcal{M}_{g,1}[p]$ of the crossed homomorphism $k$ modulo $p$ and $\tau_1^Z$ are both extensions of the first Johnson homomorphism modulo $p$ to $\mathcal{M}_{g,1}[p]$, therefore by Proposition \ref{prop_unicity-ext-johnson} these extensions coincide and hence the restriction of the $2$-cocycle $\varsigma_g$ on $\mathcal{M}_{g,1}[p]$ coincides with $-(\tau_1^Z)^*(Q_g)$.
\end{proof}

For future reference we single out:
\begin{cor}\label{cor-extp-nontriv-cocy}
For an integer $g\geq 4$ and a prime $p \geq 5$,
the cohomology class of the 2-cocycle  $(\tau_1^Z)^*(48\;{}^t\!J_g)$ is not zero in 
$H^2(\mathcal{M}_{g,1}[p];\mathbb{Z}/p).$
\end{cor}

\subsubsection{The 2-cocycles from $\mathfrak{sp}_{2g}(\mathbb{Z}/p)$.}
\label{subsubsec:2-cocy-sp}

We start by inspecting the pull-back of trivial $2$-cocycles on $\mathbb{Z}/p$ to $\mathfrak{sp}_{2g}(\mathbb{Z}/p)$ along $(tr\circ \pi_{gl}).$

Given a trivial $2$-cocycle on $\mathbb{Z}/p$, if we pull-back this $2$-cocycle to $\mathfrak{sp}_{2g}(\mathbb{Z}/p)$ along $(tr\circ \pi_{gl})$ for each $g\geq 4$,
by Proposition \ref{prop_cocy-sp} we get a family of $2$-cocycles that satisfy conditions ($1'$)-($3'$) with a $GL_g(\mathbb{Z})$-invariant trivialization, because $(tr\circ \pi_{gl})$ is $GL_g(\mathbb{Z})$-invariant. Pulling back further to $\mathcal{M}_{g,1}[p]$ gives a family of $2$-cocycles satisfying all properties of Theorem~\ref{teo_cocy_p} and hence this family of $2$-cocycles provides an invariant of rational homology $3$-spheres.

Observe that the trivial $2$-cocycles on $\mathbb{Z}/p$ are exactly the coboundary of 
maps form $\mathbb{Z}/p$ to $\mathbb{Z}/p,$ which are formed by all polynomials of degree $p-1$ with coefficients in $\mathbb{Z}/p.$
And the map by which we lift these cocycles to $\mathcal{M}_{g,1}[p]$ is the invariant $\varphi$ given in Proposition~\ref{prop:stab_modp}. Therefore the aforementioned invariants of rational homology 3-spheres are given by all polynomials of degree $p-1$ with coefficients in $\mathbb{Z}/p$ and indeterminate the invariant $\varphi.$

We now inspect the other 2-cocycles that are candidates to produce invariants.

Consider the generator of $H^2(\mathbb{Z}/p;\mathbb{Z}/p)\simeq\mathbb{Z}/p$ that corresponds to the identity by the natural isomorphism $\nu:Ext_{\mathbb{Z}}^1(\mathbb{Z}/p,\mathbb{Z}/p)\rightarrow Hom(\mathbb{Z}/p,\mathbb{Z}/p)$ (cf. Section \ref{sec-saunders}).
Explicitly, this generator is the class of the $2$-cocycle $d:\mathbb{Z}/p\times\mathbb{Z}/p\rightarrow \mathbb{Z}/p$ given by the carrying in $p$-adic numbers,
\begin{equation}
\label{2-cocy-d}
d(x,y)=\left\lbrace\begin{array}{cc}
0, & \text{if} \quad x+y<p \\
1, & \text{if} \quad x+y\geq p.
\end{array}\right. 
\end{equation}
We prove that:

\begin{prop}
\label{gen-sp-coc-GL}
Given an integer $g\geq 4$ and an odd prime $p$, the group $H^2(\mathfrak{sp}_{2g}(\mathbb{Z}/p);\mathbb{Z}/p)^{GL_g(\mathbb{Z})}$ is generated by the cohomology classes of the $2$-cocycles ${}^t\!K_g$ and $(tr\circ \pi_{gl})^*d$.
\end{prop}

\begin{proof}
Since $\mathfrak{sp}_{2g}(\mathbb{Z}/p)$ and $\mathbb{Z}/p$ are $p$-elementary abelian groups and moreover $\mathfrak{sp}_{2g}(\mathbb{Z}/p)$ acting trivially on $\mathbb{Z}/p$, the UCT (cf. Section \ref{sec-saunders}) gives an isomorphism:
\begin{equation*}
\theta\;\oplus \;\nu:\;H^2(\mathfrak{sp}_{2g}(\mathbb{Z}/p);\mathbb{Z}/p)^{GL_g(\mathbb{Z})}\longrightarrow
\begin{array}{c}
Hom(\Lambda^2 \mathfrak{sp}_{2g}(\mathbb{Z}/p),\mathbb{Z}/p)^{GL_g(\mathbb{Z})} \\
 \oplus \\
 Hom(\mathfrak{sp}_{2g}(\mathbb{Z}/p),\mathbb{Z}/p)^{GL_g(\mathbb{Z})}.
\end{array}
\end{equation*}

We first compute the image of the class of the 2-cocycle $(tr\circ\pi_{gl})^*d$ by $\theta\oplus \nu.$ By construction $\nu(d)=id$, and by naturality of $\nu$ we get that
$$\nu((tr\circ\pi_{gl})^*d)=(tr\circ\pi_{gl})^*\nu(d)=(tr\circ\pi_{gl})^*id=tr\circ\pi_{gl}.$$
On the other hand, since $(tr\circ\pi_{gl})^*d$ is a symmetric 2-cocycle and $\theta$ is given by the antisymmetrization of cocycles, the image of its class by $\theta$ is zero.
Therefore,
$$(\theta\oplus \nu)((tr\circ\pi_{gl})^*d)=(0,tr\circ\pi_{gl}).$$

We now compute the image of the class of the bilinear form ${}^t\!K_g$ by $\theta \oplus \nu.$ By definition, $\theta(\;{}^t\!K_g)=(\;{}^t\!K_g-K_g).$ On the other hand,
in Section \ref{subsec-hom-tools} we proved that $\nu$ factors through the symmetrization of cocycles. Moreover, the class of any symmetric bilinear form $B$ with values in $\mathbb{Z}/p$ (with $p$ an odd prime) is zero, because there is a trivialization of $B$ given by $f(x)=\frac{1}{2}B(x,x)$. Then the image of any bilinear form by $\nu$ is zero too. Therefore,
$$(\theta\oplus \nu)(\;{}^t\!K_g)=(\;{}^t\!K_g-K_g,0).$$

We conclude by Lemma \ref{lem:spZ/dcoinv} and
Corollary~\ref{cor_sp_bil_anti}.
\end{proof}

Next we compute:

\begin{prop}
\label{ker_im_extp}
Given an integer $g\geq 4$ and an odd prime $p$, the kernel and the image of the pull-back
$$(\alpha\circ \Psi)^*:H^2(\mathfrak{sp}_{2g}(\mathbb{Z}/p);\mathbb{Z}/p)^{GL_g(\mathbb{Z})}\longrightarrow H^2(\mathcal{M}_{g,1}[p];\mathbb{Z}/p)^{GL_g(\mathbb{Z})}$$
are respectively generated by the cohomology class of ${}^t\!K_g-(tr\circ\pi_{gl})^*d$ and the image of the cohomology class of $(tr\circ\pi_{gl})^*d.$
\end{prop}

\begin{proof}
By Proposition \ref{prop_inj_MGC} the map 
\[\Psi^*: H^2(Sp_{2g}(\mathbb{Z},p);\mathbb{Z}/p)\rightarrow H^2(\mathcal{M}_{g,1}[p];\mathbb{Z}/p)\]
is injective. Then it is enough to compute the kernel and the image of the map
\[
\alpha^*: H^2(\mathfrak{sp}_{2g}(\mathbb{Z}/p);\mathbb{Z}/p)^{GL_g(\mathbb{Z})}\rightarrow H^2(Sp_{2g}(\mathbb{Z},p);\mathbb{Z}/p)^{GL_g(\mathbb{Z})}.
\]
From Proposition \ref{gen-sp-coc-GL} we know that $H^2(\mathfrak{sp}_{2g}(\mathbb{Z}/p);\mathbb{Z}/p)^{GL_g(\mathbb{Z})}$ is $2$-di\-men\-sio\-nal  as a $\mathbb{Z}/p$-vector space. We show first that the image of the class of the $2$-cocycle $(tr\circ\pi_{gl})^*d$ by $\alpha^*$ is not zero and second that the class of the $2$-cocycle ${}^t\!K_g-(tr\circ\pi_{gl})^*d$ belongs to the kernel of $\alpha^*$.

By naturality of the UCT, there is a commutative diagram with exact rows:
$$
\begin{tikzcd}
	0 \arrow[r] &  Ext^1_\mathbb{Z}(H_1(Sp_{2g}(\mathbb{Z},p);\mathbb{Z});\mathbb{Z}/p) \arrow[r] &  H^2(Sp_{2g}(\mathbb{Z},p); \mathbb{Z}/p) \arrow[r,dash] & {}\\
0 \arrow[r] &	Ext^1_\mathbb{Z}(\mathfrak{sp}_{2g}(\mathbb{Z}/p);\mathbb{Z}/p) \arrow[u] \arrow[r] & H^2(\mathfrak{sp}_{2g}(\mathbb{Z}/p); \mathbb{Z}/p) \arrow[u, "\alpha^*"']  \arrow[r,dash] & {} \\
 \arrow[r]& Hom(H_2(Sp_{2g}(\mathbb{Z},p);\mathbb{Z});\mathbb{Z}/p) \arrow[r] & 0  & {} \\
\arrow[r]& \arrow[u, "\alpha^*"']  Hom(\wedge^2\mathfrak{sp}_{2g}(\mathbb{Z}/p);\mathbb{Z}/p)  \arrow[r] &0 .& {}
\end{tikzcd}
$$
Since the class of the abelian $2$-cocycle $(tr\circ\pi_{gl})^*d$ belongs to the group $Ext^1_\mathbb{Z}(\mathfrak{sp}_{2g}(\mathbb{Z}/p);\mathbb{Z}/p),$ by commutativity of this diagram the pull-back of this $2$-cocycle to $Sp_{2g}(\mathbb{Z},p)$ is not trivial.

We now show that the cohomology class of ${}^t\!K_g-(tr\circ\pi_{gl})^*d$ belongs to the kernel of $\alpha^*$
by constructing a $2$-cocycle on $\mathfrak{sp}_{2g}(\mathbb{Z}/p)$ whose lift to $Sp_{2g}(\mathbb{Z},p)$ is trivial and subsequently identifying its cohomology class with the class of ${}^t\!K_g-(tr\circ\pi_{gl})^*d$.

Consider the  group extension given in \cite[Lemma 3.10]{Putman},
\[
\xymatrix@C=5mm@R=5mm{0 \ar@{->}[r] & \mathfrak{sp}_{2g}(\mathbb{Z}/p) \ar@{->}[r] & Sp_{2g}(\mathbb{Z}/p^3) \ar@{->}[r] & Sp_{2g}(\mathbb{Z}/p^2)  \ar@{->}[r] & 1}.
\]
Taking the restriction to $ \mathfrak{sp}_{2g}(\mathbb{Z}/p) = \ker \big(Sp_{2g}(\mathbb{Z}/p^2) \rightarrow Sp_{2g}(\mathbb{Z}/p)\big),$
we get another extension
\begin{equation}
\label{ses-sp-sp}
\xymatrix@C=7mm@R=7mm{0 \ar@{->}[r] & \mathfrak{sp}_{2g}(\mathbb{Z}/p) \ar@{->}[r]^-{i} & Sp_{2g}(\mathbb{Z}/p^3,p) \ar@{->}[r]^-{\alpha_{3|2}} & \mathfrak{sp}_{2g}(\mathbb{Z}/p)  \ar@{->}[r] & 0,}
\end{equation}
where $ Sp_{2g}(\mathbb{Z}/p^3,p) = \ker \big(Sp_{2g}(\mathbb{Z}/p^3) \rightarrow Sp_{2g}(\mathbb{Z}/p)\big)$.

Let $c$ be a $2$-cocycle associated to this extension. By construction its cohomology class is $Sp_{2g}(\mathbb{Z}/p)$-invariant (cf. \cite[Sec. III.10]{brown}). Then the push-out $c$ by the $GL_g(\mathbb{Z})$-invariant homomorphism $(tr\circ \pi_{gl}):\mathfrak{sp}_{2g}(\mathbb{Z}/p)\rightarrow \mathbb{Z}/p$ gives a $2$-cocycle $(tr\circ \pi_{gl})_*(c)$, whose cohomology class belongs to $H^2(\mathfrak{sp}_{2g}(\mathbb{Z}/p);\mathbb{Z}/p)^{GL_g(\mathbb{Z})}.$
In fact, this cohomology class belongs to the kernel of $\alpha^*$:
denote $\Psi_{p^3}$ the symplectic representation modulo $p^3.$
Since $\alpha=\alpha_{3|2} \circ \Psi_{p^3} $, it is enough to show that the class of the $2$-cocycle $(tr\circ \pi_{gl})_*(c)$ belongs to the kernel of $\alpha_{3|2}^*.$
By construction, the $2$-cocycle $\alpha_{3|2}^*(c)$ is trivial and the pull-back $\alpha_{3|2}^*$ commutes with the push-out $(tr\circ \pi_{gl})_*$. Therefore
$\alpha_{3|2}^*(tr\circ \pi_{gl})_*(c)=(tr\circ \pi_{gl})_*(\alpha_{3|2}^*(c))=0.$

Next we show that $-(tr\circ \pi_{gl})_*(c)$ is cohomologous to ${}^t\!K_g-(tr\circ\pi_{gl})^*d,$ identifying the images of their classes by the isomorphism (cf. Section~\ref{subsec-hom-tools}),
\begin{equation*}
\theta\;\oplus \;\nu:\;H^2(\mathfrak{sp}_{2g}(\mathbb{Z}/p);\mathbb{Z}/p)^{GL_g(\mathbb{Z})}\longrightarrow
\begin{array}{c}
Hom(\Lambda^2 \mathfrak{sp}_{2g}(\mathbb{Z}/p),\mathbb{Z}/p)^{GL_g(\mathbb{Z})} \\
 \oplus \\
 Hom(\mathfrak{sp}_{2g}(\mathbb{Z}/p),\mathbb{Z}/p)^{GL_g(\mathbb{Z})}.
\end{array}
\end{equation*}

By the proof of Proposition \ref{gen-sp-coc-GL},
$$(\theta\oplus \nu)(\;{}^t\!K_g-(tr\circ\pi_{gl})^*d)=(\;{}^t\!K_g-K_g,-tr\circ\pi_{gl}).$$

We now compute the image of $(tr\circ \pi_{gl})_*(c)$ by $\theta\oplus \nu.$
By naturality of $\theta$, a direct computation shows that
$\theta((tr\circ \pi_{gl})_* c)=(tr\circ \pi_{gl})_*\theta(c)=(tr\circ \pi_{gl})_*[\; \cdot \; ,\; \cdot \;],$ 
where $[\; \cdot \; ,\; \cdot \;]$ stands for the bracket of the Lie algebra $\mathfrak{sp}_{2g}(\mathbb{Z}/p)$.
By Corollary \ref{cor_sp_bil_anti}, $n\theta(\;{}^t\!K_g)=(tr\circ \pi_{gl})_*\theta(c)$ for some $n\in \mathbb{Z}/p$. Evaluating these homomorphism on $l_{11}\wedge u_{11}$, we get that
\begin{align*}
\theta((tr\circ \pi_{gl})_*c)(l_{11}\wedge u_{11}) &= (tr\circ \pi_{gl})\left[\left(\begin{smallmatrix}
0 & 0 \\
e_{11} & 0
\end{smallmatrix}\right),\left(\begin{smallmatrix}
0 & e_{11} \\
0 & 0
\end{smallmatrix}\right)\right] \\
&  = (tr\circ \pi_{gl})\left(\begin{smallmatrix}
-e_{11} & 0 \\
0 & e_{11}
\end{smallmatrix}\right)=-1,
\end{align*}
and
$
(\;{}^t\!K_g-K_g)(l_{11}\wedge u_{11})=1.
$
Therefore $n=-1$ and
$$\theta((tr\circ \pi_{gl})_*c)=-(\;{}^t\!K_g-K_g).$$
On the other hand, if we compute the image of $c$ by the natural homomorphism
$$\nu:H^2(\mathfrak{sp}_{2g}(\mathbb{Z}/p);\mathfrak{sp}_{2g}(\mathbb{Z}/p))\longrightarrow Hom(\mathfrak{sp}_{2g}(\mathbb{Z}/p),\mathfrak{sp}_{2g}(\mathbb{Z}/p)),$$ 
given $Id+p\widetilde{x}\in Sp_{2g}(\mathbb{Z}/p^3,p)$ a preimage of $x\in \mathfrak{sp}_{2g}(\mathbb{Z}/p)$ by $\alpha_{3|2},$ we get that
$
\nu(c)(x)=i^{-1}((Id+p\widetilde{x})^p) =i^{-1}(Id+p^2\widetilde{x})=x.$ Then $\nu(c)=id$ and by naturality of $\nu,$
$$\nu((tr\circ \pi_{gl})_*c)=(tr\circ \pi_{gl})_*\nu(c)=(tr\circ \pi_{gl})_*id=tr\circ \pi_{gl}.$$
\end{proof}

\begin{prop}
Given an integer $g\geq 4$ and an odd prime $p$, the torsor of the $2$-cocycle $(\alpha\circ \Psi)^*(\;{}^t\!K_g-(tr\circ \pi_{gl})^*d)$ is trivial.
\end{prop}

\begin{proof}
To make the notation lighter, let $C_g$ be the $2$-cocycle $(\alpha\circ \Psi)^*(\;{}^t\!K_g-(tr\circ \pi_{gl})^*d).$ By construction the torsor class $\rho(C_g):H_1(\mathcal{AB}_{g,1}[p];\mathbb{Z})\otimes (\mathfrak{sp}_{2g}(\mathbb{Z}/p)\oplus \Lambda^3 H_p)\rightarrow \mathbb{Z}/p$ (cf. Proposition \ref{prop:torsorforprimep}) can be described as follows. Fix an arbitrary trivialization $q_g$ of $C_g$. For each tensor $f\otimes l\in \mathcal{AB}_{g,1}[p]\otimes (\mathfrak{sp}_{2g}(\mathbb{Z}/p)\oplus \Lambda^3 H_p)$, choose arbitrary lifts $\phi\in \mathcal{AB}_{g,1}[p]$ and $\lambda\in \mathcal{M}_{g,1}[p].$ Since $C_g$ is the coboundary of $q_g$, we get that
\begin{align*}
\rho(C_g)(f\otimes l)= & q_g(\phi \lambda\phi^{-1})-q_g(\lambda)=q_g(\phi \lambda\phi^{-1}\lambda^{-1})-C_g(\phi \lambda\phi^{-1}\lambda^{-1},\lambda)= \\
 = & q_g(\phi \lambda\phi^{-1}\lambda^{-1})= q_g(\phi \lambda)-q_g(\lambda\phi) +C_g(\phi \lambda\phi^{-1}\lambda^{-1},\lambda\phi)=\\
 = & q_g(\phi \lambda)-q_g(\lambda\phi) = -C_g(\phi,\lambda)+C_g(\lambda,\phi)=0,
\end{align*}
where the last equality follows as $C_g$ satisfies condition (3).
\end{proof}

Therefore, as we wanted, the $2$-cocycle $(\alpha\circ \Psi)^*(\;{}^t\!K_g-(tr\circ \pi_{gl})^*d)$ satisfies all hypothesis of Theorem \ref{teo_cocy_p} and hence, this $2$-cocycle induces $p$ invariants of rational homology $3$-spheres $\mathcal{S}^3[p]$, which differ by multiples of the invariant $\varphi$ given in Proposition \ref{prop:stab_modp}. Denote by $\mathcal{R}$ the invariant associated to the aforementioned 2-cocycle that takes zero value on the Lens space $L(1+2p+2p^2,0)$. In Section \ref{subs:Desc.-inv} we give an explicit description of this invariant.

\subsubsection{Mixing 2-cocycles from the abelianization of $\mathcal{M}_{g,1}[p]$.}

Finally, we show that there does not exist a pair of families of $2$-cocycles from $\Lambda^3 H_p$ and $\mathfrak{sp}_{2g}(\mathbb{Z}/p)$ respectively, satisfying conditions ($1'$)-($3'$), such that their lift to $\mathcal{M}_{g,1}[p]$ is not trivial but the lift of their sum is.
By Proposition \ref{prop-extp-unique-coc} and Corollary \ref{cor-extp-nontriv-cocy}
it is enough to show that there does not exist a $2$-cocycle $C_g$ on $\mathfrak{sp}_{2g}(\mathbb{Z}/p)$ such that $(\alpha\circ \Psi)^*C_g$ is cohomologous to $(\tau_1^Z)^*(\;{}^t\!J_g).$

Assume that such $2$-cocycle exists.
By Proposition \ref{prop_inj_MGC} there is a commutative diagram where all unmarked coefficients are $\mathbb{Z}/p$:

\[
\begin{tikzcd}
	& H^2(Sp_{2g}(\mathbb{Z},p^2)) & \\
	H^2(Sp_{2g}(\mathbb{Z})) \ar[ur, hook, end anchor=south west] \ar[r,hook] \ar[d, sloped,"\sim"] \ar[d,"\Psi^\ast"'] & H^2(Sp_{2g}(\mathbb{Z},p)) \ar[u] \ar[d, hook, "\Psi^\ast"' ] & H^2(\mathfrak{sp}_{2g}(\mathbb{Z}/p)) \ar[ul, "0"', end anchor = south east] \ar[l, "\alpha^\ast"'] \ar[d] \ar[dl, end anchor= north east] \\
	H^2(\mathcal{M}_{g,1}) \ar[r,hook] & H^2(\mathcal{M}_{g,1}[p]) &  H^2(H_1(\mathcal{M}_{g,1}[p];\mathbb{Z})). \ar[l] \\
	& & H^2(\Lambda^3 H_p) \ar[ul, "(\tau_1^Z)^\ast", end anchor= south east] \ar[u]
\end{tikzcd}
\]
As we have seen in Section \ref{sec-2coc-estp}, the $2$-cocycle $(\tau_1^Z)^*(\;{}^t\!J_g)$ is cohomologous to the restriction of a generator of $ H^2(\mathcal{M}_{g,1};\mathbb{Z}/p)\simeq \mathbb{Z}/p$  to $\mathcal{M}_{g,1}[p]$. 
Then by left-hand side commutative square there exists a $2$-cocycle $\mu$ on $Sp_{2g}(\mathbb{Z},p)$ such that $\Psi^*(\mu)$ is cohomologous to $(\tau_1^Z)^*(\;{}^t\!J_g),$ and
since $\Psi^*$ is injective, $\mu$ is cohomologous to $\alpha^*(C_g).$
But, by commutativity of the left upper triangle, $\mu$ restricted to $Sp_{2g}(\mathbb{Z},p^2)$ is not cohomologous to $0$,
whereas by commutativity of the right upper triangle, the $2$-cocycle $\alpha^*(C_g)$ vanishes on $Sp_{2g}(\mathbb{Z},p^2)$. Hence we get a contradiction.

To sum up, we have the following result:

\begin{teo}
\label{teo-main}
Given a prime number $p\geq 5,$ the invariants of rational homology spheres in $\mathcal{S}^3[p]$ induced by families of $2$-cocycles on the abelianization of the level-$p$ mapping class group are given by all the linear combinations of the invariant $\mathcal{R}$ and the powers of the invariant $\varphi.$
\end{teo}

\subsection{Description of the invariants}
\label{subs:Desc.-inv}

In this section, given an integer $d\geq 3$, we construct a family of functions $\mathcal{R}_g:\mathcal{M}_{g,1}[d]\rightarrow \mathbb{Z}/d$ that reassemble into an invariant $\mathcal{R}$ of rational homology spheres in $\mathcal{S}^3[d]$ which coincides with the invariant given at the end of Section \ref{subsubsec:2-cocy-sp} for $d$ a prime number $p\geq 5$.
Then we prove that the invariants $\varphi$ and $\mathcal{R}$ are homological invariants. To be more precise, for a fixed positive integer $d\neq 2$, given $M\in \mathcal{S}^3[d]$ and $n=|H_1(M;\mathbb{Z})|$, we show that the reduction modulo $d$ of $n$ and the invariants $\varphi$, $\mathcal{R}$ are completely determined by the first three coefficients of the $d$-adic expansion of $n$ and viceversa. As a consequence of this characterization, we will show that the invariants $\varphi$ and $\mathcal{R}$ give an obstruction for a rational homology sphere $M\in \mathcal{S}^3[d]$ to belong to $\mathcal{S}^3[d^2]$ and $\mathcal{S}^3[d^3]$.

In order to give an explicit description of the functions $\mathcal{R}_g$, we first introduce some definitions and elementary results about the $d$-adic integers and the classic Faddeev-LeVerrier algorithm to compute the determinant of a matrix of the form $Id+dA$.

Denote by $\mathbb{Z}_d$ the ring of $d$-adic integers. There is an isomorphism $$\mathbb{Z}/d^k\mathbb{Z} \longrightarrow \mathbb{Z}_d/d^k\mathbb{Z}_d$$
that sends an element $a\in \mathbb{Z}/d^k\mathbb{Z}$ to the following series expansion:
$$a_0+d a_1+d^2a_2+\cdots d^{k-1} a_{k-1},$$
with $a_i\in \lbrace 0, 1, \ldots , d-1 \rbrace$ for all $i$.
Given positive integers $k,l$ with $l<k$, we define the following projection maps:
\begin{align*}
r_l: \mathbb{Z}/d^k\mathbb{Z} &\rightarrow \mathbb{Z}/d\mathbb{Z} \\
a & \mapsto a_l
\end{align*}
In this section we only use this projection map for $l=1,2$. 
Notice that these projection maps are not homomorphisms. Nevertheless, given elements $a\in \mathbb{Z}/d^k$ and $b\in \mathbb{Z}/d^{k-1}$, we have that
\begin{equation}
\label{eq:property r1}
r_1(a+db)=r_1(a)+b \; (mod \;d).
\end{equation}

Now we introduce the classic Faddeev-LeVerrier algorithm (cf.\cite{reut}), which tells us that the characteristic polynomial of a $g\times g$ matrix $A$ is given by:
\begin{equation}
\label{eq:char-pol}
p_A(\lambda)=det(\lambda Id -A)=\sum_{k=0}^{g} c_k \lambda^{k},
\end{equation}
where the coefficients $c_k$ are given inductively, from top to bottom, by:
$$c_g=1\qquad \text{and} \qquad c_{g-m}=-\frac{1}{m} \sum_{k=1}^m c_{g-m+k} tr(A^k).$$

Dividing equation \eqref{eq:char-pol} by $\lambda^g$, replacing $A$ by $-A$ and setting $\lambda=1/d$ we get that the determinant of a $g\times g$ matrix of the form $Id+dA$ can be written as the following polynomial of degree $g$ and indeterminate $d$:
\begin{equation}
\label{eq:det-expan}
det(Id +d A)=\sum_{k=0}^{g} (-1)^{g-k} c_k \;d^{g-k}=\sum_{k=0}^{g} (-1)^{k} c_{g-k} \;d^{k}=\sum_{k=0}^{g} s_k \;d^{k},
\end{equation}
with $s_k=(-1)^{k}c_{g-k}$.
In particular,
$$s_0=1, \qquad s_1=tr(A) \qquad \text{and} \qquad s_2=\frac{1}{2}(tr(A)^2-tr(A^2)).$$

Now we are ready to give an explicit description of functions $\mathcal{R}_g$.
In sequel, given integers $d\geq 3$ and $k\geq 2$, we denote by $M_g(\mathbb{Z}/d^k)$ the set of $g\times g$ matrices with coefficients in $\mathbb{Z}/d^k$ and by $M_g(\mathbb{Z}/d^k,d)$ the subset of matrices of the form $Id+dA$ with $A\in M_g(\mathbb{Z}/d^{k-1})$. 

Consider the function given by the following composition of maps:
\begin{equation}
\label{eq:def-R}
\xymatrix{
\mathcal{R}_g:\mathcal{M}_{g,1}[d] \ar[rr]^-{Sp\; rep.}_-{mod \;d^3} && Sp_{2g}(\mathbb{Z}/d^3,d) \ar[r]^-{pr_1} & M_g(\mathbb{Z}/d^3,d) \\ 
& {} \ar[r]^-{det} & \mathbb{Z}/d^3\mathbb{Z} \ar[r]^-{r}&\mathbb{Z}/d,
}
\end{equation}
where
$pr_1\left(\begin{smallmatrix}
A & B \\
C & D
\end{smallmatrix}\right)=D$ and $r(x)=r_2(x)-\frac{1}{2}r_1(x)^2.$

Notice that the maps $\mathcal{R}_g$ are compatible with the stabilization map, $\mathcal{AB}_{g,1}$-invariant, constant on the double cosets of
$\mathcal{A}_{g,1}[d]\backslash \mathcal{M}_{g,1}[d]/\mathcal{B}_{g,1}[d]$ and zero on the 3-sphere $\mathbf{S}^3$, hence they reassemble into a normalized invariant $\mathcal{R}$ of rational homology spheres in $\mathcal{S}^3[d]$.

In all what follows, to compute the image of $\mathcal{R}_g$ we will often use the expansion series of the determinant given in \eqref{eq:det-expan}, which allows us to rewrite the map $r \circ det: M_g(\mathbb{Z}/d^3,d) \rightarrow \mathbb{Z}/d$ as the composition of maps:
\begin{equation}
\label{eq:comp-R}
\xymatrix@C=8mm@R=1mm{
M_g(\mathbb{Z}/d^3,d) \ar[r]^-{pr_2} & M_g(\mathbb{Z}/d^2) \ar[r]^-{T} & \mathbb{Z}/d.
\\
D=Id+dD_1 \ar[r]& D_1 \ar[r]& r_1(tr(D_1))-\frac{1}{2}tr(D_1^2)}
\end{equation}

We now show that, when $d$ is a prime $p\geq 5$ the invariant $\mathcal{R}$ coincides with the invariant given at the end of Section \ref{subsubsec:2-cocy-sp}. More precisely,

\begin{prop}
\label{prop-new-inv}
Given an integer $g\geq 4$ and a prime $p\geq 5$, the coboundary of $\mathcal{R}_g$ is the $2$-cocycle $\Psi_{p^2}^*(\;{}^t\!K_g-(tr\circ \pi_{gl})^*d)$.
\end{prop}

\begin{proof}
Denote by $\overline{\mathcal{R}}_g: Sp_{2g}(\mathbb{Z}/p^3,p)\rightarrow \mathbb{Z}/p$ the map that when pulled-back along the mod $p^3$ symplectic representation coincides with $\mathcal{R}_g$. 
Observe that $\Psi_{p^2}$ can be decomposed as $\Psi_{p^2}=\alpha_{3|2}\circ\Psi_{p^3}.$
Therefore to prove the first part of the statement it is enough to show that the coboundary of $ \overline{\mathcal{R}}_g$ is the 2-cocycle $\alpha_{3|2}^*(\;{}^t\!K_g-(tr\circ \pi_{gl})^*d).$ 

Let $X=\left(\begin{smallmatrix}
A & B \\
C & D
\end{smallmatrix}\right)$, $Y=\left(\begin{smallmatrix}
E & F \\
G & H
\end{smallmatrix}\right)$ be elements of $Sp_{2g}(\mathbb{Z}/p^3,p)$. Then
$$ D=Id_g+pD_1, \quad C=pC_1, \quad H=Id_g+pH_1, \quad F=pF_1,$$
where $D_1,H_1,C_1$ and $F_1$ are matrices with coefficients in $\mathbb{Z}/p^2$. 
Denote by $\overline{D}_1,\overline{H}_1,\overline{C}_1$ and $\overline{F}_1$ 
the reduction modulo $p$ of these matrices.

The product $XY=\left(\begin{smallmatrix}
A & B \\
C & D
\end{smallmatrix}\right)\left(\begin{smallmatrix}
E & F \\
G & H
\end{smallmatrix}\right)$ is given by $\left(\begin{smallmatrix}
* & * \\
* & CF+DH
\end{smallmatrix}\right)$, with
\begin{align*}
CF+DH= & Id_g+p(D_1+H_1)+p^2(\overline{C}_1\overline{F}_1+\overline{D}_1\overline{H}_1) \\
= & Id_g+p(D_1+H_1+p(\overline{C}_1\overline{F}_1+\overline{D}_1\overline{H}_1)).
\end{align*}
Using equation \eqref{eq:property r1} and the fact that for any pair of matrices $X,$ $Y$ the equality $tr(XY)=tr(YX)$ holds, we have that
$$\overline{\mathcal{R}}_g(X)+\overline{\mathcal{R}}_g(Y)-\overline{\mathcal{R}}_g(XY)=$$
\begin{align*}
= & r_1\left(tr(D_1)\right)-\dfrac{1}{2}tr(\overline{D}_1^2)+r_1\left( tr(H_1)\right)-\dfrac{1}{2}tr(\overline{H}_1^2) \\
& -r_1\left(tr(D_1)+tr(H_1)+p\;tr(\overline{C}_1\overline{F}_1)+p\;tr(\overline{D}_1\overline{H}_1)\right)+\dfrac{1}{2}tr((\overline{D}_1+\overline{H}_1)^2) \\
= & r_1( tr(D_1)) +r_1( tr(H_1))-r_1(tr(D_1)+tr(H_1))-tr(\overline{C}_1\overline{F}_1)) \\
= & d(tr(D_1),tr(H_1))-tr(\overline{C}_1\overline{F}_1) \\
= & (tr\circ \pi_{gl}\circ \alpha_{3|2})^*d(X,Y)-\alpha_{3|2}^*\;{}^t\!K_g(X,Y).
\end{align*} 

%
\end{proof}

We now relate more precisely our invariants to the $d$-adic expansion of the first homology group of the involved homology sphere, for an integer $d \geq 3$. Let $M$ be a rational homology sphere in $\mathcal{S}^3[d]$, set $n=|H_1(M;\mathbb{Z})|$. The positive integer $n$ is an invariant of rational homology spheres and, for any $k\geq 1$, its reduction modulo $d^k$ is an invariant of rational homology spheres as well.
Remember that an element $n\in \mathbb{Z}/d^k\mathbb{Z}$ is uniquely written as the following series expansion:
\begin{equation}
\label{eq:series-p-adic}
n=n_0+dn_1+d^2n_2+\cdots d^{k-1} n_{k-1},
\end{equation}
with $n_i\in \lbrace 0,\ldots d-1 \rbrace$ for all $i\in \mathbb{N}$.
As a consequence, the coefficients $n_i$ provide $\mathbb{Z}/d$-valued invariants of rational homology spheres.

In sequel we show that the invariant given by the reduction modulo $d$ of $n$, the invariant $\varphi$ and the invariant $\mathcal{R}$ are completely determined by the first three coefficients of the $d$-adic expansion of $n$ given in \eqref{eq:series-p-adic} and viceversa.

Remember that given a homology sphere $M\in \mathcal{S}^3[d]$, there exist $f\in \mathcal{M}_{g,1}[d]$ such that $M\simeq S^3_f$ with $H_1(f;\mathbb{Z})=Id+dX$ where
$X=\left(\begin{smallmatrix}
A & B \\
C & D 
\end{smallmatrix}\right).$
In this case, by Lemma \ref{lema_n_detH}, $n=|H_1(M;\mathbb{Z})|=|det(Id+dD)|$ and by Proposition \ref{prop:abinv_modp}, $\varphi(M)=tr(A)=-tr(D)\;(\text{mod } d).$

Since we are working with invariants of rational homology spheres in $\mathcal{S}^3[d]$, by Theorem \ref{teo_rat_homology_gen}, the reduction modulo $d$ of $n$ is $1$ or $-1$. Next we treat these two cases separately.
\begin{itemize}
\item If $n\equiv 1\; (\text{mod }d)$, then $n=det(Id+dD)$. Consider the following $d$-adic expansion of $n$ modulo $d^3$:
$$n=n_0+dn_1+d^2n_2.$$
Using the series expansion of the determinant \eqref{eq:det-expan} reduced modulo $d^3$ we get the following equalities in $\mathbb{Z}/d^3$:
\begin{align*}
n & = 1+d \;tr(D)+\frac{d^2}{2}(tr(D)^2-tr(D^2)) \\
 & = 1+d \;r_0(tr(D))+d^2(r_1(tr(D))+\frac{1}{2}tr(D)^2-\frac{1}{2}tr(D^2)).
\end{align*}
By the uniqueness of a $d$-adic expansion we get the following equalities in $\mathbb{Z}/d$:
$$n_0=1,\quad n_1= tr(D),\quad n_2 = r_1(tr(D))+\frac{1}{2}tr(D)^2-\frac{1}{2}tr(D^2).$$
Therefore we get the following equalities in $\mathbb{Z}/d$:
\begin{equation}
\label{eq:p-adic1}
n_0=|H_1(M;\mathbb{Z})|,\quad n_1= -\varphi(M),\quad n_2 = \mathcal{R}(M)+\frac{1}{2}\varphi(M)^2.
\end{equation}

\item If $n\equiv -1\; (\text{mod }d)$, then $-n=det(Id+dD)$.
Given a $d$-adic expansion $n_0+dn_1+d^2n_2$ of $n$ modulo $d^3$, the $d$-adic expansion of $-n$ modulo $d^3$ is given by
$$-n=(d-n_0)+d(d-n_1-1)+d^2(d-n_2-1)\quad (\text{mod }d^3)$$
As in the previous point, using the series expansion of the determinant \eqref{eq:det-expan} reduced modulo $d^3$ and the uniqueness of a $d$-adic expansion we get the following equalities in $\mathbb{Z}/d$:
$$d-n_0=1,\quad d-n_1-1= tr(D),$$
$$d-n_2-1 = r_1(tr(D))+\frac{1}{2}tr(D)^2-\frac{1}{2}tr(D^2).$$
Therefore we get the following equalities in $\mathbb{Z}/d$:
\begin{equation}
\label{eq:p-adic-1}
\begin{aligned}
n_0=|H_1(M;\mathbb{Z})|,\quad n_1= \varphi(M)-1, \\
n_2 = -\mathcal{R}(M)-\frac{1}{2}\varphi(M)^2-1.
\end{aligned}
\end{equation}

\end{itemize}

Then, given $n=n_0+dn_1+d^2n_2$ the $d$-adic expansion of $n=|H_1(M;\mathbb{Z})|$ modulo $d^3,$ the equalities \eqref{eq:p-adic1}, \eqref{eq:p-adic-1} can be rewritten together as the following equalities in $\mathbb{Z}/d$:
\begin{equation}
\label{eq:p-adic-general}
\begin{aligned}
n_0=|H_1(M;\mathbb{Z})|,\quad n_1= -n_0\varphi(M)+\frac{n_0-1}{2}, \\
n_2 = n_0\mathcal{R}(M)+\frac{1}{2}n_0\Big(n_0n_1+\dfrac{n_0-1}{2}\Big)^2+\frac{n_0-1}{2}.
\end{aligned}
\end{equation}
From these equalities we get that in $\mathbb{Z}/d$,
\begin{equation}
\label{eq:p-adic to invariants}
\begin{aligned}
|H_1(M;\mathbb{Z})|=n_0, \quad \varphi(M)= -n_0n_1-\frac{n_0-1}{2}, \\
\mathcal{R}(M)=n_0n_2 +\frac{n_0-1}{2}-\frac{1}{2}\Big(n_0n_1+\dfrac{n_0-1}{2}\Big)^2.
\end{aligned}
\end{equation}
From these formulae we get the following result:

\begin{prop}
Given an integer $d\geq 3$, the invariants $\varphi,$ $\mathcal{R}$ of rational homology spheres in $\mathcal{S}^3[d]$ are homological invariants, i.e. if $M_1,M_2\in \mathcal{S}^3[d]$ with $H_1(M_1;\mathbb{Z})=H_1(M_2;\mathbb{Z})$, then $\varphi(M_1)=\varphi(M_2)$ and $\mathcal{R}(M_1)=\mathcal{R}(M_2)$.
\end{prop}

Using the same formulae we also get that the invariants $\varphi$ and $\mathcal{R}$ give us an obstruction of a rational homology sphere in $\mathcal{S}^3[d]$ to belong to the second or third levels of the filtration 
of rational homology spheres
\begin{equation}
\label{eq:filt-p}
\mathcal{S}^3[d]\supset \mathcal{S}^3[d^2]\supset \mathcal{S}^3[d^3]\supset \cdots \supset \mathcal{S}^3[d^k]\supset \cdots 
\end{equation}

To be more precise, 

\begin{prop}
\label{prop-osbtr-d}
Given an integer $d\geq 3$, a rational homology sphere $M\in \mathcal{S}^3[d]$ belongs to $\mathcal{S}^3[d^2]$ if and only if $\varphi(M)=0$, and belongs to $\mathcal{S}^3[d^3]$ if and only if $\varphi(M)=0$ and $\mathcal{R}(M)=0$.
\end{prop}

\begin{proof}
Given a homology sphere $M\in \mathcal{S}^3[d]$ and $n=|H_1(M;\mathbb{Z})|$, then $M$ belongs to $\mathcal{S}^3[d^2]$ if and only if the first two digits of the $d$-adic expansion of $n$ are $(1,0)$ for $n\equiv 1 \;(\text{mod } d)$ and $(d-1,d-1)$ for $n\equiv -1 \;(\text{mod } d)$.
By equations \eqref{eq:p-adic-general} in both cases we have that $\varphi(M)=0$. Similarly, we have that $M$ belongs to $\mathcal{S}^3[d^3]$ if and only if the first three digits of the $d$-adic expansion of $n$ are $(1,0,0)$ for $n\equiv 1 \;(\text{mod } d)$ and $(d-1,d-1,d-1)$ for $n\equiv -1 \;(\text{mod } d)$. Again we conclude by equations \eqref{eq:p-adic-general}.
\end{proof}

Moreover observe that the invariant $\mathcal{R}$ defined on the first level of filtration \eqref{eq:filt-p} and the invariant $\varphi$ defined on the second level of the same filtration are closely related:
\begin{prop}
\label{prop:lift-varphi-R}
Given an integer $d\geq 3$, we have a commutative diagram
\begin{center} 
\begin{tikzcd}
	\mathcal{S}[d] \ar[drr,"\mathcal{R}"]  & &   \\
	\mathcal{S}[d^2] \ar[r,"-\varphi"] \ar[u, hook] & \mathbb{Z}/d^2 \ar[r] & \mathbb{Z}/d. 
\end{tikzcd}
\end{center}
\end{prop}

\begin{proof}
Given a homology sphere $M\in \mathcal{S}^3[d^2]$ and $n=|H_1(M;\mathbb{Z})|$, consider $n_0+dn_1+d^2n_2$ the $d$-adic expansion of $n$ in $\mathbb{Z}/d^3$. By Proposition \ref{prop-osbtr-d}, the map $\varphi:\mathcal{S}^3[d]\rightarrow \mathbb{Z}/d$ sends $M$ to zero and by equations \eqref{eq:p-adic to invariants} we have that
$$\mathcal{R}(M)=n_0n_2 +\frac{n_0-1}{2}.$$
Then, by equations \eqref{eq:p-adic to invariants} replacing $d$ by $d^2$, the value $\mathcal{R}(M)$ coincides with the image of $M$ through the map $-\varphi:\mathcal{S}^3[d^2]\rightarrow \mathbb{Z}/d^2$ reduced modulo $d$.
\end{proof}

Applying formulae \eqref{eq:p-adic to invariants} we can also rewrite Theorem \ref{teo-main} as follows:

\begin{teo}\label{teo-main-p-adic-inv}
Given a prime number $p\geq 5$, the invariants of homology $3$-spheres in $\mathcal{S}^3[p]$ induced by families of $2$-cocycles on the abelianization of the level-$p$ mapping class group are homological invariants. More precisely, given $M\in \mathcal{S}^3[p]$ and $n_0, n_1, n_2\in \mathbb{Z}/p$ the first three coefficients of the $p$-adic expansion of $n=|H_1(M;\mathbb{Z})|$, the following functions form a basis for the space of the aforementioned invariants,
\[
 \mathcal{P}= n_0n_2+\frac{n_0-1}{2} \text{ and } \varphi^k= \Big(n_0n_1+\frac{n_0-1}{2}\Big)^k \quad \text{with } k=1,\ldots p-1.
\]
\end{teo}

\begin{rem}
Clearly each digit in the $p$-adic expansion of $|H_1(S^3_\phi,\mathbb{Z})|$ is an invariant of the homology sphere. As we have seen, those invariants that come from abelian $2$-cocycles are determined by the first $3$ digits only and in turn determine those. It would be interesting to further explore the behavior of the remaining digits.
\end{rem}

\begin{rem}
	A natural question arises as whether or not the invariants $\mathcal{R}$ and $\varphi$, defined for different values of the integer $d$ are compatible. It turns out that this is not the case. Indeed, fixed two coprime integers $d$, $e$,
	if we denote by $\varphi: \mathcal{S}[d^2] \rightarrow \mathbb{Z}/d^2$ the invariant defined for $d^2$ and by ${}^e\varphi$ the same invariant but defined on $\mathcal{S}[d^2e^2]$ and with values in $\mathbb{Z}/d^2e^2$. Then the following diagram does not commute
	\begin{center}
		\begin{tikzcd}
			\mathcal{S}[d^2] \ar[drr,"\varphi"] & & \\
			\mathcal{S}[d^2e^2] \ar[r, "{}^e\varphi"] \ar[u,hook] &  \mathbb{Z}/d^2 e^2 \ar[r] & \mathbb{Z}/d^2.
		\end{tikzcd}
	\end{center}
Indeed, the lens space $L =L(1 + d^2e^2,d)$ belongs to $\mathcal{S}[d^2e^2]$, but $
\varphi(L) = -e^2 \;(\text{mod } d^2)$ and ${}^e\varphi(L) = -1 \;(\text{mod } d^2)$ by Proposition~\ref{prop:valuephiLens}.
\end{rem}

\subsection{The Perron conjecture}

In \cite{per} B. Perron conjectured an extension of the Casson invariant $\lambda$ on the level-$p$ mapping class group with values on $\mathbb{Z}/p.$ This extension consists in writing an element of the level-$p$ mapping class group $\mathcal{M}_{g,1}[p]$ as a
product of an element of the Torelli group $\mathcal{T}_{g,1}$ and an element of the subgroup $D_{g,1}[p]$, which is the group generated by $p$-powers
of Dehn twists, and take the Casson invariant modulo $p$ of the element of the Torelli group.
More precisely, B. Perron conjectured that

\begin{conj}
\label{conj2}
Given an integer $g\geq 4,$ a prime $p\geq 5$ and $S^3_\phi \in \mathcal{S}^3[p]$ with $\phi=f\cdot m,$ where $f\in \mathcal{T}_{g,1},$ $m\in D_{g,1}[p].$ Then the map $\gamma_p:\mathcal{M}_{g,1}[p] \rightarrow \mathbb{Z}/p$ given by
$$\gamma_p(\phi)=\lambda(S^3_f) \; (\text{mod }p)$$
is a well defined invariant on $\mathcal{S}^3[p].$
\end{conj}

Assuming the conjecture is true we identify the associated $2$-cocycle:

\begin{prop}
\label{prop_obst_perron}
If Conjecture \ref{conj2} is true, then the associated
trivial $2$-cocycle to the function $\gamma_p$ is $(\tau_1^Z)^*(-2\;{}^t\!J_g).$
\end{prop}

\begin{proof}
Consider $\phi_1,\phi_2\in \mathcal{M}_{g,1}[p]$ with $\phi_1=f_1\cdot m_1,$ and $\phi_2=f_2\cdot m_2$ where $f_i\in \mathcal{T}_{g,1}$ and $m_i\in D_{g,1}[p]$ for $i=1,2.$
By \cite[Thm. 3]{pitsch}, on integral values, the trivial $2$-cocycle associated to the Casson invariant $\lambda$ is $\tau_1^*(-2\;{}^t\!J_g)$, where $\tau_1:\mathcal{T}_{g,1}\rightarrow \Lambda^3 H$ is the first Johnson homomorphism.
Then the following equalities hold in $\mathbb{Z}/p:$
\begin{align*}
\gamma_p(\phi_1)+\gamma_p(\phi_2)-\gamma_p(\phi_1\phi_2)& = 
\gamma_p(f_1m_1)+\gamma_p(f_2m_2)-\gamma_p(f_1m_1f_2m_2) \\
& = \gamma_p(f_1)+\gamma_p(f_2)-\gamma_p(f_1f_2(f_2^{-1}m_1f_2)m_2) \\
& =  \gamma_p(f_1)+\gamma_p(f_2)-\gamma_p(f_1f_2) \\
& = \lambda (f_1)+\lambda(f_2)- \lambda(f_1f_2)\\
& =  -2\;{}^t\!J_g(\tau_1(f_1),\tau_1(f_2)) \\
& =-2\;{}^t\!J_g(\tau_1^Z(f_1m_1),\tau_1^Z(f_2m_2))\\
& =  (\tau_1^Z)^*(-2\;{}^t\!J_g)(\phi_1,\phi_2),
\end{align*}
where the penultimate equality follows from the proof of \cite[Thm. 5.11]{coop}.
\end{proof}
But by Corollary~\ref{cor-extp-nontriv-cocy} the $2$-cocycle $(\tau_1^Z)^*(-2\;{}^t\!J_g)$ is not trivial. Therefore,

\begin{cor}
The extension of the Casson invariant modulo $p$ to the level-$p$ mapping class group proposed by B. Perron is not a well defined invariant of rational homology spheres in $\mathcal{S}^3[p]$.
\end{cor}

\newpage

\appendix

\section{Computations}

We now turn to some computations for the homology and cohomology of the level $d$ congruence subgroups and related groups. For these computations recall the canonical inclusion $GL_g(\mathbb{Z}) \hookrightarrow Sp_{2g}(\mathbb{Z})$, given by considering those symplectic matrices that respect the Lagrangian decomposition $H = A \oplus B$.
Moreover, the symmetric group $\mathfrak{S}_g \hookrightarrow GL_g(\mathbb{Z})$ under the action of $GL_g(\mathbb{Z})$ permutes the indices of the symplectic basis $\{a_i,b_i\}_{1\leq i\leq g}.$
In all what follows we denote $e_{ij}$ the elementary matrix with $1$ at the place $ij$ and zero elsewhere.

\begin{lema}\label{lem:spZ/dcoinv}
	Fix integers $g\geq 3$ and $d\geq 2$.
	\begin{enumerate}[a)]
		\item We have $H_1(Sym_g(\mathbb{Z}))_{GL_{g}(\mathbb{Z})} \simeq 0 \simeq H_1(Sym_g(\mathbb{Z}/d))_{GL_g(\mathbb{Z})}$.
		\item The canonical projection 
		\[
		\begin{array}{rcl}
		\pi_{gl}: \mathfrak{sp}_{2g}(\mathbb{Z}/d) & \longrightarrow & \mathfrak{gl}_{g}(\mathbb{Z}/d)\\ \left(\begin{smallmatrix}
		\alpha & \beta \\
		\gamma & -{}^t\alpha
		\end{smallmatrix}\right) & \longmapsto & \alpha
		\end{array}
		\]
		and the trace map $tr:\mathfrak{gl}_{g}(\mathbb{Z}/d)\rightarrow \mathbb{Z}/d$ induce isomorphisms:
		\[
		 (\mathfrak{sp}_{2g}(\mathbb{Z}/d))_{GL_g(\mathbb{Z})}\simeq (\mathfrak{gl}_g(\mathbb{Z}/d))_{GL_g(\mathbb{Z})}\simeq \mathbb{Z}/d  .
		\]
		\item We have $(\mathfrak{sl}_{g}(\mathbb{Z}/d))_{GL_g(\mathbb{Z})}\simeq 0$.
		\item The maps in point $b)$ induce isomorphisms
		\[
		(\mathfrak{sp}_{2g}(\mathbb{Z}/d))_{SL_g(\mathbb{Z})}\simeq (\mathfrak{gl}_g(\mathbb{Z}/d))_{SL_g(\mathbb{Z})}\simeq \mathbb{Z}/d \text{ and } (\mathfrak{sl}_{g}(\mathbb{Z}/d))_{SL_g(\mathbb{Z})}\simeq 0.
		\]
	\end{enumerate}
\end{lema}
\begin{proof} \textbf{Point a).} We only treat the case with $\mathbb{Z}$ coefficients, the other one is similar. Of course, writing $H_1(Sym_g(\mathbb{Z});\mathbb{Z})$ is slightly pedantic: the group in question is abelian, and we drop the homological writing. 
	The group $Sym_g(\mathbb{Z})$ is generated by the diagonal elementary matrices $e_{ii}$ and the sums of the corresponding off-diagonal symmetric matrices $se_{ij}:=e_{ij} + e_{ji}$.

Recall that
the action of $G\in GL_g(\mathbb{Z})$ on $S\in Sym_g(\mathbb{Z})$ is given by the rule $G. S=GS{}^tG$.
	 In particular, the canonical inclusion from the symmetric group in the general linear group acts on the elementary matrices by permuting the indices, so the coinvariant module $Sym_g(\mathbb{Z})_{GL_g(\mathbb{Z})}$ is generated by (the classes of) $e_{11}$ and $se_{12}$. We now compute:
	\begin{itemize}		
		\item Acting by $G=Id+e_{21}$ on $e_{11}$, gives
		$$ e_{11}= Ge_{11}{}^tG=
		e_{11}+e_{22}+se_{12}.$$
		Therefore in the coinvariants quotient:
		\begin{equation}
		\label{eqp:1} e_{11}=  -se_{12}.
		\end{equation}

		\item Acting by $G=Id+e_{21}+e_{31}$ on $e_{11}$, 			gives
		\begin{equation*}
		e_{11}=Ge_{11}{}^tG = e_{11}+e_{22}+e_{33}+se_{12}+se_{23}+se_{13}.
		\end{equation*}
		Therefore, by relation \eqref{eqp:1}, in the coinvariants quotient,
		\begin{equation}
		\label{eqp:2}
		e_{11}=Ge_{11}{}^tG=3e_{11}+3se_{12}=0.
		\end{equation}
	\end{itemize}
	
	\textbf{Point b).} We first prove that the projection $\pi_{gl}:\; \mathfrak{sp}_{2g}(\mathbb{Z}/d) \rightarrow\mathfrak{gl}_{g}(\mathbb{Z}/d)$ induces an isomorphism  $(\mathfrak{sp}_{2g}(\mathbb{Z}/d))_{GL_g(\mathbb{Z})}\simeq (\mathfrak{gl}_g(\mathbb{Z}/d))_{GL_g(\mathbb{Z})}$.
	
	Taking the $GL_g(\mathbb{Z})$-coinvariants in the decomposition~\eqref{dec_sp} provided in Section~\ref{subsec:Liealgebras}, we get that the module $(\mathfrak{sp}_{2g}(\mathbb{Z}/d))_{GL_g(\mathbb{Z})}$ is isomorphic to
	$$
	(\mathfrak{gl}_{g}(\mathbb{Z}/d))_{GL_g(\mathbb{Z})}\oplus (Sym_{g}^A(\mathbb{Z}/d))_{GL_g(\mathbb{Z})}\oplus (Sym_{g}^B(\mathbb{Z}/d))_{GL_g(\mathbb{Z})} .
	$$
	
	Observe that the map $G \rightarrow {}^tG^{-1}$ is an automorphism of $GL_g(\mathbb{Z})$, hence the $GL_g(\mathbb{Z})$-modules $Sym_{g}^A(\mathbb{Z}/d)$ and $Sym_{g}^B(\mathbb{Z}/d)$ are isomorphic. We conclude by point $a)$.
	
	Next we prove that the trace map $tr:\mathfrak{gl}_{g}(\mathbb{Z}/d)\rightarrow \mathbb{Z}/d$ induces an isomorphism $(\mathfrak{gl}_g(\mathbb{Z}/d))_{GL_g(\mathbb{Z})}\simeq \mathbb{Z}/d$.
	Clearly the trace map is a surjective and $GL_g(\mathbb{Z})$-invariant homomorphism on $\mathfrak{gl}_g(\mathbb{Z}/d)$. Hence it factors through the coinvariants quotient module.	To show that  $tr : \mathfrak{gl}_g(\mathbb{Z}/d)_{GL_g(\mathbb{Z})} \rightarrow  \mathbb{Z}/d$ is injective  remember that $\mathfrak{gl}_g(\mathbb{Z}/d)$ is generated by the elementary matrices $e_{ij}$. Again, the action of the symmetric subgroup shows that the classes of the matrices $e_{11}$ and $e_{12}$ generate the coinvariants quotient $\mathfrak{gl}_g(\mathbb{Z}/d)_{GL_g(\mathbb{Z})}$.  Because $tr(e_{11})=1 \neq 0$, it is enough to check that $e_{12}$ is $0$ in the coinvariants quotient. 
	Using the action by $G=Id+e_{21},$ we compute:
	\begin{equation}
	\label{eq:rel-gl}
	Ge_{12}G^{-1} = -e_{11}+e_{22}+e_{12}-e_{21}.
	\end{equation}
	Therefore, in the coinvariants module
	$$
	e_{12}
	= Ge_{12}G^{-1}
	= -e_{11}+e_{22}+e_{12}-e_{21}=0.$$

	\textbf{Point c).} Observe that $\mathfrak{sl}_g(\mathbb{Z}/d)$ is generated by the matrices of the form $(e_{ii}-e_{jj})$ and $e_{ij}$ with $i\neq j$. Again, the action of the symmetric subgroup shows that (the classes of) the matrices $e_{12}$ and $(e_{11}-e_{22})$ generate the coinvariants quotient $\mathfrak{sl}_g(\mathbb{Z}/d)_{GL_g(\mathbb{Z})}$. By equation \eqref{eq:rel-gl} given in the previous point we have that in the coinvariants module,
$(e_{11}-e_{22})=-e_{12}.$ Finally, using the action by $G=Id+e_{23}$, we compute:
$$Ge_{12}G^{-1} = e_{12}-e_{13}.$$
Therefore in the coinvariants module, $e_{12}=0$.
		
	\textbf{Point d).} Notice that in the proof of Point $b)$ and $c)$ we only used conjugations by elements of $SL_g(\mathbb{Z}).$ Therefore the same computations yield the isomorphism $(\mathfrak{sp}_{2g}(\mathbb{Z}/d))_{SL_g(\mathbb{Z})}\simeq \mathbb{Z}/d$ via the trace map.
\end{proof}

Now using this Lemma we compute the abelian groups:
\[
H_1(Sp_{2g}(\mathbb{Z},d);\mathbb{Z})_{GL_g(\mathbb{Z})}, \; H_1(Sp^A_{2g}(\mathbb{Z},d);\mathbb{Z})_{GL_g(\mathbb{Z})}, \;
H_1(Sp^B_{2g}(\mathbb{Z},d);\mathbb{Z})_{GL_g(\mathbb{Z})}.
\]

\begin{prop}\label{prop_iso_M[d],sp(Z/d)}
	For fixed integers $g\geq 3$ and $d\geq 2$, the homomorphisms $\alpha: Sp_{2g}(\mathbb{Z},d) \rightarrow \mathfrak{sp}_{2g}(\mathbb{Z}/d)$, given in \eqref{def-alpha}, and $tr\circ \pi_{gl}: \mathfrak{sp}_{2g}(\mathbb{Z}/d)\rightarrow \mathbb{Z}/d$, given in Lemma \ref{lem:spZ/dcoinv}, respectively induce isomorphisms
	\[
	H_1(Sp_{2g}(\mathbb{Z},d);\mathbb{Z})_{GL_g(\mathbb{Z})}\simeq (\mathfrak{sp}_{2g}(\mathbb{Z}/d))_{GL_g(\mathbb{Z})}\simeq \mathbb{Z}/d.
	\]
\end{prop}
\begin{proof}
	We handle separately the cases  $d$ odd and  $d$ even for the first isomorphism. The second isomorphism is due to Lemma \ref{lem:spZ/dcoinv}.

	\textbf{For an odd $d$.}
	From \cite{per}, \cite{sato_abel} or \cite{Putman_abel} we know that
	for any $g\geq 3$ and an odd integer $d\geq 3,$
	$$[Sp_{2g}(\mathbb{Z},d),Sp_{2g}(\mathbb{Z},d)]=Sp_{2g}(\mathbb{Z},d^2).$$
	Then the short exact sequence \eqref{ses_abel_d1} shows that $H_1(Sp_{2g}(\mathbb{Z},d))\simeq \mathfrak{sp}_{2g}(\mathbb{Z}/d),$ and taking $GL_g(\mathbb{Z})$-coinvariants we get the result.

	\textbf{For an even $d$.}
	Consider the short exact sequence \eqref{eq:ses-Sp-d2},
	\begin{equation*}
	\xymatrix@C=7mm@R=10mm{0 \ar@{->}[r] & H_1(\Sigma_{g,1};\mathbb{Z}/2) \ar@{->}[r] & H_1(Sp_{2g}(\mathbb{Z},d);\mathbb{Z}) \ar@{->}[r] & \mathfrak{sp}_{2g}(\mathbb{Z}/d) \ar@{->}[r] & 0 .}
	\end{equation*}
	Taking $GL_g(\mathbb{Z})$-coinvariants we get an exact sequence,
	\[
	\begin{tikzcd}
	 H_1(\Sigma_{g,1};\mathbb{Z}/2)_{GL_g(\mathbb{Z})} \arrow[r] & H_1(Sp_{2g}(\mathbb{Z},d);\mathbb{Z})_{GL_g(\mathbb{Z})} \arrow[dl,out=350,in=170,overlay] \\ (\mathfrak{sp}_{2g}(\mathbb{Z}/d))_{GL_g(\mathbb{Z})} \arrow[r] & 0 .
	\end{tikzcd}
	\]
	Finally a direct computation shows that $(H_1(\Sigma_{g,1};\mathbb{Z}/2))_{GL_g(\mathbb{Z})}=0,$ and we conclude by exactness.
\end{proof}

\begin{prop}\label{prop:com_SpB}
	For each $d\geq 3$ and $g\geq 3,$ the following groups are zero:
$$H_1(Sp_{2g}^A(\mathbb{Z},d);\mathbb{Z})_{GL_g(\mathbb{Z})}, \quad H_1(Sp_{2g}^B(\mathbb{Z},d);\mathbb{Z})_{GL_g(\mathbb{Z})}.$$
For $d= 2$ and $g\geq 3$, the aforementioned groups are isomorphic to the group $H_1(GL_g(\mathbb{Z},2);\mathbb{Z})_{GL_g(\mathbb{Z})}$, which in turn is isomorphic to $\mathbb{Z}/2$.
\end{prop}
\begin{proof}
	We only prove the result for $Sp_{2g}^B(\mathbb{Z},d)$. The case of $Sp_{2g}^A(\mathbb{Z},d)$ is similar.
	For $d\geq 3$, by Lemma \ref{lem:extensionshandelbody}, we have a split extension of groups with compatible $GL_g(\mathbb{Z})$-actions,
	\[
	\xymatrix@C=7mm@R=13mm{1 \ar@{->}[r] & Sym_g(d\mathbb{Z}) \ar@{->}[r] & Sp_{2g}^B(\mathbb{Z},d) \ar@{->}[r] & SL_g(\mathbb{Z},d) \ar@{->}[r] & 1 ,}
	\]
	where $GL_g(\mathbb{Z})$ acts on $SL_g(\mathbb{Z},d)$ by conjugation, on $Sp_{2g}^B(\mathbb{Z},d)$ by conjugation of matrices of the form $\left(\begin{smallmatrix}
	G & 0 \\
	0 & {}^tG^{-1}
	\end{smallmatrix}\right)$ with $G\in GL_g(\mathbb{Z}),$
	and on $Sym_g(d\mathbb{Z})$ by $G. S=GS\;{}^t G$.
	
	Taking  $GL_g(\mathbb{Z})$-coinvariants in the associated $3$-terms exact sequence we get the exact sequence
	\[
	\begin{tikzcd}
		H_1(Sym_g(d\mathbb{Z});\mathbb{Z})_{GL_g(\mathbb{Z})} \arrow[r] &  H_1(Sp_{2g}^B(\mathbb{Z},d);\mathbb{Z})_{GL_g(\mathbb{Z})}  \arrow[dl,out=350,in=170,overlay] \\   H_1(SL_g(\mathbb{Z},d);\mathbb{Z})_{GL_g(\mathbb{Z})}  \arrow[r] & 0.
	\end{tikzcd}
\]

	The abelian group $Sym_g(d\mathbb{Z})$ is isomorphic to $Sym_g(\mathbb{Z})$ and the leftmost homology group is trivial  by Lemma~\ref{lem:spZ/dcoinv} hereafter. 
	Finally, in \cite[Theorem 1.1]{Lee}, Lee and Szczarba showed that for $g\geq 3$ and any prime $p,$ $H_1(SL_g(\mathbb{Z},p))\simeq \mathfrak{sl}_g(\mathbb{Z}/p),$ as modules over $SL_g(\mathbb{Z}/p).$ Actually, the same proof holds true for any integer $d$.
	Then by Lemma \ref{lem:spZ/dcoinv}, we get that $(\mathfrak{sl}_g(\mathbb{Z}/d))_{GL_g(\mathbb{Z})}=0$, which finishes the proof for the case $d\geq 3$.
	For the case $d=2$, by Remark \ref{rem:d=2}, following the same arguments  replacing $SL_g(\mathbb{Z},2)$ by $GL_g(\mathbb{Z},2)$ we get an isomorphism
	$$H_1(Sp_{2g}^B(\mathbb{Z},2);\mathbb{Z})_{GL_g(\mathbb{Z})}\simeq H_1(GL_g(\mathbb{Z},2);\mathbb{Z})_{GL_g(\mathbb{Z})}.$$
	Finally we prove that this last group is isomorphic to $\mathbb{Z}/2$.
	
	Consider the short exact sequence
	\[
	\xymatrix@C=7mm@R=13mm{1 \ar@{->}[r] & SL_g(\mathbb{Z},2) \ar@{->}[r] & GL_g(\mathbb{Z},2) \ar@{->}[r]^-{det} & \mathbb{Z}/2 \ar@{->}[r] & 1 ,}
	\]
Taking  $GL_g(\mathbb{Z})$-coinvariants in the associated $3$-terms exact sequence we get an exact sequence 
$$\xymatrix@C=5mm@R=3mm{ H_1(SL_g(\mathbb{Z},2);\mathbb{Z})_{GL_g(\mathbb{Z})} \ar@{->}[r] &  H_1(GL_g(\mathbb{Z},2);\mathbb{Z})_{GL_g(\mathbb{Z})}  \ar@{->}[r]  &  \mathbb{Z}/2  \ar@{->}[r] & 0.}$$
Then we conclude by \cite[Theorem 1.1]{Lee} and Lemma \ref{lem:spZ/dcoinv}.
\end{proof}

\begin{prop}\label{exh_Spd2}
	Given integers $g\geq 3$ and $d$ even with $4\nmid d,$ the inclusion of $Sp_{2g}(\mathbb{Z},d)$ in $Sp_{2g}(\mathbb{Z},2)$ induces  epimorphisms in homology:
	\[
	H_i(Sp_{2g}(\mathbb{Z},d);\mathbb{Z})\twoheadrightarrow H_i(Sp_{2g}(\mathbb{Z},2);\mathbb{Z}) \qquad \text{for } i=1,2.
	\]
\end{prop}

\begin{proof} Set $d=2q$ with $q$ an odd positive integer.

 \textbf{Surjectivity for $H_1$.}
	By \cite[Thm.~1]{newman} there is a short exact sequence
	$$\xymatrix@C=7mm@R=10mm{1 \ar@{->}[r] & Sp_{2g}(\mathbb{Z},2q) \ar@{->}[r] & Sp_{2g}(\mathbb{Z}) \ar@{->}[r] & Sp_{2g}(\mathbb{Z}/2q) \ar@{->}[r] & 1 } .$$
	By \cite[Theorem 5]{newman}, $Sp_{2g}(\mathbb{Z}/2q)=Sp_{2g}(\mathbb{Z}/2)\oplus Sp_{2g}(\mathbb{Z}/q)$, and since $2$ is invertible in $\mathbb{Z}/q$, we have that $Sp_{2g}(\mathbb{Z}/q,2)=Sp_{2g}(\mathbb{Z}/q)$. Then, the restriction to $Sp_{2g}(\mathbb{Z},2)$ gives another short exact sequence
	\begin{equation}
	\label{ses_Spdq}
	\xymatrix@C=7mm@R=10mm{1 \ar@{->}[r] & Sp_{2g}(\mathbb{Z},2q) \ar@{->}[r] & Sp_{2g}(\mathbb{Z},2) \ar@{->}[r] & Sp_{2g}(\mathbb{Z}/q) \ar@{->}[r] & 1 }.
	\end{equation}
	Consider its associated $3$-term exact sequence,
	\[
	\begin{tikzcd}
	H_1(Sp_{2g}(\mathbb{Z},2q);\mathbb{Z})_{ Sp_{2g}(\mathbb{Z}/q)} \arrow[r] & H_1(Sp_{2g}(\mathbb{Z},2);\mathbb{Z}) \arrow[dl, out=350,in=170,overlay] \\
	 H_1(Sp_{2g}(\mathbb{Z}/q);\mathbb{Z}) \arrow[r] & 0 .
	\end{tikzcd}
\]

	By \cite[Lemma 3.7]{Putman} for $g \geq 3$, $H_1(Sp_{2g}(\mathbb{Z}/q);\mathbb{Z})=0,$ and by exactness we get the result.
	
	\textbf{Surjectivity for $H_2$.}
	Taking the associated Hochschild-Serre spectral sequence of \eqref{ses_Spdq}, since for $g\geq 3$ we know that $H_2(Sp_{2g}(\mathbb{Z}/q);\mathbb{Z})=0$ (cf. \cite[ Theorem~2.13 and Proposition 3.3.a]{stein}), the $E_\infty$  page has the shape
	\begin{figure}[H]
		\begin{center}
			\begin{tikzpicture}[scale=.7]
			\draw[thick] (0,3) -- (0,0) -- (3,0);
			
			\node at (0.5,0.5) {*};
			\node at (1.5,0.5) {*};
			\node at (0.5,1.5) {*};
			\node at (2.5,0.5) {$0$};
			\node at (1.5,1.5) {$B$};
			\node at (0.5,2.5) {$A$};
			\end{tikzpicture}
		\end{center}
	\end{figure}
	where $A$ is the image of $H_2(Sp_{2g}(\mathbb{Z},2q);\mathbb{Z})$ in $H_2(Sp_{2g}(\mathbb{Z},2);\mathbb{Z})$ and $B$ a quotient of $H_1(Sp_{2g}(\mathbb{Z}/q);H_1(Sp_{2g}(\mathbb{Z},2q);\mathbb{Z})).$
	Then we get a short exact sequence
	\begin{equation}
	\label{eq:ses-Sp2-AB}
	\xymatrix@C=7mm@R=10mm{0 \ar@{->}[r] & A \ar@{->}[r] & H_2(Sp_{2g}(\mathbb{Z},2);\mathbb{Z}) \ar@{->}[r] & B \ar@{->}[r] & 0 }.
	\end{equation}
	Next we prove that $H_1(Sp_{2g}(\mathbb{Z}/q);H_1(Sp_{2g}(\mathbb{Z},2q);\mathbb{Z}))$ is zero.
	
	Consider the short exact sequence \eqref{eq:ses-Sp-d2},
	\begin{equation*}
	\xymatrix@C=7mm@R=10mm{0 \ar@{->}[r] & H_1(\Sigma_{g,1};\mathbb{Z}/2) \ar@{->}[r] & H_1(Sp_{2g}(\mathbb{Z},2q);\mathbb{Z}) \ar@{->}[r] & \mathfrak{sp}_{2g}(\mathbb{Z}/2q) \ar@{->}[r] & 0 .}
	\end{equation*}
	To make notations lighter we set $V=H_1(\Sigma_{g,1};\mathbb{Z}/2)$. The associated long homology exact sequence in low degrees gives an exact sequence,
	
	\begin{equation}
	\label{eq:exact-seq-V}
	\begin{tikzcd}
	H_1(Sp_{2g}(\mathbb{Z}/q);V) \arrow[r] & H_1(Sp_{2g}(\mathbb{Z}/q);H_1(Sp_{2g}(\mathbb{Z},2q);\mathbb{Z})) \arrow[d,out =east, in=west,overlay] \\ &  H_1(Sp_{2g}(\mathbb{Z}/q);\mathfrak{sp}_{2g}(\mathbb{Z}/2q)). 
	\end{tikzcd}
	\end{equation}
	We show that $H_1(Sp_{2g}(\mathbb{Z}/q);V)$ is zero. Notice first that the action of $Sp_{2g}(\mathbb{Z}/q)$ is trivial since it is induced from the trivial action of $Sp_{2g}(\mathbb{Z},2)$ on $H_1(\Sigma_{g,1},\mathbb{Z}/2)$. Thus, $H_1(Sp_{2g}(\mathbb{Z}/q);V)=H_1(Sp_{2g}(\mathbb{Z}/q);\mathbb{Z})\otimes V$ and by \cite[Lemma 3.7]{Putman} this last group is zero.

	Next we show that $H_1(Sp_{2g}(\mathbb{Z}/q);sp_{2g}(\mathbb{Z}/2q))$ is zero.
	
	Consider the short exact sequence
	\begin{equation*}
	\xymatrix@C=7mm@R=10mm{0 \ar@{->}[r] & \mathfrak{sp}_{2g}(\mathbb{Z}/2) \ar@{->}[r] & \mathfrak{sp}_{2g}(\mathbb{Z}/2q) \ar@{->}[r] & \mathfrak{sp}_{2g}(\mathbb{Z}/q) \ar@{->}[r] & 0 .}
	\end{equation*}
	The associated long homology exact sequence in low degrees gives an exact sequence,
	\[
	\begin{tikzcd}
	H_1(Sp_{2g}(\mathbb{Z}/q);\mathfrak{sp}_{2g}(\mathbb{Z}/2)) \arrow[r] & H_1(Sp_{2g}(\mathbb{Z}/q);\mathfrak{sp}_{2g}(\mathbb{Z}/2q)) \arrow[d,out=east,in=west] \\
	& H_1(Sp_{2g}(\mathbb{Z}/q);\mathfrak{sp}_{2g}(\mathbb{Z}/q)).
	\end{tikzcd}
	\]
	
	By \cite[Thm. G]{Putman}, $H_1(Sp_{2g}(\mathbb{Z}/q);\mathfrak{sp}_{2g}(\mathbb{Z}/q))=0$. Then by above exact sequence it is enough to show that 
	\[
	H_1(Sp_{2g}(\mathbb{Z}/q);\mathfrak{sp}_{2g}(\mathbb{Z}/2))=0.
	\]
	To make notations lighter we set $U=\mathfrak{sp}_{2g}(\mathbb{Z}/2).$ Consider the $3$-term exact sequence associated to the short exact sequence \eqref{ses_Spdq}, 
	\begin{equation*}
	\xymatrix@C=7mm@R=10mm{H_1(Sp_{2g}(\mathbb{Z},2q);U) \ar@{->}[r] & H_1(Sp_{2g}(\mathbb{Z},2);U) \ar@{->}[r] & H_1(Sp_{2g}(\mathbb{Z}/q);U) \ar@{->}[r] & 0 .}
	\end{equation*}
	Since $Sp_{2g}(\mathbb{Z},2q)$ and $Sp_{2g}(\mathbb{Z},2)$ act trivially on $\mathfrak{sp}_{2g}(\mathbb{Z}/2)$, by the UCT this exact sequence becomes
	\[
	\begin{tikzcd}
	H_1(Sp_{2g}(\mathbb{Z},2q);\mathbb{Z})\otimes U \arrow[r] & H_1(Sp_{2g}(\mathbb{Z},2);\mathbb{Z})\otimes U \arrow[d,out=east,in=west,overlay] & \\ & H_1(Sp_{2g}(\mathbb{Z}/q);U) \arrow[r] & 0 .
	\end{tikzcd}
	\]
	Since $H_1(Sp_{2g}(\mathbb{Z},2q);\mathbb{Z})\rightarrow H_1(Sp_{2g}(\mathbb{Z},2);\mathbb{Z})$ is surjective by the first part of the proof, by exactness we get that
	$$H_1(Sp_{2g}(\mathbb{Z}/q);\mathfrak{sp}_{2g}(\mathbb{Z}/2))=0.$$
	Then, by exact sequence \eqref{eq:exact-seq-V}, $H_1(Sp_{2g}(\mathbb{Z}/q);H_1(Sp_{2g}(\mathbb{Z},2q);\mathbb{Z}))=0$ and therefore, by short exact sequence \eqref{eq:ses-Sp2-AB}, $A\simeq H_2(Sp_{2g}(\mathbb{Z},2);\mathbb{Z})$.
\end{proof}

\begin{prop}
\label{prop_inj_MGC}
Given integers $g\geq 4,$ $k\geq 1$ and an odd prime $p$, there is a commutative diagram of injective homomorphisms,
$$
\xymatrix@C=7mm@R=10mm{  H^2(Sp_{2g}(\mathbb{Z});\mathbb{Z}/p) \ar@{^{(}->}[d]^-{\res} \ar@{->}[r]^-{\inf}_{\sim} & H^2(\mathcal{M}_{g,1};\mathbb{Z}/p) \ar@{^{(}->}[d]^-{\res} \\
 H^2(Sp_{2g}(\mathbb{Z},p^k);\mathbb{Z}/p) \ar@{^{(}->}[r]^-{\inf} & H^2(\mathcal{M}_{g,1}[p^k];\mathbb{Z}/p). }
$$
\end{prop}

\begin{proof}
Consider the short exact sequence
$$
\xymatrix@C=7mm@R=10mm{1 \ar@{->}[r] & Sp_{2g}(\mathbb{Z},p^k) \ar@{->}[r] & Sp_{2g}(\mathbb{Z}) \ar@{->}[r]^-{r_{p^k}} & Sp_{2g}(\mathbb{Z}/p^k) \ar@{->}[r] & 1 },$$
where the surjectivity of $r_{p^k}$ was proved by  M. Newmann and J. R. Smart in \cite[Thm.~1]{newman}.
Write the associated 7-term exact sequence
\[
 \xymatrix@C=-4mm@R=5mm{
 0 \ar[r] & H^1(Sp_{2g}(\mathbb{Z}/p^k);\mathbb{Z}/p) \ar[r] & H^1(Sp_{2g}(\mathbb{Z});\mathbb{Z}/p) \ar@{->} `r/8pt[d] `/10pt[l] `^dl[ll] `^r/3pt[dl] [dl] \\
 &   H^1(Sp_{2g}(\mathbb{Z},p^k);\mathbb{Z}/p)^{Sp_{2g}(\mathbb{Z}/p^k)}  \ar[r] & H^2(Sp_{2g}(\mathbb{Z}/p^k);\mathbb{Z}/p) \ar@{->} `r/8pt[d] `/10pt[l] `^dl[ll] `^r/3pt[dll] [dl] \\
  & H^2(Sp_{2g}(\mathbb{Z});\mathbb{Z}/p)_1 \ar[r] & H^1(Sp_{2g}(\mathbb{Z}/p^k);H^1(Sp_{2g}(\mathbb{Z},p^k);\mathbb{Z}/p)) 
}
 \]
where,
$$H^2(Sp_{2g}(\mathbb{Z});\mathbb{Z}/p)_1 :=\ker(\xymatrix@C=4mm@R=10mm{ H^2(Sp_{2g}(\mathbb{Z});\mathbb{Z}/p) \ar@{->}[r]^-{res} & H^2(Sp_{2g}(\mathbb{Z},p^k);\mathbb{Z}/p)}).$$
Now we show that $H^2(Sp_{2g}(\mathbb{Z});\mathbb{Z}/p)_1$ is zero.

By \cite[Lemma 3.7]{Putman} and \cite[Thm. 3.8]{Putman}, we know that, for $k\geq 2$ and $g\geq 3$, the groups
$H_2(Sp_{2g}(\mathbb{Z}/p^k);\mathbb{Z})$ and $H_1(Sp_{2g}(\mathbb{Z}/p^k);\mathbb{Z})$ are zero.
Then, by the UCT,
\begin{equation*}
\begin{aligned}
H^2(Sp(\mathbb{Z}/p^k);\mathbb{Z}/p)
\simeq & \; Hom(H_2(Sp(\mathbb{Z}/p^k);\mathbb{Z}),\mathbb{Z}/p) \\
 & \; \oplus \text{Ext}^1_\mathbb{Z}(H_1(Sp(\mathbb{Z}/p^k);\mathbb{Z}),\mathbb{Z}/p) =0,
\end{aligned}
\end{equation*}
and
\begin{equation*}
\begin{aligned} H^1(Sp_{2g}(\mathbb{Z},p^k);\mathbb{Z}/p) \simeq & \; Hom(H_1(Sp_{2g}(\mathbb{Z},p^k);\mathbb{Z}),\mathbb{Z}/p) \\ &  \;\oplus \text{Ext}^1_\mathbb{Z}(H_0(Sp_{2g}(\mathbb{Z},p^k);\mathbb{Z}),\mathbb{Z}/p) \\
= &  \; Hom(\mathfrak{sp}_{2g}(\mathbb{Z}/p^k),\mathbb{Z}/p)\oplus \text{Ext}^1_\mathbb{Z}(\mathbb{Z},\mathbb{Z}/p)\\
= &  \;(\mathfrak{sp}_{2g}(\mathbb{Z}/p))^*.
\end{aligned}
\end{equation*}
As a consequence,
$$H^1(Sp_{2g}(\mathbb{Z}/p^k);H^1(Sp_{2g}(\mathbb{Z},p^k);\mathbb{Z}/p))=
H^1(Sp_{2g}(\mathbb{Z}/p^k);(\mathfrak{sp}_{2g}(\mathbb{Z}/p))^*),$$
and by Lemma \ref{lema_duality_homology} and Step 2 in \cite[Section 4.3]{Putman} this last group is zero.

Therefore $H^2(Sp_{2g}(\mathbb{Z});\mathbb{Z}/p)_1=0$ and hence by definition of this group, there is a monomorphism 
\begin{equation}
\label{eq:inj-Sp}
\res:\;H^2(Sp_{2g}(\mathbb{Z});\mathbb{Z}/p)\hookrightarrow H^2(Sp_{2g}(\mathbb{Z},p^k);\mathbb{Z}/p).
\end{equation}

Next, using this injection we deduce the analogous injection in terms of the mapping class group, $\res: H^2(\mathcal{M}_{g,1};\mathbb{Z}/p) \hookrightarrow H^2(\mathcal{M}_{g,1}[p^k];\mathbb{Z}/p).$
Consider the commutative diagram with exact rows:
$$
\xymatrix@C=7mm@R=10mm{1 \ar@{->}[r] & \mathcal{T}_{g,1} \ar@{=}[d] \ar@{->}[r] & \mathcal{M}_{g,1}[p^k] \ar@{->}[r] \ar@{^{(}->}[d] & Sp_{2g}(\mathbb{Z},p^k) \ar@{^{(}->}[d] \ar@{->}[r] & 1 \\
1 \ar@{->}[r] & \mathcal{T}_{g,1} \ar@{->}[r] & \mathcal{M}_{g,1} \ar@{->}[r] & Sp_{2g}(\mathbb{Z}) \ar@{->}[r] & 1.}$$
Then there is a ladder of 5-term exact sequences,
$$
\xymatrix@C=7mm@R=10mm{0 \ar@{->}[r] & H^1(Sp_{2g}(\mathbb{Z});\mathbb{Z}/p) \ar@{->}[d]^-{\res} \ar@{->}[r]^-{\inf} & H^1(\mathcal{M}_{g,1};\mathbb{Z}/p) \ar@{->}[d]^-{\res} \ar@{-}[r]^-{\res} & \\
0 \ar@{->}[r] & H^1(Sp_{2g}(\mathbb{Z},p^k);\mathbb{Z}/p) \ar@{->}[r]^-{\inf} & H^1(\mathcal{M}_{g,1}[p^k];\mathbb{Z}/p) \ar@{-}[r]^-{\res} & }
$$
$$
\xymatrix@C=4mm@R=10mm{\ar@{->}[r] & H^1(\mathcal{T}_{g,1};\mathbb{Z}/p)^{Sp_{2g}(\mathbb{Z})} \ar@{=}[d] \ar@{->}[r] & H^2(Sp_{2g}(\mathbb{Z});\mathbb{Z}/p) \ar@{->}[d]^-{\res} \ar@{->}[r]^-{\inf} & H^2(\mathcal{M}_{g,1};\mathbb{Z}/p) \ar@{->}[d]^-{\res} \\
\ar@{->}[r]& H^1(\mathcal{T}_{g,1};\mathbb{Z}/p)^{Sp_{2g}(\mathbb{Z},p^k)} \ar@{->}[r] & H^2(Sp_{2g}(\mathbb{Z},p^k);\mathbb{Z}/p) \ar@{->}[r]^-{\inf} & H^2(\mathcal{M}_{g,1}[p^k];\mathbb{Z}/p)}
$$
Next we prove that the map $\inf:H^2(Sp_{2g}(\mathbb{Z},p^k);\mathbb{Z}/p)\rightarrow H^2(\mathcal{M}_{g,1}[p^k];\mathbb{Z}/p)$ is injective and the map $\inf: H^2(Sp_{2g}(\mathbb{Z});\mathbb{Z}/p)\rightarrow H^2(\mathcal{M}_{g,1};\mathbb{Z}/p)$ is an isomorphism. Then by the injection \eqref{eq:inj-Sp} and a trivial diagram chase we will conclude.

To prove that $\inf:H^2(Sp_{2g}(\mathbb{Z},p^k);\mathbb{Z}/p)\rightarrow H^2(\mathcal{M}_{g,1}[p^k];\mathbb{Z}/p)$ is injective we show that $\res: H^1(\mathcal{M}_{g,1}[p^k],\mathbb{Z}/p) \rightarrow H^1(\mathcal{T}_{g,1},\mathbb{Z}/p)^{Sp_{2g}(\mathbb{Z},p^k)} $ is surjective and we will conlude by exact ladder above.
Recall that we denote $H = H_1(\Sigma_{g,1}; \mathbb{Z})$ and $H_d = H \otimes \mathbb{Z}/d$ for any integer $d$.

By the UCT there are isomorphisms:
\begin{align*}
& H^1(\mathcal{M}_{g,1}[p^k],\mathbb{Z}/p)\simeq  Hom(H_1(\mathcal{M}_{g,1}[p^k];\mathbb{Z}),\mathbb{Z}/p),\\[1ex]
& \begin{aligned} H^1(\mathcal{T}_{g,1},\mathbb{Z}/p)^{Sp_{2g}(\mathbb{Z},p^k)}\simeq \; & Hom(H_1(\mathcal{T}_{g,1};\mathbb{Z}),\mathbb{Z}/p)^{Sp_{2g}(\mathbb{Z},p^k)} \\
\simeq \; & Hom(\Lambda^3 H_{p^k},\mathbb{Z}/p),\end{aligned} \\[1ex]
& \begin{aligned} H^1(Sp_{2g}(\mathbb{Z},p^k);\mathbb{Z}/p) \simeq \; &  Hom(H_1(Sp_{2g}(\mathbb{Z},p^k);\mathbb{Z}),\mathbb{Z}/p) \\
\simeq \; & Hom(\mathfrak{sp}_{2g}(\mathbb{Z}/p^k),\mathbb{Z}/p).
\end{aligned}
\end{align*}

By \cite[Thm. 0.4]{sato_abel} we have a split extension of $\mathbb{Z}/p^k$-modules,
$$\xymatrix@C=7mm@R=10mm{0 \ar@{->}[r] & \Lambda^3 H_{p^k} \ar@{->}[r] & H_1(\mathcal{M}_{g,1}[p^k];\mathbb{Z})\ar@{->}[r] & \mathfrak{sp}_{2g}(\mathbb{Z}/p^k)\ar@{->}[r] & 0.}$$
Applying the functor $Hom(-,\mathbb{Z}/p),$ we get an exact sequence, 
\[
	\begin{tikzcd}
	0 \arrow[r] & Hom(\mathfrak{sp}_{2g}(\mathbb{Z}/p^k), \mathbb{Z}/p)
	\arrow[r] &  Hom(H_1(\mathcal{M}_{g,1}[p^k];\mathbb{Z}), \mathbb{Z}/p)\arrow[dl, out=350,in=170,overlay] \\
  & Hom(\Lambda^3 H_{p^k}, \mathbb{Z}/p)\arrow[r] & 0.
  \end{tikzcd}
\]
Therefore, $ H^1(\mathcal{M}_{g,1}[p^k],\mathbb{Z}/p) \stackrel{\text{res}}{\longrightarrow} H^1(\mathcal{T}_{g,1},\mathbb{Z}/p)^{Sp_{2g}(\mathbb{Z},p^k)} $ is surjective.

Next we show that the map $H^2(Sp_{2g}(\mathbb{Z});\mathbb{Z}/p)\stackrel{\text{inf}}{\longrightarrow} H^2(\mathcal{M}_{g,1};\mathbb{Z}/p)$ is an isomorphism.
By the UCT, and the Center kills lemma for the last equality,
\begin{align*}
H^1(\mathcal{T}_{g,1},\mathbb{Z}/p)^{Sp_{2g}(\mathbb{Z})}  & \simeq  Hom(H_1(\mathcal{T}_{g,1};\mathbb{Z}),\mathbb{Z}/p)^{Sp_{2g}(\mathbb{Z})} \\
& \simeq  Hom(\Lambda^3 H_p,\mathbb{Z}/p)^{Sp_{2g}(\mathbb{Z})}=0.
\end{align*}
As a consequence, from exact ladder above, the map $H^2(Sp_{2g}(\mathbb{Z});\mathbb{Z}/p)\stackrel{\text{inf}}{\longrightarrow} H^2(\mathcal{M}_{g,1};\mathbb{Z}/p)$ is injective.
Moreover, \cite[Thm. 5.1]{Putman} and \cite[Thm. 5.8]{farb} together with the UCT shows that for $g\geq 4$, $H^2(Sp_{2g}(\mathbb{Z});\mathbb{Z}/p)\simeq \mathbb{Z}/p$ and $H^2(\mathcal{M}_{g,1};\mathbb{Z}/p)\simeq \mathbb{Z}/p$.
Hence the map $\textrm{inf}$ is an isomorphism.
\end{proof}

\begin{prop}
\label{prop_unicity-ext-johnson}
Given an integer $g\geq 2$ and an odd prime $p$, the restriction map induces isomorphisms
$$
 H^1(\mathcal{M}_{g,1};\Lambda^3H_p)\simeq H^1(\mathcal{M}_{g,1}[p];\Lambda^3H_p)^{Sp_{2g}(\mathbb{Z}/p)}\simeq H^1(\mathcal{T}_{g,1};\Lambda^3H_p)^{Sp_{2g}(\mathbb{Z})}.$$
\end{prop}

\begin{proof}
Consider the following commutative diagram with exact rows:
$$
\xymatrix@C=7mm@R=10mm{1 \ar@{->}[r] & \mathcal{T}_{g,1} \ar@{->}[d] \ar@{->}[r] & \mathcal{M}_{g,1} \ar@{=}[d] \ar@{->}[r] & Sp_{2g}(\mathbb{Z}) \ar@{->}[d] \ar@{->}[r] & 1 \\
1 \ar@{->}[r] & \mathcal{M}_{g,1}[p] \ar@{->}[r] & \mathcal{M}_{g,1} \ar@{->}[r] & Sp_{2g}(\mathbb{Z}/p) \ar@{->}[r] & 1 }$$
and the $Sp_{2g}(\mathbb{Z}/p)$-module $\Lambda^3H_p.$
Then there is a ladder of 5-term exact sequences,
$$
\xymatrix@C=5mm@R=10mm{0 \ar@{->}[r] & H^1(Sp_{2g}(\mathbb{Z}/p);\Lambda^3H_p) \ar@{->}[d]\ar@{->}[r] & H^1(\mathcal{M}_{g,1};\Lambda^3H_p) \ar@{=}[d]\ar@{->}[r] &  \\
0 \ar@{->}[r] & H^1(Sp_{2g}(\mathbb{Z});\Lambda^3H_p) \ar@{-}[r] & H^1(\mathcal{M}_{g,1};\Lambda^3H_p) \ar@{->}[r] & }$$
$$
\xymatrix@C=4mm@R=10mm{ H^1(\mathcal{M}_{g,1}[p];\Lambda^3H_p)^{Sp_{2g}(\mathbb{Z}/p)} \ar@{->}[d] \ar@{->}[r] & H^2(Sp_{2g}(\mathbb{Z}/p);\Lambda^3H_p) \ar@{->}[d]\ar@{->}[r] & H^2(\mathcal{M}_{g,1};\Lambda^3H_p) \ar@{=}[d]\\
H^1(\mathcal{T}_{g,1};\Lambda^3H_p)^{Sp_{2g}(\mathbb{Z})} \ar@{->}[r] & H^2(Sp_{2g}(\mathbb{Z});\Lambda^3H_p) \ar@{->}[r] & H^2(\mathcal{M}_{g,1};\Lambda^3H_p)}$$
By the Center kills Lemma, since $-Id$ acts on $\Lambda^3H_p$ as the multiplication by $-1,$ we have that $H^i(Sp_{2g}(\mathbb{Z}/p);\Lambda^3H_p)=0=H^i(Sp_{2g}(\mathbb{Z});\Lambda^3H_p)$ for $i=1,2,$ and by exactness we get the result. 
\end{proof}
Analogously, if we take the handlebody subgroups $\mathcal{B}_{g,1},$ $\mathcal{B}_{g,1}[p],$ $\mathcal{TB}_{g,1}$ instead of $\mathcal{M}_{g,1},$ $\mathcal{M}_{g,1}[p],$ $\mathcal{T}_{g,1},$ proceeding in the same way that in above Proposition \ref{prop_unicity-ext-johnson} we get:

\begin{prop}
\label{prop_res_B}
For any odd prime $p$ and an integer $g\geq 2$, the restriction maps induce isomorphisms
$$
 H^1(\mathcal{B}_{g,1};\Lambda^3H_p)\simeq H^1(\mathcal{B}_{g,1}[p];\Lambda^3H_p)^{Sp^B_{2g}(\mathbb{Z}/p)}\simeq H^1(\mathcal{TB}_{g,1};\Lambda^3H_p)^{Sp^B_{2g}(\mathbb{Z})}.$$
\end{prop}


We now turn to compute the coinvariant quotients of tensor and exterior powers of the two pieces of $H_1(\mathcal{M}_{g,1}[p];\mathbb{Z})=\Lambda^3 H_p\oplus \mathfrak{sp}_{2g}(\mathbb{Z}/p),$ which we use when computing invariants in Section \ref{sec-triv-2coc-H1}.

For the first piece, $\Lambda^3 H_p$, the exterior power of triples of elements from the symplectic basis of $H_p$ given by $\{a_i,b_i\}_{1\leq i\leq g}$ form a basis of the $\mathbb{Z}/p$-vector space $\Lambda^3 H_p$, and then the tensor product of pairs of elements from the basis of $\Lambda^3 H_p$ produces a basis of the $\mathbb{Z}/p$-vector space $\Lambda^3 H_p\otimes\Lambda^3 H_p$.

\begin{prop}
\label{prop-coinv-tens-extp}
Given an integer $g\geq 4$ and an odd prime $p$, the $\mathbb{Z}/p$-vector space $(\Lambda^3 H_p\otimes \Lambda^3 H_p)_{GL_g(\mathbb{Z})}$ is generated by the following six elements:
\begin{equation*}
\begin{aligned}
(a_1\wedge a_2 \wedge a_3) \otimes (b_1\wedge b_2 \wedge b_3)
& \qquad
(b_1\wedge b_2 \wedge b_3) \otimes (a_1\wedge a_2 \wedge a_3) \\
(a_1\wedge a_2 \wedge b_2) \otimes (b_1\wedge a_2 \wedge b_2) &\qquad (b_1\wedge a_2 \wedge b_2) \otimes (a_1\wedge a_2 \wedge b_2) \\
(a_1\wedge a_2 \wedge b_2) \otimes (b_1\wedge a_3 \wedge b_3)& \qquad (b_1\wedge a_2 \wedge b_2) \otimes (a_1\wedge a_3 \wedge b_3)
\end{aligned}
\end{equation*}
\end{prop}

\begin{proof}
Let $c,c'$ stand for $a$ or $b,$ and $l,m,n$ pair-wise distinct indices.
Notice that the $\mathbb{Z}/p$-vector space $\Lambda^3H_p\otimes \Lambda^3 H_p$ is generated by the $\mathfrak{S}_g$-orbits of the following elements:
$$(c_1\wedge c_2\wedge c_3) \otimes (c'_{l}\wedge c'_{m}\wedge c'_{n}),\qquad    (c_1\wedge a_2\wedge b_2)\otimes (c'_{l}\wedge c'_{m}\wedge c'_{n}),$$
$$(c_1\wedge c_2\wedge c_3) \otimes (c'_{l}\wedge a_{m}\wedge b_m),\qquad    (c_1\wedge a_2\wedge b_2)\otimes (c'_{l}\wedge a_m\wedge b_m).$$

Taking the quotient by the action of $GL_g(\mathbb{Z})$, we now reduce the number of generators in the coinvariants module $\left(\Lambda^3H_p\otimes \Lambda^3 H_p\right)_{GL_g(\mathbb{Z})}.$

Let $v$ be one of above elements, if $v$ has an index $k$ that appears an odd number of times, the action by $G=Id-2e_{kk}$ sends $v$ to $-v,$ and hence this element is annihilated in the coinvariants module.
Therefore we are left with:
$$(c_1\wedge c_2\wedge c_3) \otimes (c'_{1}\wedge c'_{2}\wedge c'_{3}), \quad(c_1\wedge a_2\wedge b_2)\otimes (c'_{1}\wedge a_m\wedge b_m),$$
and picking one representative in each $\mathfrak{S}_g$-orbit we are left with the generating set:
$$(c_1\wedge c_2 \wedge c_3) \otimes (c'_1\wedge c'_2 \wedge c'_3),\quad
(c_1\wedge a_2 \wedge b_2) \otimes (c'_1\wedge a_2 \wedge b_2),$$
$$(c_1\wedge a_2 \wedge b_2) \otimes (c'_1\wedge a_3 \wedge b_3).$$
We follow to reduce more the number of generators.
\begin{itemize}
\item For the elements $(c_1\wedge c_2 \wedge c_3) \otimes (c'_1\wedge c'_2 \wedge c'_3),$ suppose that $c_i=c'_i$ for some $i=1,2,3,$ without loss of generality we can suppose $c_1=c_1'.$
If $c_1=c_1'=a_1$, acting by $G=Id+e_{14}$,
\begin{align*}
G.(a_1\wedge c_2 \wedge c_3) \otimes G.(a_4\wedge c'_2 \wedge c'_3)
= & (a_1\wedge c_2 \wedge c_3) \otimes ((a_1+a_4)\wedge c'_2 \wedge c'_3)\\
= & (a_1\wedge c_2 \wedge c_3) \otimes (a_1\wedge c'_2 \wedge c'_3) \\
& +(a_1\wedge c_2 \wedge c_3) \otimes (a_4\wedge c'_2 \wedge c'_3).
\end{align*}
Then, in the coinvariants module, $(a_1\wedge c_2 \wedge c_3) \otimes (a_1\wedge c'_2 \wedge c'_3)=0.$ Analogously, for $c_1=c_1'=b_1$, acting by $G=Id+e_{41}$ one gets that, in the coinvariants module, $(b_1\wedge c_2 \wedge c_3) \otimes (b_1\wedge c'_2 \wedge c'_3)=0.$ Therefore we are left with the elements with $c_i\neq c'_i$ for $i=1,2,3,$ and picking one representative in each $\mathfrak{S}_g$-orbit, we are only left with:
\begin{equation}
\label{elem_1}
\begin{aligned}
(a_1\wedge a_2 \wedge a_3) \otimes (b_1\wedge b_2 \wedge b_3)
& \qquad
(b_1\wedge b_2 \wedge b_3) \otimes (a_1\wedge a_2 \wedge a_3) \\
(a_1\wedge a_2 \wedge b_3) \otimes (b_1\wedge b_2 \wedge a_3)
& \qquad
(b_1\wedge b_2 \wedge a_3) \otimes (a_1\wedge a_2 \wedge b_3) 
\end{aligned}
\end{equation}

\item For the elements $(c_1\wedge a_2 \wedge b_2) \otimes (c'_1\wedge a_m \wedge b_m)$ with $m=2,3$, as in the previous case, if $c_1=c_1'=a_1,$ acting by $G=Id+e_{14}$,
\begin{align*}
G.(a_1\wedge a_2 \wedge b_2) \otimes G.(a_4\wedge a_m \wedge b_m)
= & (a_1\wedge a_2 \wedge b_2) \otimes ((a_1+a_4)\wedge a_m \wedge b_m) \\
= & (a_1\wedge a_2 \wedge b_2) \otimes (a_1\wedge a_m \wedge b_m) \\
& + (a_1\wedge a_2 \wedge b_2) \otimes (a_4\wedge a_m \wedge b_m).
\end{align*}
Then, in the coinvariants module, $(a_1\wedge a_2 \wedge b_2) \otimes (a_1\wedge a_m \wedge b_m)=0.$ Analogously, if $c_1=c_1'=b_1$, acting by $G=Id+e_{41}$ one gets that $(b_1\wedge a_2 \wedge b_2) \otimes (b_1\wedge a_m \wedge b_m)=0$ in the coinvariants module. Therefore we are only left with:
\begin{equation}
\label{elem_2}
\begin{aligned}
(a_1\wedge a_2 \wedge b_2) \otimes (b_1\wedge a_2 \wedge b_2) &\qquad (b_1\wedge a_2 \wedge b_2) \otimes (a_1\wedge a_2 \wedge b_2) \\
(a_1\wedge a_2 \wedge b_2) \otimes (b_1\wedge a_3 \wedge b_3)& \qquad (b_1\wedge a_2 \wedge b_2) \otimes (a_1\wedge a_3 \wedge b_3)
\end{aligned}
\end{equation}
\end{itemize}

Therefore we only have to treat the elements of \eqref{elem_1} and \eqref{elem_2}.

\begin{itemize}

\item Acting by $G=Id+e_{32}$, in the coinvariants module,
\begin{align*}
0= & (a_1\wedge a_2\wedge b_2) \otimes ( b_1\wedge a_2\wedge b_3)= G.(a_1\wedge a_2\wedge b_2) \otimes G.( b_1\wedge a_2\wedge b_3) \\
= & (a_1\wedge (a_2+a_3)\wedge b_2) \otimes ( b_1\wedge (a_2+a_3)\wedge (b_3-b_2)) \\
=& - (a_1\wedge a_2\wedge b_2) \otimes ( b_1\wedge a_2\wedge b_2)+(a_1\wedge a_2\wedge b_2) \otimes ( b_1\wedge a_3\wedge b_3)\\
&+ (a_1\wedge a_3\wedge b_2) \otimes ( b_1\wedge a_2\wedge b_3)).
\end{align*}
Thus
\begin{equation*}
\begin{aligned}
(a_1\wedge a_2 \wedge b_3) \otimes (b_1\wedge b_2 \wedge a_3) = & (a_1\wedge a_2\wedge b_2) \otimes ( b_1\wedge a_3\wedge b_3) \\
& -(a_1\wedge a_2\wedge b_2) \otimes ( b_1\wedge a_2\wedge b_2). 
\end{aligned}
\end{equation*}
Analogously, taking the same action we get that
\begin{equation*}
\begin{aligned}
(b_1\wedge b_2 \wedge a_3) \otimes (a_1\wedge a_2 \wedge b_3) = &  ( b_1\wedge a_3\wedge b_3)\otimes (a_1\wedge a_2\wedge b_2) \\
& -(a_1\wedge a_2\wedge b_2) \otimes ( b_1\wedge a_2\wedge b_2). 
\end{aligned}
\end{equation*}
\end{itemize}
Therefore we are only left with the generating set of the statement.
\end{proof}

Taking the antisymmetrization of the generators given in Proposition \ref{prop-coinv-tens-extp} we get the following result:
\begin{cor}
\label{cor-coinv-ext-extp}
Given an integer $g\geq 4$ and an odd prime $p$, the $\mathbb{Z}/p$-vector space $(\Lambda^3 H_p\wedge \Lambda^3 H_p)_{GL_g(\mathbb{Z})}$ is generated by the following three elements:
$$(a_1\wedge a_2 \wedge a_3) \wedge (b_1\wedge b_2 \wedge b_3),
\qquad
(a_1\wedge a_2 \wedge b_2) \wedge (b_1\wedge a_2 \wedge b_2),$$
$$(a_1\wedge a_2 \wedge b_2) \wedge (b_1\wedge a_3 \wedge b_3).$$
\end{cor}

For the second piece, $\mathfrak{sp}_{2g}(\mathbb{Z}/p)$,
recall that there is a decomposition as $GL_{g}(\mathbb{Z})$-modules,
$$\mathfrak{sp}_{2g}(\mathbb{Z}/p)=\mathfrak{gl}_g(\mathbb{Z}/p)\oplus Sym^A_g(\mathbb{Z}/p)\oplus Sym^B_g(\mathbb{Z}/p),$$
where the action of $GL_g(\mathbb{Z})$ on $\mathfrak{gl}_g(\mathbb{Z}/d)$ is by conjugation, the action on $Sym^A_g(\mathbb{Z}/p)$ is given by $G.\beta=G\beta\;{}^t G$
and the action on $Sym^B_g(\mathbb{Z}/p)$ by $G.\gamma={}^tG^{-1}\gamma \;G^{-1}$.

Recall that $e_{ij}$ denotes the elementary matrix with $1$ at the place $ij$ and zero elsewhere. Let $i\neq j,$ define the following matrices in $\mathfrak{sp}_{2g}(\mathbb{Z}/p):$
$$u_{ij}=\left(\begin{smallmatrix}
0 & e_{ij}+e_{ji} \\
0 & 0
\end{smallmatrix}\right), \quad u_{ii}=\left(\begin{smallmatrix}
0 & e_{ii} \\
0 & 0
\end{smallmatrix}\right),\quad
l_{ij}=\left(\begin{smallmatrix}
0 & 0 \\
e_{ij}+e_{ji} & 0
\end{smallmatrix}\right),\quad
l_{ii}=\left(\begin{smallmatrix}
0 & 0 \\
e_{ii} & 0
\end{smallmatrix}\right),$$
$$
n_{ij}=\left(\begin{smallmatrix}
e_{ij} & 0 \\
0 & -e_{ji}
\end{smallmatrix}\right), \quad n_{ii}=\left(\begin{smallmatrix}
e_{ii} & 0 \\
0 & -e_{ii}
\end{smallmatrix}\right).
$$
The set formed by these matrices with $1\leq i<j\leq g$ is a basis of the $\mathbb{Z}/p$-vector space 
$\mathfrak{sp}_{2g}(\mathbb{Z}/p).$ And then the tensor product of pairs of these matrices produces a basis of the $\mathbb{Z}/p$-vector space $\mathfrak{sp}_{2g}(\mathbb{Z}/p)\otimes\mathfrak{sp}_{2g}(\mathbb{Z}/p)$.

\begin{prop}
\label{prop-coinv-tens-sp}
Given an integer $g\geq 4$ and an odd prime $p$, the $\mathbb{Z}/p$-vector space $(\mathfrak{sp}_{2g}(\mathbb{Z}/p)\otimes\mathfrak{sp}_{2g}(\mathbb{Z}/p))_{GL_g(\mathbb{Z})}$ is generated by the following four elements:
\begin{equation*}
n_{11}\otimes n_{11},\quad n_{11}\otimes n_{22},\quad
u_{11}\otimes l_{11},\quad l_{11}\otimes u_{11}.
\end{equation*}
\end{prop}

\begin{proof}
Denote by $N,$ $U,$ $L$ the $GL_g(\mathbb{Z})$-submodules of $\mathfrak{sp}_{2g}(\mathbb{Z}/p)$ formed by the respective sets of matrices $\{n_{ii},n_{ij}\}$, $\{u_{ii},u_{ij}\}$, $\{l_{ii},l_{ij}\}$ with $1\leq i<j\leq g$.
Then the $\mathbb{Z}/p$-vector space $\mathfrak{sp}_{2g}(\mathbb{Z}/p)\otimes\mathfrak{sp}_{2g}(\mathbb{Z}/p)$ is generated by the following $GL_g(\mathbb{Z})$-submodules:
$$N\otimes N, \quad N\otimes U, \quad N\otimes L , \quad U\otimes N , \quad U\otimes U $$
$$U\otimes L, \quad L\otimes N, \quad L\otimes U , \quad L\otimes L.$$
Taking the quotient by the action of $GL_g(\mathbb{Z})$, we now reduce the number of generators in the coinvariant module $(\mathfrak{sp}_{2g}(\mathbb{Z}/p)\otimes\mathfrak{sp}_{2g}(\mathbb{Z}/p))_{GL_g(\mathbb{Z})}.$

\begin{itemize}

\item Reduction of generators in $N\otimes N$.
The submodule $N\otimes N$ is generated by the $\mathfrak{S}_g$-orbits of the following elements:
\begin{align*}
& n_{11}\otimes n_{11},\quad n_{11}\otimes n_{22},
\quad n_{11}\otimes n_{12}, \quad n_{11}\otimes n_{21},\quad n_{11}\otimes n_{23},\\
& n_{12}\otimes n_{11},\quad n_{12}\otimes n_{22},
\quad n_{12}\otimes n_{12},\quad n_{12}\otimes n_{21},\quad n_{12}\otimes n_{13},\\
& n_{12}\otimes n_{31}, \quad n_{12}\otimes n_{23},\quad n_{12}\otimes n_{32},\quad n_{12}\otimes n_{34},\quad n_{12}\otimes n_{33}.
\end{align*}

A direct computation shows that the action by $Id-2e_{22}$ (resp. $Id-2e_{33}$) sends the generators containing just one element among $\{n_{12}, n_{21}, n_{23}, n_{32}\}$ (resp. $\{n_{23},n_{32},n_{34}\}$) to their respective opposite, and hence these elements are annihilated in the coinvariants module. Therefore we are left with:
$$n_{11}\otimes n_{11},\quad n_{11}\otimes n_{22},
\quad n_{12}\otimes n_{12}\quad n_{12}\otimes n_{21}.$$
We proceed to reduce even more the number of generators.

Using the action by $G_{ij}=Id+e_{ij}$, in the coinvariants module,
\begin{enumerate}
\item $G_{12}.(n_{11}\otimes n_{22})=
(n_{11}-n_{12})\otimes(n_{22}+n_{12})= n_{11}\otimes n_{22} -n_{12}\otimes n_{12}.$

Then we get that $n_{12}\otimes n_{12}=0.$

\item $\begin{aligned}[t]
G_{21}.(n_{11}\otimes n_{12})= &
(n_{11}+n_{21})\otimes(-n_{11}+n_{22}+n_{12}-n_{21})= \\[-0.1cm]
= & -n_{11}\otimes n_{11}+n_{11}\otimes n_{22} +n_{21}\otimes n_{12}.
\end{aligned}$

Then we get that $n_{21}\otimes n_{12}=n_{11}\otimes n_{11}-n_{11}\otimes n_{22}.$ 
\end{enumerate}
Therefore on $N\otimes N$ we are only left with:
$$n_{11}\otimes n_{11}, \quad n_{11}\otimes n_{22}.$$

\item Reduction of generators in $U\otimes L$ and $L\otimes U$.
The group $U\otimes L$
is generated by the $\mathfrak{S}_g$-orbits of the following elements:
\begin{align*}
& u_{11}\otimes l_{11},\quad u_{11}\otimes l_{22},
\quad u_{11}\otimes l_{12}, \quad u_{11}\otimes l_{23},\quad u_{12}\otimes l_{11}, \\
& u_{12}\otimes l_{33},
\quad u_{12}\otimes l_{12}, \quad u_{12}\otimes l_{23},\quad u_{12}\otimes l_{34}.
\end{align*}

A direct computation shows that the action by $Id-2e_{22}$ (resp. $Id-2e_{33}$) sends the generators containing exactly one element among $\{u_{12}, u_{23}, l_{12}, l_{23} \}$ (resp. $\{l_{23},l_{34}\}$) to their respective opposite, and hence these elements are annihilated in the coinvariants module. Therefore we are only left with:
$$u_{11}\otimes l_{11},\quad u_{11}\otimes l_{22},
\quad u_{12}\otimes l_{12}.$$

Using the action by $G_{ij}=Id+e_{ij},$ in the coinvariants module,

\begin{enumerate}
\item $ G_{12}.(u_{11}\otimes l_{11})=
u_{11}\otimes (l_{11}+l_{22}-l_{12})= u_{11}\otimes l_{11}+u_{11}\otimes l_{22}.$

Then we get that $u_{11}\otimes l_{22}=0.$

\item
$\begin{aligned}[t]
G_{21}.(u_{11}\otimes l_{12})
& =(u_{11}+u_{22}+u_{12})\otimes (l_{12}-2l_{11}) \\
& = u_{12}\otimes l_{12} -2(u_{11}\otimes l_{11})-2(u_{22}\otimes l_{11}).
\end{aligned}$

Then we get that,
$$u_{12}\otimes l_{12}= 2(u_{11}\otimes l_{11}).$$
Therefore on $U\otimes L$ we are only left with the element $u_{11}\otimes l_{11}.$
Analogously on $L\otimes U,$ we are only left with the element $l_{11}\otimes u_{11}.$
\end{enumerate}

\item Nullity of the generators in $N\otimes U,$ $N\otimes L,$ $U\otimes N$ and $L\otimes N.$ We only prove that the generators in  $N\otimes U$ vanish in the coinvariants module. The other cases are similar.

The submodule $N\otimes U$
is generated by the $\mathfrak{S}_g$-orbits of the following elements:
\begin{align*}
&n_{11}\otimes u_{11},\quad n_{11}\otimes u_{22},
\quad n_{11}\otimes u_{12} \quad n_{11}\otimes u_{23}, \\
& n_{12}\otimes u_{11}, \quad n_{12}\otimes u_{22},\quad n_{12}\otimes u_{33},
\quad n_{12}\otimes u_{12}, \\
& n_{12}\otimes u_{13},\quad n_{12}\otimes u_{23},\quad n_{12}\otimes u_{34}.
\end{align*}

A direct computation shows that the action by $Id-2e_{22}$ (resp. $Id-2e_{33}$) sends the generators containing just one element among $\{n_{12},u_{12},u_{23} \}$ (resp. $\{u_{23},u_{34}\}$) to their respective opposite, and hence these elements are annihilated in the coinvariants module. Therefore we are only left with:
$$n_{11}\otimes u_{11},\quad n_{11}\otimes u_{22},
\quad n_{12}\otimes u_{12}.$$
We now show that these elements are also annihilated in the coinvariants module.

Using the action by $G_{ij}=Id+e_{ij},$ in the coinvariants module,
\begin{enumerate}
\item $\begin{aligned}[t] G_{23} .(n_{11}\otimes u_{33}) & =
n_{11}\otimes (u_{22}+u_{33}+u_{23})\\
& =n_{11}\otimes u_{22}+n_{11}\otimes u_{33}.
\end{aligned}
$

Then we get that $n_{11}\otimes u_{22}=0.$

\item $\begin{aligned}[t] G_{12}.(n_{11}\otimes  u_{22}) & =
(n_{11}-n_{12})\otimes ( u_{11}+u_{22}+u_{12})\\
& = n_{11}\otimes  u_{11}-n_{12}\otimes u_{12}.
\end{aligned}$

Then we get that $n_{11}\otimes  u_{11}=n_{12}\otimes u_{12} .$

\item $\begin{aligned}[t] G_{12}.(n_{11}\otimes u_{12}) & = (n_{11}-n_{12})\otimes( 2u_{11}+u_{12}) \\
& = 2(n_{11}\otimes u_{11})- n_{12}\otimes u_{12}.
\end{aligned}$

Then we get that $2(n_{11}\otimes u_{11})=n_{12}\otimes u_{12}$ and by the relation obtained in above equation (2) we deduce that $n_{11}\otimes u_{11}=0$
and $n_{12}\otimes u_{12}=0$.
\end{enumerate}

\item Nullity of the generators in $U\otimes U$ and $L\otimes L.$
We prove that the generators in $U\otimes U$ vanish in the coinvariants module. The case of $L\otimes L$ is analogous.

The submodule $U\otimes U$
is generated by the $\mathfrak{S}_g$-orbits of the following elements:
\begin{align*}
& u_{11}\otimes u_{11},\quad u_{11}\otimes u_{22},
\quad u_{11}\otimes u_{12},\quad u_{11}\otimes u_{23}, \\
& u_{12}\otimes u_{11}, \quad u_{12}\otimes u_{33}, \quad u_{12}\otimes u_{12}, \quad u_{12}\otimes u_{13},\quad u_{12}\otimes u_{34}.
\end{align*}

A direct computation shows that the action by $Id-2e_{22}$ (resp. $Id-2e_{33}$) sends the generators containing just one element among $\{ u_{12}, u_{23} \}$ (resp. $\{u_{23},u_{34}\}$) to their respective opposite, and hence these elements are annihilated in the coinvariants module. Therefore we are only left with:
$$u_{11}\otimes u_{11},\quad u_{11}\otimes u_{22},
\quad u_{12}\otimes u_{12}.$$
We now show that these elements are also annihilated in the coinvariants module.

Using the action by $G_{ij}=Id+e_{ij},$ in the coinvariants module,
\begin{enumerate}
\item $G_{12}.(u_{11}\otimes u_{22})=
u_{11}\otimes (u_{11}+u_{22}+u_{12}) =u_{11}\otimes u_{22}+u_{11}\otimes u_{11}.$

Then we get that $u_{11}\otimes u_{11}=0.$
\item $ G_{23}.(u_{11}\otimes u_{33})=
u_{11}\otimes (u_{22}+u_{33}+u_{23}) =u_{11}\otimes u_{22}+u_{11}\otimes u_{33}.$

Then we get that $u_{11}\otimes u_{22}=0.$
\item $ G_{21}.(u_{11}\otimes u_{11})=
(u_{11}+u_{22}+u_{12})\otimes (u_{11}+u_{22}+u_{12})=u_{12}\otimes u_{12}.$

Then we get that $u_{12}\otimes u_{12}=0.$
\end{enumerate}

\end{itemize}

\end{proof}
Taking the antisymmetrization of the generators given in Proposition \ref{prop-coinv-tens-sp} we get the following result:
\begin{cor}
Given an integer $g\geq 4$ and an odd prime $p$, the $\mathbb{Z}/p$-vector space $(\mathfrak{sp}_{2g}(\mathbb{Z}/p)\wedge\mathfrak{sp}_{2g}(\mathbb{Z}/p))_{GL_g(\mathbb{Z})}$ is generated by the element $u_{11}\wedge l_{11}.$
\end{cor}

We now  compute the $GL_g(\mathbb{Z})$-coinvariants quotient of the algebra $\mathcal{A}_2(H_p)$ of uni-trivalent trees with just two trivalent vertices and labels in $H_p$, introduced after Proposition \ref{prop-ker-im-tau}.

\begin{prop} \label{prop-2sym-gen}
Given an integer $g\geq 3$ and $p$ an odd prime, the group $(\mathcal{A}_2(H_p))_{GL_g(\mathbb{Z})}$ is generated by the following two elements:
$$\tree{b_1}{b_2}{a_2}{a_1},\quad \tree{a_1}{b_1}{b_2}{a_2}.$$
\end{prop}

\begin{proof}
To make the notation lighter let us set
$$H(a,b,c,d):=\tree{a}{b}{c}{d}\in \mathcal{A}_2(H_p).$$
The module $\mathcal{A}_2(H_p)$ is generated by elements of the form
$$H(c_i, c'_j, d_k, d'_l) \qquad \text{with} \qquad
\begin{array}{l}
c,c',d,d'= a \text{ or } b, \\
i,j,k,l\in \{1,\ldots ,g\}
\end{array}
.$$
Taking the quotient by the action of $GL_g(\mathbb{Z})$, we now reduce the number of generators in the coinvariant module $(\mathcal{A}_2(H_p))_{GL_g(\mathbb{Z})}$.

Notice that the elements with $i\neq j,k,l$ are annihilated by the action of $Id-2e_{i,i}$. This argument also holds for the other coefficients. Therefore picking one element in each $\mathfrak{S}_g$-orbit we are left with
$$H(c_1,c'_2,d_2,d'_1),\quad
H(a_1,b_1,b_2,a_2),\quad
H(a_1,b_1,b_1,a_1).$$

Using the action by $G_{ij}=Id+e_{ij},$ we compute:

\begin{enumerate}
\item $ G_{13}\cdot H(a_3,c_2',d_2,a_1) = H(a_1,c_2',d_2,a_1)+H(a_3,c_2',d_2,a_1).$

Then we get that
$H(a_1,c_2',d_2,a_1)=0.$

\item $G_{32}\cdot H(c_1,b_3,b_2,d_1')=H(c_1,b_3,b_2,d_1')-H(c_1,b_2,b_2,d_1').$

Then we get that $H(c_1,b_2,b_2,d_1')=0.$

\end{enumerate}

Therefore the set of generators of the form $H(c_1,c'_2,d_2,d'_1)$ is reduced to the following two elements:
$$H(b_1,b_2,a_2,a_1), \quad H(b_1,a_2,b_2,a_1).$$

\begin{enumerate}
\item[(3)] Using the above relations and AS relation we compute:
$$\begin{aligned}[t]
0=& H(a_1,b_2,b_2,a_1)= G_{21}\cdot H(a_1,b_2,b_2,a_1) \\
=& H(a_1,b_1,b_1,a_1)-H(a_1,b_1,b_2,a_2) -H(a_1,b_2,b_1,a_2) \\
& -H(a_2,b_1,b_2,a_1) -H(a_2,b_2,b_1,a_1)+H(a_2,b_2,b_2,a_2) \\
=& 2H(a_1,b_1,b_1,a_1) -2H(a_1,b_1,b_2,a_2)-2H(a_1,b_2,b_1,a_2).
\end{aligned}$$
Then dividing by $2$ we get that,
\begin{align*}
H(a_1,b_1,b_1,a_1)= & H(a_1,b_1,b_2,a_2)+H(a_1,b_2,b_1,a_2) \\
= & H(a_1,b_1,b_2,a_2)-H(b_1,a_2,b_2,a_1).
\end{align*}

\item[(4)] Using the IHX, AS relations we get the following identity:
\begin{align*}
H(b_1,a_2,b_2,a_1) = & H(b_1,b_2,a_2,a_1)+H(a_2,b_2,a_1,b_1) \\
= & H(b_1,b_2,a_2,a_1)-H(a_1,b_1,b_2,a_2).
\end{align*}
\end{enumerate}
\end{proof}

\newpage

\bibliographystyle{abbrv}


\end{document}